\documentclass[10pt,reqno,final]{amsart}
\usepackage{amssymb}
\usepackage{amsfonts}
\usepackage{stmaryrd}
\usepackage{amsmath,amsfonts,amssymb,amsthm,version,graphicx,fancybox,pifont}
\usepackage[notcite,notref]{showkeys}
\usepackage{url,hyperref}
\usepackage{subfigure,multirow}
\usepackage{xcolor}
\usepackage{exscale}
\usepackage{relsize}
\usepackage{cases}

\allowdisplaybreaks

\catcode`\@=11
\@addtoreset{equation}{section}   
\renewcommand{\theequation}{\arabic{section}.\arabic{equation}}
\@addtoreset{table}{section}   
\renewcommand\thefigure{\thesection.\@arabic\c@figure}
\@addtoreset{figure}{section}   
\renewcommand\thetable{\thesection.\@arabic\c@table}

 \newcommand{\new}{\newcommand*}
 \new{\rnew}{\renewcommand*}
 \new{\newe}{\newenvironment*}
 \new{\stl}{\setlength}
 \stl{\arraycolsep}{0.5mm}

\newtheorem{thm}{\bf Theorem}
\newtheorem{col}{Corollary}[section]

\newenvironment{theorem}{\begin{thm}} {\end{thm}}
\newtheorem{lmm}{\bf Lemma}

\newenvironment{lemma}{\begin{lmm}}{\end{lmm}}
\theoremstyle{remark}
\newtheorem{rem}{Remark}[section]
\theoremstyle{definition}
\newtheorem{defn}{Definition}[section]

\renewcommand{\theequation}{\arabic{section}.\arabic{equation}}

\newcommand {\bgeq}[1]{\begin{equation}\label{#1}}
\newcommand \edeq {\end{equation}}
\newcommand \bgth {\begin{theorem}\label}
\newcommand \edth {\end{theorem}}
\newcommand \bglm {\begin{lemma}\label}
\newcommand \edlm {\end{lemma}}
\newcommand {\bgar}[1]{\begin{array}{#1}}
\newcommand {\edar}{\end{array}}

\def \pd{\partial}
\def \ds{\displaystyle}

\def \cal {\mathcal}

\title[$hp$-version $C^1$-CPG method for second-order IVPs]
{$hp$-version  $C^1$-continuous Petrov-Galerkin method for
nonlinear second-order initial value problems with application  to
wave equations}

\author[L. Wang, \quad M. Zhang   \quad H. Tian, \quad \mbox{and} \quad L. Yi]
{Lina Wang$^{1,2}$,  \quad Mingzhu Zhang$^{1}$, \quad  Hongjiong
Tian$^{1,3}$,    \quad and \quad Lijun Yi$^{1,3,\dag}$}

\subjclass[2000]{65L60,  65L05,  65L70}

\keywords{$hp$-version; $C^1$-continuous Petrov-Galerkin method;
second-order initial value problem; wave equation;
superconvergence.}

\thanks{$^{\dag}$Corresponding author,  E-mail address: ylj5152@shnu.edu.cn (L. Yi).}

\thanks{$^*$The work of H. Tian is supported   by the National Natural Science
Foundation of China (Grant No. 12271368). The work of L. Yi is
supported   by the National Natural Science Foundation of China
(Grant Nos. 12171322, 11771298 and 12271366), the Natural Science
Foundation of Shanghai (Grant Nos. 21ZR1447200  and 22ZR1445500),
and the Science and Technology Innovation Plan of Shanghai (Grant
No. 20JC1414200).}

\begin{document}
\maketitle
{\footnotesize

 \centerline{$^1$Department of  Mathematics, Shanghai Normal University, Shanghai
200234, China}

\vskip 1mm
 \centerline{$^2$School of Science, Henan University of Engineering, Zhengzhou  451191, China}

\vskip 1mm

 \centerline{$^3$Scientific Computing Key Laboratory of
Shanghai Universities, Shanghai 200234, China}
 }

\begin{abstract}
We introduce and analyze  an $hp$-version  $C^1$-continuous
Petrov-Galerkin (CPG)  method for  nonlinear  initial value problems
of second-order ordinary differential equations.
 We derive  a-priori error estimates in the $L^2$-, $L^\infty$-,
 $H^1$- and $H^2$-norms   that are completely
 explicit in the local  time steps and local approximation
degrees.   Moreover, we show that the $hp$-version    $C^1$-CPG
  method superconverges at the nodal points of the time partition
  with regard to the time steps and approximation degrees.
As an application, we apply  the $hp$-version   $C^1$-CPG
  method to  time discretization of  nonlinear wave equations.
Several numerical  examples  are presented to verify the theoretical
results.
\end{abstract}

\section{ Introduction}

The initial value problems (IVPs)  of second-order ordinary
differential equations (ODEs) have been widely  used  in many
fields. Moreover, a large number of second-order time dependent
problems, especially nonlinear wave equations,  such as the
sine-Gorden and  Klein-Gorden
  equations, are often transformed into IVPs of second-order ODEs
after appropriate spatial discretization methods. During the past
few decades,  great progress has been made in the study of numerical
methods for solving  the IVPs of  ODEs. The most popular and
frequently used approaches  for the numerical integration of
second-order ODEs  are mainly based on  implicit and explicit finite
difference, Runge-Kutta, collocation, and Newmark-type schemes. For
a general overview we  refer the reader to  the monographs
\cite{B08,HLW,HNW,HW,L}  and the references therein.

Galerkin-type methods for solving   IVPs of ODEs can be traced back
to the 1970s. We mention here the papers by  \cite{Hu1,Hu2}, where
the continuous Galerkin (CG) schemes have been introduced and
analyzed for  first-order IVPs \cite{EF}. The
discontinuous Galerkin (DG) schemes  have  also been studied for the
numerical integration of IVPs \cite{DD,DHT}. It is worth noting
that  the   error analyses of the above works were mainly based on
the   $h$-version approach, namely, the convergence is achieved by
decreasing the  time steps at a fixed and typically low-degree
approximation. This is in contrast to the concepts of $p$- and
$hp$-version approaches (originated from  the finite element
methods), where the $p$-version approach uses a fixed partition but
increases the polynomial degree  to increase accuracy and  the
$hp$-version approach combines    the $h$- and $p$-refinement
techniques. For an overview of the $p$- and $hp$-version methods, we
refer the reader   to  the monographs \cite{Schwab,SB} and an
excellent survey paper \cite{Ba-1}.

Due to  great flexibility with respect to the  local time steps and
local approximation degrees,  the  $hp$-version Galerkin methods for
numerical solutions of  IVPs have been widely studied in recent
years; see, e.g.,  the pioneer work by \cite{SS}, where an
$hp$-version DG time stepping method has been  introduced and
analyzed for  first-order IVPs;  see also \cite{Wihler,Wihler1,Y2}, where  $hp$-version CG methods  for first-order IVPs
have been studied. Relevant applications and analysis of the
$hp$-version    DG and CG methods for parabolic problems,
integro-differential equations, and fractional differential
equations can be found in \cite{BS,MS,SS1,ST,WGSS,YG3}.

Very recently, an $hp$-version   DG method was introduced  for
  linear second-order IVPs in \cite{AMDNQ}, where   suboptimal
error estimate  (with respect to the  polynomial degree) was
obtained in a suitable mesh-dependent norm.
 Moreover,  an $hp$-version   $C^0$-continuous Petrov-Galerkin ($C^0$-CPG)
method  based on the CG and DG  methodologies  was developed in
\cite{WY} for  nonlinear second-order IVPs, where the globally
$C^0$-continuous piecewise polynomials were used for the trial
spaces  and  an optimal $H^1$-error estimate was proved. However, it
seems more natural to use $C^1$-continuous approximations for
second-order IVPs due to the following considerations:

\begin{itemize}

 \item

The  DG \cite{AMDNQ} and  $C^0$-CPG \cite{WY} formulations
 incorporate the initial values in a weak sense and
thus leads to the appearance of jump terms (on functions or/and
derivatives) in the numerical scheme,  which brings additional
consideration in the
  analysis and computation. In contrast,
  the  $C^1$-CPG  method presented in this article shall produce globally
$C^1$-continuous  approximations and  there are no jump terms in the
numerical scheme, which greatly  simplifies the analysis.

\item For a given polynomial degree, the $C^1$-CPG method has fewer   degrees of
freedom (DOF) than the  DG \cite{AMDNQ} and $C^0$-CPG \cite{WY}
methods  on each time interval. From another point of view, if we
employ the same number of DOF on each time interval, the $C^1$-CPG
method  exhibits higher convergence rates (with respect to
 the time steps) than the DG and $C^0$-CPG methods at least for
smooth solutions.

\end{itemize}

In this paper  we  consider the numerical integration of  a
nonlinear second-order IVP of the form
\begin{eqnarray}\label{IVP}
\left\{\begin{array}{ll} u''(t)=f(t,u(t),u'(t)),\quad & t\in  [0,T],\\[2mm]
u(0)=u_0,\quad u'(0)=u_1,
\end{array}\right.
\end{eqnarray}
although the results carry over to systems of such equations. Let
$I:=(0,T)$ for some $T>0$. Here, $u: \bar{I}\rightarrow \mathbb{R}$
denotes the unknown solution, $f: \bar{I} \times \mathbb{R} \times
\mathbb{R} \rightarrow \mathbb{R}$ is a given function,  and the
initial values $u_0, u_1 \in \mathbb{R}$. We assume that the
function $f$ in (\ref{IVP}) is continuous for $t\in \bar{I}$ and
satisfies the following  uniformly Lipschitz condition
\begin{equation}\label{Lip-con}
|f(t,u_1,v_1)-f(t,u_2, v_2)|\le
L\left(|u_1-u_2|+|v_1-v_2|\right),\quad \forall u_i, v_i \in
\mathbb{R},~~ i=1,2
\end{equation}
for $t\in \bar{I}$,  where $L>0$ is the Lipschitz constant that is
independent of $t\in \bar{I}$.

The main purpose of  this paper is to propose and analyze an
$hp$-version $C^1$-CPG method   for the numerical integration of the
problem (\ref{IVP}). Here, the $C^1$-CPG method  uses
$C^1$-continuous piecewise polynomials for the trial spaces and
discontinuous piecewise polynomials for the test spaces. Due to
discontinuous character of the test spaces, the $C^1$-CPG scheme can
be decoupled into local problems on each time step, and thus it can
be regarded as a time stepping scheme. We show that the $hp$-version
$C^1$-CPG scheme is well defined provided that the time steps are
sufficient small. Based on a piecewise projector (on to the trial
spaces) $\Pi^{{\bf{r}}}u$ (see (\ref{pir})), we present a rigorous
 error analysis of the $hp$-version $C^1$-CPG method.
  We derive a-priori error  estimates in the $L^2$-, $L^\infty$-,
 $H^1$-  and $H^2$-norms that are completely explicit  in the local time steps and local approximation
degrees.   These error estimates  imply that the  $hp$-version
$C^1$-CPG method can achieve arbitrary high order convergence rates
(i.e., spectral accuracy) provided that the solution is  smooth
enough. Moreover, we prove that the $hp$-version $C^1$-CPG method
superconverges at the nodal points of the time partition   with
respect to the time steps and approximation degrees.


%
%
%
%
%
%

As an application, we apply the $hp$-version  $C^1$-CPG method
 developed for the  second-order IVP (\ref{IVP}) to   nonlinear wave equations. Specifically, we use the $hp$-version
$C^1$-CPG time stepping method to handle the time integration of the
second-order nonlinear  differential system arising after spatial
discretization obtained by the standard spectral Galerkin or
conforming finite element Galerkin method.

The apaper is organized as follows. In Section \ref{sec2}, we
introduce the $hp$-version   $C^1$-CPG   method
 for  the problem (\ref{IVP}) and prove the  existence and uniqueness of the discrete solutions.
In Section \ref{sec3}, we carry out a rigorous error analysis of the
$hp$-version  $C^1$-CPG  scheme.  In Section \ref{sec4}, we apply
the $hp$-version  $C^1$-CPG   method to  time discretization of the
nonlinear second-order  wave equations. In Section \ref{sec5}, we
present some numerical examples  to illustrate the  theoretical
results. Finally, we give some concluding remarks in Section
\ref{sec6}.

Throughout the paper,  we shall follow the usual notations  and
conventions for the
 Sobolev spaces  and their norms \cite{ada}.
  For an open interval $J$, we denote by $L^2(J)$  the   Lebesgue space of square integrable functions on
$J$ with values in $\mathbb{R}$  and by $L^\infty(J)$ the space of
all bounded functions on $J$. For any  non-negative integer $k$,  we
equip the Sobolev space $W^{k,p}(J)$ with the standard norm
$\|\cdot\|_{W^{k,p}(J)}$ and seminorm  $|\cdot|_{W^{k,p}(J)}$. The
fractional-order space $W^{s,p}(J), ~s \ge 0$, is defined by the
usual $K$-method of interpolation. In particular, we   set
$H^s(J)=W^{s,2}(J)$. Moreover, we denote by $C$  a  generic positive
constant independent of the discretization parameters of interest
(such as time steps and approximation degrees) but may take
different values in different places.

\section{$hp$-version of the $C^1$-continuous Petrov-Galerkin
method}\label{sec2}

In this section, we shall introduce  the $hp$-version $C^1$-CPG
method for  problem (\ref{IVP}) and   discuss the well-posedness and
algebraic form of the proposed scheme.

\subsection{Galerkin time discretization}

We first introduce  an arbitrary partition $\cal T_h=\{I_n\}_{n=1}^N
$ of the time interval $I=(0,T)$ into $N$  subintervals
$I_n:=(t_{n-1},t_n),  ~1\le n\le N$, with the nodal points given by
$$0=t_0 < t_1 <t_2 < \cdots < t_{N-1} < t_N=T.$$
We define the local time steps $k_n:=t_n-t_{n-1}, ~1\le n\le N$ and
denote by $k=\max_{1\le n\le N} \{ k_n\}$ the length of the largest
subinterval. Moreover, we assign to each subinterval $I_n$ an
approximation degree $r_n \ge 2$  and store these  polynomial
degrees in the vector ${\bf r}=\{r_n \}_{n=1}^N$. Then,   the
$hp$-version trial and test spaces used for the Galerkin
discretization  of (\ref{IVP}),  are given by
\begin{eqnarray}\label{trial}
S^{{\bf r}, 2 }(\cal T_h)=\{ u\in H^2(I):  u|_{I_n} \in
P_{r_n}(I_n),  1\le n\le N \}
\end{eqnarray}
 and
\begin{eqnarray}\label{test}
S^{{\bf r}-2, 0 }(\cal T_h)=\{ u\in L^2(I):  u|_{I_n} \in
P_{r_n-2}(I_n),  1\le n\le N \},
\end{eqnarray} respectively. Here,  we denote by
$P_{r_n}(I_n)$   the space of polynomials of degree at most $r_n$ on
$I_n$, and the space $P_{r_n-2}(I_n)$ is defined analogously.

The  $hp$-version of the $C^1$-CPG method for  (\ref{IVP}) is: find
$U \in S^{{\bf r}, 2 }(\cal T_h)$ such that
\begin{eqnarray}\label{C1CPG-FEM}
 \left\{\begin{array}{ll}
\ds\sum\limits_{n=1}^N \ds\int_{I_n} U''(t)\varphi(t)dt =
\ds\sum\limits_{n=1}^N \ds\int_{I_n} f(t,U(t),U'(t))\varphi(t)dt  ,\\[4mm]
U(0) = u_0,\quad U'(0) = u_1
\end{array}\right.
\end{eqnarray}
for all $ \varphi \in S^{{\bf r}-2, 0 }(\cal T_h)$.

\begin{rem}\label{rem1}
Due to the discontinuous character of the test space $S^{{\bf r}-2,
0 }(\cal T_h)$, the  $hp$-version  $C^1$-CPG method in
(\ref{C1CPG-FEM}) can be regarded as a time stepping scheme:  if the
$C^1$-CPG solution $U$ is given on the time intervals $I_m, 1\le m
\le n-1,$ we can find $U|_{I_n} \in P_{r_n}(I_n)$  by solving the
following   problem
\begin{eqnarray}\label{C1CPG-FEM-1}
\left\{\begin{array}{ll}
 \ds\int_{I_n} U''(t)\varphi(t)dt =
 \ds\int_{I_n} f(t,U(t),U'(t))\varphi(t)dt,\\[4mm]
U|_{I_n} (t_{n-1}) = U|_{I_{n-1}} (t_{n-1}),\quad
U'|_{I_n} (t_{n-1}) = U'|_{I_{n-1}} (t_{n-1})
\end{array}\right.
\end{eqnarray}
for all $\varphi \in  P_{r_n-2}(I_n)$. Here,  $U|_{I_1}(0)=u_0$ and
$U'|_{I_1}(0)=u_1$.
\end{rem}

 We next show the existence and  uniqueness of the discrete
 solutions. The proof of this lemma will be given
in Appendix \ref{app-A-1}.

\begin{lemma}\label{exi-uni}
  Assume that the partition $\cal T_h$  of $(0,T)$ satisfies
 \begin{equation}\label{mesh-cond}
\ds\frac{Lk_n}{2}\sqrt{8+k_n^2} <1, \quad 1\le n \le N.
\end{equation}
Then, the discrete problem (\ref{C1CPG-FEM}) admits a unique
solution $U \in S^{{\bf r}, 2}(\cal{T}_h)$.
  \end{lemma}

\begin{rem}
Lemma \ref{exi-uni} shows that the existence and uniqueness of
  the  $hp$-version $C^1$-CPG  solution is completely
  independent of the approximation degrees $r_n, 1\le n \le N$.
\end{rem}

\subsection{Algebraic formulation}

Clearly, the $C^1$-CPG formulation
(\ref{C1CPG-FEM-1}) can be understood as an implicit single-step
scheme. We now derive  the algebraic formulation corresponding to
the local variational problem (\ref{C1CPG-FEM-1}) on $I_n,~ 1\le n\le
N$.

 Suppose that $\{\phi_{n, l}(t)\}_{l=1}^{r_n+1}$ is
a set of basis of the polynomial space $P_{r_n}(I_n)$. Let
$U_n(t):=U|_{I_n}$ be the $C^1$-CPG  approximation of $u$ on $I_n$.
Then, the local approximation
 $U_n \in P_{r_n}(I_n)$  can be written as
\begin{equation}\label{solu}
U_n(t)=\ds\sum_{l=1}^{r_n+1} \hat{u}_{n,l}\phi_{n, l}(t).
\end{equation}
Inserting  (\ref{solu}) into (\ref{C1CPG-FEM-1}) and selecting the
test function $\varphi=\phi_{n, l}(t)$ with $1\le l \le r_n-1$ leads
to the following nonlinear system
\begin{equation}\label{nonlinear-sys}
A_n \widehat{U}_n=F_n(\widehat{U}_n),
\end{equation}
 where the solution vector
$\widehat{U}_n:=\left(\hat{u}_{n,1}, \hat{u}_{n,2}, \cdots,
\hat{u}_{n, r_n+1}\right)^T$, the matrix $A_n=
\left(a^n_{i,j}\right)_{1\le i, j\le r_n+1}$ is given by
 \begin{equation*}\label{matrx}
A_n=
\begin{pmatrix}
\ds\int_{I_n}\phi_{n, 1}''\phi_{n, 1}dt  \quad & \ds\int_{I_n}\phi_{n, 2}''\phi_{n, 1}dt \quad& \cdots  \quad&  \ds\int_{I_n}\phi_{n, r_n+1}''\phi_{n, 1}dt\\
 \vdots  \quad &  \vdots \quad& \vdots  \quad&   \vdots\\
\ds\int_{I_n}\phi_{n, 1}''\phi_{n, r_n-1}dt  \quad &
\ds\int_{I_n}\phi_{n, 2}''\phi_{n, r_n-1}dt \quad& \cdots  \quad&
\ds\int_{I_n}\phi_{n, r_n+1}''\phi_{n, r_n-1}dt\\
\phi_{n, 1}(t_{n-1})  \quad &  \phi_{n, 2}(t_{n-1})  \quad& \cdots  \quad&   \phi_{n, r_n+1}(t_{n-1}) \\
\phi_{n, 1}'(t_{n-1})  \quad &  \phi_{n, 2}'(t_{n-1})  \quad& \cdots
\quad&   \phi_{n, r_n+1}'(t_{n-1})
\end{pmatrix}
 \end{equation*}
and the right-hand side vector $F_n(\widehat{U}_n) \in
\mathbb{R}^{(r_n+1)}$ is given by
 \begin{equation*}
F_n(\widehat{U}_n)=
\begin{pmatrix}
\ds\int_{I_n} f\Big(t, \sum_{l=1}^{r_n+1} \hat{u}_{n,l}\phi_{n, l}(t), \sum_{l=1}^{r_n+1} \hat{u}_{n,l}\phi_{n, l}'(t)\Big)\phi_{n, 1}(t)dt \\
\vdots \\
\ds\int_{I_n} f\Big(t, \sum_{l=1}^{r_n+1} \hat{u}_{n,l}\phi_{n, l}(t), \sum_{l=1}^{r_n+1} \hat{u}_{n,l}\phi_{n, l}'(t)\Big)\phi_{n, r_n-1}(t)dt \\
U_{n-1}(t_{n-1})\\
U_{n-1}'(t_{n-1})
\end{pmatrix}.
 \end{equation*}
 Here,  $U_{n-1}(t_{n-1})$ and $U_{n-1}'(t_{n-1})$ are given   values from the previous time
 step $I_{n-1}$.

In practice,  the entries   of the matrix $A_n$ can be precomputed
exactly for a given polynomial degree $r_n$ without numerical
quadrature. For example, if we use the following shifted Legendre
polynomials as a set of basis for the polynomial space
$P_{r_n}(I_n)$, i.e.,
$$\phi_{n, l}(t):=L_{l-1}\left( \frac{2t-t_n-t_{n-1}}{k_n} \right),
\quad t\in I_n, ~~1\le l \le r_n+1,$$ where $L_l$ denotes the
standard Legendre polynomial of degree $l$. Noting the fact that
\cite{STW}
$$ L_l''(x)=\ds\sum_{\substack{m=0\\[0.5mm]m+l ~even}}^{l-2} (m+\frac 12)\big( l(l+1)-m(m+1)
\big)L_m(x), \quad l\ge 2,$$ and then using  orthogonality property
of the Legendre polynomials, we have
\begin{equation}\label{matx-a}
a_{i,j}^n=\ds\int_{I_n}\phi_{n, j}''\phi_{n, i}dt=\frac{2}{k_n}
\ds\int_{-1}^1 L_{j-1}''L_{i-1}dx = \left\{
\begin{aligned}
&\frac{2}{k_n}(i+j-1)(j-i), &&~~ j-2\ge i~~\mbox{and}~~i+j~~\mbox{is even},\\
&0,   &&~~  \mbox{otherwise}
\end{aligned}
\right.
\end{equation}
for $1\le i\le r_n-1$ and $1\le j \le r_n+1$.
 Moreover, using the
facts $L_l(-1)=(-1)^l,~ l\ge 0$ and $L_l'(-1)=\frac
12(-1)^{l-1}l(l+1),~l\ge 1$, we further get the entries of the last
two rows of $A_n$, i.e.,
$$a_{r_n,j}^n= \phi_{n, j}(t_{n-1})=(-1)^{j-1}, \quad 1\le j\le r_n+1,$$
$$a_{r_n+1,1}^n= \phi_{n, 1}'(t_{n-1})=0, \quad a_{r_n+1,j}^n= \phi_{n, j}'(t_{n-1})=\frac{(-1)^j(j-1)j}{k_n}, \quad 2\le j\le r_n+1.$$
Due to (\ref{matx-a}),    $A_n$ is a sparse matrix.

 In order to obtain a global
$C^1$-CPG approximation, we only need to solve the local algebraic
system (\ref{nonlinear-sys})  step by step on $I_n, ~1\le n\le N$,
which avoids solving a large system of nonlinear equations.
Moreover, (\ref{nonlinear-sys}) is a small nonlinear algebraic
system with only $r_n+1$ unknowns, and hence we can use the usual
iterative methods such as the Newton's method and the simple fixed
point iteration method to solve it very accurately.

\begin{rem}
We underline that,  the  matrix $A_n$ appears in the nonlinear
system (\ref{nonlinear-sys}) can be precomputed by an analytical
way. In particular, if we employ uniform time partitions and uniform
polynomial degrees, then the matrix $A_n$ at each time step $I_n,
~1\le n\le N$, is the same one. This implies that we can input $A_n$
once and for all before the time loop, which saves a lot of
computational  time and storage. As a result, the overall
computational  cost at each time step $I_n$ is dominated by the cost
of calculating the right-hand side vector $F_n(\widehat{U}_n)$ and
the iterative process, which  depends on the structure of the
nonlinear function $f(t,u,u')$.
\end{rem}

\section{Error analysis}\label{sec3}

In this section, we  will derive several a-priori error  estimates
for the $hp$-version  $C^1$-CPG method  which are explicit in the
local time steps $k_n$ and  local   approximation orders $r_n$. In
particular, we prove that the $hp$-version  $C^1$-CPG method
exhibits superconvergence at the nodal points of the time partition.

\subsection{Generalized Jacobi polynomials}

Let $\Lambda:=(-1,1)$. We denote by $J^{\alpha,\beta}_n(x), ~\alpha,
\beta>-1$, the  classical Jacobi polynomials, which  are orthogonal
 with respect to the Jacobi weight function $\omega^{\alpha,\beta}(x):=(1-x)^{\alpha}(1+x)^{\beta}$ over
 $(-1,1)$, namely,
\begin{equation}\label{classiJ-ort}
\ds\int_{-1}^{1}J^{\alpha, \beta}_n(x) J^{\alpha,
\beta}_m(x)\omega^{\alpha, \beta}(x)dx=\gamma^{\alpha,
\beta}_n\delta_{m,n},  \quad n,m \ge 0,
\end{equation}
where $\delta_{m,n}$ is the Kronecker symbol and
\begin{equation}\label{gammamn}
\gamma^{\alpha,
\beta}_n=\ds\frac{2^{\alpha+\beta+1}\Gamma(n+\alpha+1)\Gamma(n+\beta+1)
}{(2n+\alpha+\beta+1)\Gamma(n+1)\Gamma(n+\alpha+\beta+1)}.
\end{equation}
It is well-known that $J^{0,0}_n(x):=L_n(x)$ is the Legendre
polynomial of degree $n$, and there hold the orthogonalities (see, e.g., \cite{STW})
\begin{equation}\label{Leg-ort}
\ds\int_{-1}^{1}L_n(x)L_m(x)dx=\ds\frac{2}{2n+1}\delta_{m,n}
\end{equation}
and
\begin{equation}\label{Leg-ort-1}
\ds\int_{-1}^{1}L_n^{(k)}(x)L_m^{(k)}(x)(1-x^2)^k
dx=\ds\frac{2}{2n+1}\ds\frac{(n+k)!}{(n-k)!}\delta_{m,n}.
\end{equation}

For our purpose, we also  introduce the  generalized Jacobi
polynomials $J^{k, l}_n(x)$ with the parameters $k, l \le -1$ being
any negative integers (see \cite{STW})
$$J^{k,l}_n(x):=(1-x)^{-k}(1+x)^{-l} J^{-k,-l}_{n+k+l}(x), \quad   n\ge -(k+l).$$
In particular, we are interested in the   generalized Jacobi
polynomials with negative indexes $(-2,-2)$ and $(-1,-1)$, namely
 $$J^{-2,-2}_n(x)=(1-x)^{2}(1+x)^{2} J^{2,2}_{n-4}(x), \quad n \ge 4$$
and
$$J^{-1,-1}_n(x)=(1-x)(1+x) J^{1,1}_{n-2}(x), \quad n\ge 2.$$

We list some important   properties of  the generalized Jacobi
polynomials in the following lemma, and their proofs are
straightforward (see \cite{STW}).

\begin{lemma}\label{pro-plo}
For any negative integers $k, l \le -1$, there holds the
orthogonality
\begin{equation}\label{J-ort}
\ds\int_{-1}^{1}J^{k,l}_n(x)
J^{k,l}_m(x)\omega^{k,l}(x)dx=\gamma^{-k,-l}_{n+k+l}\delta_{m,n},
\end{equation}
where the constant $\gamma^{-k,-l}_{n+k+l}$ is given in
(\ref{gammamn})  and $\omega^{k,l}(x)=(1-x)^k(1+x)^l$.
 Moreover, there hold
\begin{equation}\label{J-1}
J^{-2,-2}_n(x)=4(n-2)(n-3)\ds\int_{-1}^{x}\Big(\ds\int_{-1}^{\eta}L_{n-2}(s)ds\Big)d\eta,
\end{equation}
\begin{equation}\label{J-2}
J_n^{-1,-1}(\pm 1)=0,\quad   J_n^{-2,-2}(\pm 1)=\partial_x
J_n^{-2,-2}(\pm 1)=0,
\end{equation}
and
\begin{equation}\label{J-4}
\partial_x^2 J_n^{-2,-2}(x)=4(n-2)(n-3)L_{n-2}(x), \quad \partial_x
J_n^{-2,-2}(x)=-2(n-3)J_{n-1}^{-1,-1}(x).
\end{equation}
\end{lemma}

\subsection{An auxiliary  projection and its approximation properties}

For any $u\in H^2(\Lambda)$, we  expand $u''$ into the Legendre
series
\begin{equation}\label{u-exp}
 u''(x)=\ds\sum\limits_{i=0}^{\infty}a_i L_i(x)
\end{equation}
with $a_i=\frac{2i+1}{2}\int_{-1}^1 u''L_idx$. Integrating
(\ref{u-exp}) twice  over $[-1,x]$, then by (\ref{J-1})  we obtain
\begin{eqnarray} \label{u-exp-1}
 u(x)=H_3u(x) +\ds\sum\limits_{i=2}^{\infty} a_i \ds\int_{-1}^{x}\Big(\ds\int_{-1}^{\eta}L_i(s)ds\Big)d\eta = H_3u(x)+\ds\sum\limits_{i=4}^{\infty} b_i
 J^{-2,-2}_i(x),
\end{eqnarray}
where $b_i=\frac{1}{4(i-2)(i-3)} a_{i-2}$ and  $H_3u(x)$ is the
cubic Hermit interpolation of $u$, i.e.,
$$H_3u(x)=\ds\frac{x^3-3x+2}{4}u(-1)+\ds\frac{-x^3+3x+2}{4}u(1)
 +\ds\frac{x^3-x^2-x+1}{4}u'(-1)
 +\ds\frac{x^3+x^2-x-1}{4}u'(1).$$
Obviously, there hold  $H_3 u(\pm 1)=u(\pm 1)$ and $(H_3 u)'(\pm
1)=u'(\pm 1)$.

We now introduce a projector  $\Pi_{\Lambda}^{r}$    which is
essential for our error analysis.

\begin{defn}\label{proj}
  For a function $u\in H^2(\Lambda)$, we define
the projector  $\Pi_{\Lambda}^{r}: H^2(\Lambda) \rightarrow
P_{r}(\Lambda), ~r\ge 2$ by
\begin{eqnarray}\label{def-proj}
\left\{\begin{array}{ll} \ds\int_{\Lambda} (u-\Pi_{\Lambda}^{r}u)''
\varphi dx = 0,\quad
\forall \varphi \in P_{r-2}(\Lambda),\\[3mm]
\Pi_{\Lambda}^{r}u (-1) = u(-1)\label{def-proj-2},\\[3mm]
(\Pi_{\Lambda}^{r}u)' (-1) = u'(-1)\label{def-proj-3}.
\end{array}\right.
\end{eqnarray}
\end{defn}

The following lemma shows that  the operator $\Pi_{\Lambda}^{r}$ is
well defined. In particular, we derive its explicit expression based
on the generalized Jacobi polynomials $J^{-2,-2}_n(x)$ (see
(\ref{I-5})). The proof of this lemma will be given in Appendix
\ref{app-A-2}.

\begin{lemma}\label{wellpose}
  The operator $\Pi_{\Lambda}^{r}$ in Definition \ref{proj} is
  well-defined. In particular,  if $r \ge 3$, there hold $\Pi_{\Lambda}^{r}u (\pm 1)=u(\pm 1)$  and $(\Pi_{\Lambda}^{r}u)'
(\pm 1)=u'(\pm 1)$.
\end{lemma}

\begin{rem}
From (\ref{I-5}) we find that the projection $\Pi_{\Lambda}^{r} u$
with $r\ge 3$ is a truncation of the expansion of $u$ (see
(\ref{u-exp-1})).  It is worth noting that such projection   has
been studied in \cite{CJZ-21}, where the $h$-version approximation
and superconvergence properties of the projection were analyzed.
\end{rem}

We next state the   $hp$-version approximation properties of the
projector $\Pi_{\Lambda}^{r}$ in the following lemma and give its
proof in Appendix \ref{app-A-3}.

\begin{lemma}\label{Pi-pro}
Let $u\in H^{s_0+1}(\Lambda)$ for some integer $s_0 \ge 1$ and $r\ge
3$, Then we have
\begin{equation}\label{L2-pro}
\|u-\Pi_{\Lambda}^{r}u\|^2_{L^2(\Lambda)}\le
\ds\frac{(r-s)!}{(r+s-2)!}\ds\frac{1}{
(r-2)^4}\|u^{(s+1)}\|^2_{L^2(\Lambda)},
\end{equation}
\begin{equation}\label{H1-pro}
\|(u-\Pi_{\Lambda}^{r}u)'\|^2_{L^2(\Lambda)}\le
\ds\frac{(r-s)!}{(r+s-2)!}\ds\frac{1}{r(r-1)}\|u^{(s+1)}\|^2_{L^2(\Lambda)},
\end{equation}
\begin{equation}\label{H2-pro}
\|(u-\Pi_{\Lambda}^{r}u)''\|^2_{L^2(\Lambda)}\le  \ds\frac{(r-s)!}{(r+s-2)!}\|u^{(s+1)}\|^2_{L^2(\Lambda)}
\end{equation}
for any integer $s$, $1\leq s \leq \min\{r,s_0\}$.
\end{lemma}

\begin{rem}
Using Stirling's formula, it is easy to verify that, for fixed $s$,
there hold
$$\|u-\Pi_{\Lambda}^{r}u\|_{L^2(\Lambda)} \le C r^{-(s+1)} \|u^{(s+1)}\|_{L^2(\Lambda)},$$
$$\|(u-\Pi_{\Lambda}^{r}u)'\|_{L^2(\Lambda)} \le C r^{-s}\|u^{(s+1)}\|_{L^2(\Lambda)},$$
$$\|(u-\Pi_{\Lambda}^{r}u)''\|_{L^2(\Lambda)} \le C r^{-(s-1)}\|u^{(s+1)}\|_{L^2(\Lambda)}$$
 as $r\rightarrow \infty$, which imply that the estimates in Theorem \ref{Pi-pro} are optimal in terms
of the polynomial degree $r$. Here, the constants $C$ are
independent of $r$.
\end{rem}

On an arbitrary interval $J:=(a,b)$ with length $h=b-a$, we  define
the projector $\Pi_J^{r}$ via the  linear map $\mathcal{M}$ as
\begin{equation}\label{scal-proj}
\Pi_{J}^{r}u =  [ \Pi_{\Lambda}^{r} (u \circ \mathcal{M}) ] \circ
\mathcal{M}^{-1},
\end{equation}
where $\mathcal{M}: \Lambda \rightarrow J$ is
the linear transformation $x \mapsto t=\frac{a+b+hx}{2}$.

 By  scaling
to an arbitrary interval $J$ and interpolating between Sobolev
spaces of integer-order, we obtain  from Lemma \ref{Pi-pro} the
following approximation results immediately.

\begin{col}\label{scal}
Let $J:=(a,b), ~h=b-a$, $r\ge 3$ and  $u\in H^{s_{0}+1}(J)$ with
$s_0\geq 1$. Then we have
\begin{equation}\label{scal-1}
\|u-\Pi_{J}^{r}u\|^2_{L^2(J)}\le C
\left(\ds\frac{h}{2}\right)^{2s+2}
\ds\frac{\Gamma(r-s+1)}{\Gamma(r+s-1)}\ds\frac{1}{(r-2)^4}\|u\|^2_{H^{s+1}(J)},
\end{equation}
\begin{equation}\label{scal-2}
\|(u-\Pi_{J}^{r}u)'\|^2_{L^2(J)}\le C
\left(\ds\frac{h}{2}\right)^{2s}
\ds\frac{\Gamma(r-s+1)}{\Gamma(r+s-1)}\ds\frac{1}{r(r-1)}\|u\|^2_{H^{s+1}(J)},
\end{equation}
\begin{equation}\label{scal-3}
\|(u-\Pi_{J}^{r}u)''\|^2_{L^2(J)}\le C
\left(\ds\frac{h}{2}\right)^{2s-2}
\ds\frac{\Gamma(r-s+1)}{\Gamma(r+s-1)}\|u\|^2_{H^{s+1}(J)}
\end{equation}
for any real $s$,  $1\leq s \leq \min\{r ,s_0\}$. Here,
$\Gamma(\cdot)$ is the usual gamma function.
\end{col}

Given  an arbitrary partition $\cal T_h$ of $(0,T)$ with $N$
subintervals $I_n, ~1\le n\le N$. For any $u\in H^2(I)$, we can now
 define a piecewise polynomial $\Pi^{\bf{r}}u$   by
\begin{equation}\label{pir}
\Pi^{{\bf{r}}}u |_{I_n}= \Pi_{I_n}^{r_n}u,\quad   1\le n\le N,
\end{equation} where
$r_n \ge 2$  and $\Pi_{I_n}^{r_n}u$ is defined as (\ref{scal-proj}).
Due to Definition \ref{proj} and Lemma \ref{wellpose}, if $r_n\ge
3$,  then
\begin{equation}\label{C1-con}
\Pi^{\bf{r}}u (t_{n}) = u(t_{n}), \quad (\Pi^{\bf{r}}u)' (t_{n}) =
u'(t_{n}), \quad 0\le n\le N.
\end{equation}
Therefore, we have $\Pi^{\bf{r}}u \in S^{{\bf r}, 2 }(\cal T_h)$.
Moreover, from (\ref{def-proj}) we find
\begin{equation} \label{h1-orth}
\ds\int_{I_n} (u-\Pi^{\bf{r}}u)'' \varphi dt=0,\quad  \forall
\varphi \in P_{r_n-2}(I_n).
\end{equation}

As a direct consequence of   Corollary \ref{scal},  we   have the
following  results.

\begin{lemma}\label{Iu-pro}
Let $\cal T_h$ be an arbitrary partition of $I=(0,T)$. Assume that
$u\in H^{2}(I)$ satisfies $u|_{I_n}\in H^{s_{0,n}+1}$ for
$s_{0,n}\geq 1$ and $r_n \ge 3$, then we have
\begin{equation}\label{Iu-L2}
\|u-\Pi^{\bf{r}}u\|^2_{L^2(I)}\le C \ds\sum_{n=1}^N
\Big(\ds\frac{k_n}{2}\Big)^{2s_n+2}
\ds\frac{\Gamma(r_n-s_n+1)}{\Gamma(r_n+s_n-1)}\ds\frac{1}{(r_n-2)^4}\|u\|^2_{H^{s_n+1}(I_n)},
\end{equation}
\begin{equation}\label{Iu-H1}
\|(u-\Pi^{\bf{r}}u)'\|^2_{L^2(I)}\le C \ds\sum_{n=1}^N
\Big(\ds\frac{k_n}{2}\Big)^{2s_n}
\ds\frac{\Gamma(r_n-s_n+1)}{\Gamma(r_n+s_n-1)}\ds\frac{1}{r_n(r_n-1)}\|u\|^2_{H^{s_n+1}(I_n)},
\end{equation}
\begin{equation}\label{Iu-H2}
\|(u-\Pi^{\bf{r}}u)''\|^2_{L^2(I)}\le C \ds\sum_{n=1}^N
\Big(\ds\frac{k_n}{2}\Big)^{2s_n-2}
\ds\frac{\Gamma(r_n-s_n+1)}{\Gamma(r_n+s_n-1)}\|u\|^2_{H^{s_n+1}(I_n)}
\end{equation}
for any real $s_n$,  $1\leq s_n \leq \min\{r_n, s_{0,n}\}$.

Moreover, if $u\in H^2(I)$  satisfies $u|_{I_n} \in
W^{s_{0,n}+1,\infty}(I_n)$ for $s_{0,n} \ge 1$,  then we have
\begin{eqnarray}\label{Iu-infty-1}
\|u-\Pi^{\bf{r}}u\|^2_{L^\infty(I_n)}\le C
\Big(\frac{k_n}{2}\Big)^{2s_n+2} \ds\frac{
\Gamma(r_n-s_n+1)}{\Gamma(r_n+s_n-1)}\ds\frac{1}{(r_n-2)^3}\|u\|^2_{W^{s_n+1,\infty}(I_n)},
\end{eqnarray}
for any real  $s_n$, $1\le s_n \le \min\{r_n, s_{0,n}\}$.
\end{lemma}
\begin{proof}
The estimates (\ref{Iu-L2})-(\ref{Iu-H2}) are direct consequence of
Corollary \ref{scal}.
 Since $(u-\Pi^{\bf{r}}u)(t_{n-1})=(u-\Pi^{\bf{r}}u)(t_{n})=0$, we have $u-\Pi^{\bf{r}}u\in H^1_0(I_n)$.
 Hence, using the Sobolev inequality  \cite{ILLZ}
$$\|v\|^2_{L^{\infty}(a,b)}\leq \|v\|_{L^2(a,b)}\|v'\|_{L^2(a,b)},\quad \forall v\in H^1_0(a,b),$$
and the estimates (\ref{scal-1}) and (\ref{scal-2}) yields
(\ref{Iu-infty-1}).
\end{proof}

\subsection{Abstract error bounds}

Let $u$ be the exact solution of (\ref{IVP}) and $U$ be the
  $C^1$-CPG solution given by  (\ref{C1CPG-FEM}). We split
the error into   two parts:
\begin{equation}\label{err-split}
e:=u-U= \eta + \xi,
\end{equation}
where $\eta :=u-\Pi^{\bf{r}}u$ and $\xi :=\Pi^{\bf{r}}u-U$.

 Since
Lemma \ref{Iu-pro} can be used to bound $\eta$,  it remains to
consider $\xi$. To this end, we need  the following discrete
Gronwall inequality (see, e.g., \cite{Bru-1}).

\begin{lemma}\label{gronwall}
Let $\{ a_n\}_{n=1}^N$ and  $\{ b_n\}_{n=1}^N$ be two sequences of
nonnegative real numbers with $b_1\le b_2 \le \cdots \le b_N$.
Assume that there exist a   constant $C>0$ and weights $w_i
>0, ~1\le i \le N-1$ such that
$$ a_1 \le b_1, \quad a_n \le b_n +C \sum_{i=1}^{n-1} w_i a_i, \quad
2 \le n \le N.$$ Then
$$ a_n \le b_n \exp(C \sum_{i=1}^{n-1} w_i), \quad 1\le n\le N.$$
\end{lemma}

We   show that $\xi$ can be bounded by $\eta$.

\begin{lemma}\label{main-lemma}
 For $k_n$ sufficiently small and $r_n \ge 3$, there hold
\begin{equation}  \label{xi-err-main-H1}
\|\xi\|_{H^1(0, t_n)} \le C  \|\eta\|_{H^1(0, t_{n})},
\end{equation}
\begin{equation}  \label{xi-err-main-H2}
 \|\xi\|_{H^2(0, t_n)} \leq C \|\eta\|_{H^1(0, t_n)},
\end{equation}
for $1\le n \le N$, where $\xi$ and $\eta$ are
defined by the splitting (\ref{err-split}), and the constant $C>0$
solely depends on $L$ and $t_n$.
\end{lemma}
\begin{proof}
In view of (\ref{IVP}) and (\ref{C1CPG-FEM-1}), we have
\begin{eqnarray}\label{orthog}
\ds\int_{I_n} e''\varphi dt = \ds\int_{I_n} \left(f(t, u, u')-f(t,
U, U')\right) \varphi dt, \quad \forall \varphi \in P_{r_n-2}(I_n).
\end{eqnarray}
Then,   by  (\ref{h1-orth}) we get
\begin{equation}  \label{xi-1}
\ds\int_{I_n} \xi'' \varphi dt  = \ds\int_{I_n} \left(f(t, u,
u')-f(t, U, U')\right) \varphi dt, \quad \forall \varphi \in
P_{r_n-2}(I_n).
\end{equation}

For any $v\in L^2(I_n)$, we denote by $\pi^{r_n-2}v\in
P_{r_n-2}(I_n)$ the $L^2$-projection of $v$ onto $P_{r_n-2}(I_n)$
with $r_n \ge 2$, namely,
\begin{equation}\label{L2-projc}
\ds\int_{I_n} (v- \pi^{r_n-2}v)\varphi dt =0, \quad \forall \varphi
\in P_{r_n-2}(I_n).
\end{equation}
Then  for any $v\in H^{s+1}(I_n)$ with $s\ge 0$, there holds
(cf. \cite{Schwab})
\begin{equation}\label{L2-projc-appro}
\|v- \pi^{r_n-2}v\|_{L^2(I_n)} \le C
\frac{k_n^{\min\{r_n-2,s\}+1}}{r_n^{s+1}}\|v\|_{H^{s+1}(I_n)} .
\end{equation}

 Selecting $\varphi=\pi^{r_n-2} \xi'$ in
(\ref{xi-1}) and  using (\ref{Lip-con}),  yields
\begin{equation*}  \label{xi-2}
\begin{aligned}
\ds\int_{I_n} \xi''  \xi' dt = & \ds\int_{I_n}\left(f(t, u, u')-f(t,
U, U')\right) \pi^{r_n-2} \xi' dt \\ \le& L\left\{\ds\int_{I_n}
\left(|u-U|+|u'-U'|\right)^2dt\right\}^{\frac{1}{2}}
\left\{\ds\int_{I_n} |\pi^{r_n-2} \xi'|^2dt\right\}^{\frac{1}{2}}\\
\leq& \sqrt{2}L\left\{\ds\int_{I_n}
\left(|u-U|^2+|u'-U'|^2\right)dt\right\}^{\frac{1}{2}} \|\pi^{r_n-2}
\xi'\|_{L^2(I_n)}\\ =&\sqrt{2}L\|e\|_{H^1(I_n)}\|\pi^{r_n-2}
\xi'\|_{L^2(I_n)},
\end{aligned}
\end{equation*}
which  together with the $L^2$-stability of the projection operator
$\pi^{r_n-2}$ leads to
\begin{equation*}
\ds\frac{1}{2} ( |\xi'(t_n)|^2 - |\xi'(t_{n-1})|^2  ) \leq \sqrt{2}L\|e\|_{H^1(I_n)}\| \xi'\|_{L^2(I_n)}.
\end{equation*}
Hence, we have
\begin{eqnarray}  \label{xi-3}
\ds|\xi'(t_n)|^2   \le |\xi'(t_{n-1})|^{2} + \sqrt{2}L\|e\|^2_{H^1(I_n)}+\sqrt{2}L\| \xi'\|^2_{L^2(I_n)}.
\end{eqnarray}
 Moreover, selecting
$\varphi=\pi^{r_n-2}((t_{n-1}-t)\xi')$ in (\ref{xi-1}), gives
\begin{equation*}
\ds\int_{I_n} (t_{n-1}-t) \xi''  \xi' dt =\ds\int_{I_n}\left(f(t, u,
u')-f(t, U, U')\right) \pi^{r_n-2}((t_{n-1}-t)\xi') dt,
\end{equation*}
which together with (\ref{Lip-con}) yields that
\begin{eqnarray}\label{xi-4}
\begin{aligned}
&\ds\frac{1}{2}(-k_n|\xi'(t_n)|^2+\|\xi'\|^2_{L^2(I_n)})\le
L\ds\int_{I_n}\Big(|u-U|+|u'-U'|\Big)|\pi^{r_n-2}((t_{n-1}-t)\xi')|
dt\\
\leq& \sqrt{2}L\left\{\ds\int_{I_n} \left(|u-U|^2+|u'-U'|^2\right)dt\right\}^{\frac{1}{2}}
\|\pi^{r_n-2}((t_{n-1}-t)\xi')\|_{L^2(I_n)}\\
=&\sqrt{2}L\|e\|_{H^1(I_n)}\|\pi^{r_n-2}((t_{n-1}-t)\xi')\|_{L^2(I_n)}.
\end{aligned}
\end{eqnarray}
Noting  the fact that
$$ \| \pi^{r_n-2}((t_{n-1}-t)\xi')\|_{L^2(I_n)} \le   \| (t_{n-1}-t)\xi'\|_{L^2(I_n)} \le k_n \|\xi'\|_{L^2(I_n)}, $$
and  using (\ref{xi-4}) implies
\begin{eqnarray}\label{xi-5}
\begin{aligned}
\|\xi'\|^2_{L^2(I_n)}\leq&
k_n|\xi'(t_n)|^2+2\sqrt{2}Lk_n\|e\|_{H^1(I_n)}\|\xi'\|_{L^2(I_n)}\\
\leq &
k_n|\xi'(t_n)|^2+\sqrt{2}Lk_n\|e\|^2_{H^1(I_n)}+\sqrt{2}Lk_n\|\xi'\|^2_{L^2(I_n)}.
\end{aligned}
\end{eqnarray}
Assume that $k_n$ is sufficiently small such that
 $\sqrt{2}Lk_n <1$,  (\ref{xi-5}) can be rewritten as
 \begin{eqnarray}\label{xi-6}
\|\xi'\|^2_{L^2(I_n)}\leq \ds\frac{k_n}{1-\sqrt{2}Lk_n }|\xi'(t_n)|^2+\ds\frac{\sqrt{2}Lk_n}{1-\sqrt{2}Lk_n }\|e\|^2_{H^1(I_n)}.
\end{eqnarray}
Inserting (\ref{xi-6}) into (\ref{xi-3}) we deduce that
\begin{eqnarray*}
 |\xi'(t_n)|^2 \le |\xi'(t_{n-1})|^2+ \ds\frac{\sqrt{2}Lk_n}{1-\sqrt{2}Lk_n }|\xi'(t_n)|^2+\ds\frac{\sqrt{2}L}{1-\sqrt{2}Lk_n }\|e\|^2_{H^1(I_n)}.
\end{eqnarray*}
We further assume that $k_n$ is sufficiently small and there exists
a constant $\gamma >0$
 such that
 $2\sqrt{2}Lk_n \le \gamma  <1$,  then
\begin{eqnarray}\label{xi-7}
 |\xi'(t_n)|^2\leq \left(1+\ds\frac{\sqrt{2}L}{1-2\sqrt{2}Lk_n }k_n\right)|\xi'(t_{n-1})|^2+\ds\frac{\sqrt{2}L}{1-2\sqrt{2}Lk_n }\|e\|^2_{H^1(I_n)}.
\end{eqnarray}
Summing up (\ref{xi-7}) over the subintervals $I_i,~1 \leq  i \leq
n$, and using the facts that $\xi'(t_0)=0$ and
$\xi'|_{I_i}(t_i)=\xi'|_{I_{i+1}}(t_i),~ 1 \leq i \leq n-1,$ we
obtain
\begin{eqnarray}\label{xi-8}
 |\xi'(t_n)|^2 \le \ds\frac{\sqrt{2}L}{1-\gamma}\ds\sum_{i=1}^{n-1}k_{i+1} |\xi'(t_{i})|^{2}
 +\ds\frac{\sqrt{2}L}{1-\gamma}\ds\sum_{i=1}^{n}\|e\|^2_{H^1(I_i)}.
\end{eqnarray}
Applying the discrete Gronwall inequality in Lemma \ref{gronwall} to
(\ref{xi-8}),  gives
\begin{eqnarray}\label{xi-9}
\begin{aligned}
|\xi'(t_n)|^2 \le&
\ds\frac{\sqrt{2}L}{1-\gamma}\ds\sum_{i=1}^{n}\|e\|^2_{H^1(I_i)}\exp\Big(
\ds\frac{\sqrt{2}L}{1-\gamma} \ds\sum_{i=1}^{n-1}k_{i+1} \Big)\le
Ce^{Ct_n}\|e\|^2_{H^1(0,t_n)}\\ \le& C
\left(\|\eta\|^2_{H^1(0,t_n)}+\|\xi\|^2_{H^1(0,t_n)}\right),
\end{aligned}
\end{eqnarray}
with the constant $C >0$ depends on $L,  \gamma$ and $t_n$.
Inserting (\ref{xi-9}) into (\ref{xi-6}),  we conclude that
\begin{eqnarray}\label{xi-10}
\begin{aligned}
\|\xi'\|^2_{L^2(I_n)} \le& \ds\frac{Ck_n}{1-\gamma}\left(\|\eta\|^2_{H^1(0,t_n)}
+\|\xi\|^2_{H^1(0,t_n)}\right)+\ds\frac{\sqrt{2}Lk_n}{1-\gamma}\|e\|^2_{H^1(I_n)}\\
\le& Ck_n\|\eta\|^2_{H^1(0,t_n)}+ Ck_n\|\xi\|^2_{H^1(0,t_n)}.
\end{aligned}
\end{eqnarray}

On the other hand, using the fact that
$\xi(t)=\ds\int_{t_{n-1}}^{t}\xi'(s) ds+\xi(t_{n-1})$, yields
\begin{eqnarray}\label{xi-12}
\begin{aligned}
\|\xi\|^2_{L^2(I_n)}=&\ds\int_{I_n}\Big(\ds\int_{t_{n-1}}^{t}\xi'(s)ds+\xi(t_{n-1})\Big)^2
dt \leq 2\ds\int_{I_n}\Big(\ds\int_{t_{n-1}}^{t}\xi'(s)ds\Big)^2dt+2k_n|\xi(t_{n-1})|^2\\
\leq&2\ds\int_{I_n}(t-t_{n-1})\Big(\ds\int_{t_{n-1}}^{t}|\xi'(s)|^2ds\Big)dt+2k_n|\xi(t_{n-1})|^2\\
\leq&k_n^2\|\xi'\|^2_{L^2(I_n)}+2k_n|\xi(t_{n-1})|^2.
\end{aligned}
\end{eqnarray}
Since $\Pi^{\bf{r}}u \in H^2(I)$  if $r_n\geq 3$ for $1\leq n\leq
N$, and thus $\xi\in H^2(0,t_n)$. Noting that $\xi(0)=0$, then
\begin{equation}\label{xi-13}
|\xi(t_{n-1})|^2=\left(\ds\int_{0}^{t_{n-1}}\xi'(t)dt\right)^2 \leq
t_{n-1}\|\xi'\|^2_{L^2(0,t_{n-1})}.
\end{equation}
Combining (\ref{xi-12}) and  (\ref{xi-13}), gives
\begin{equation}\label{xi-14}
\|\xi\|^2_{L^2(I_n)}\leq
k_n^2\|\xi'\|^2_{L^2(I_n)}+2t_{n-1}k_n\|\xi'\|^2_{L^2(0,t_{n-1})}
\leq C k_n\|\xi'\|^2_{L^2(0,t_{n})}.
\end{equation}
By (\ref{xi-10}) and (\ref{xi-14}), we  find
 \begin{eqnarray}\label{xi-15}
\begin{aligned}
\|\xi\|^2_{H^1(I_n)}\leq&Ck_n\|\eta\|^2_{H^1(0,t_n)}+Ck_n\|\xi\|^2_{H^1(0,t_{n})}
+Ck_n\|\xi'\|^2_{L^2(0,t_{n})}\\ \leq&
Ck_n\|\eta\|^2_{H^1(0,t_n)}+Ck_n\|\xi\|^2_{H^1(0,t_{n})}.
\end{aligned}
\end{eqnarray}
Assume that $k_n$ is sufficiently small, then (\ref{xi-15}) can be
rewritten as
\begin{equation*}\label{xi-16}
\|\xi\|^2_{H^1(I_n)}
\leq Ck_n\|\eta\|^2_{H^1(0,t_n)}+Ck_n\|\xi\|^2_{H^1(0,t_{n-1})},
\end{equation*}
or equivalently,
\begin{equation*}\label{xi-17}
\ds\frac{\|\xi\|^2_{H^1(I_n)}}{k_n}
\leq C\|\eta\|^2_{H^1(0,t_n)}+C\ds\sum_{i=1}^{n-1}\ds\frac{\|\xi\|^2_{H^1(I_i)}}{k_i}k_i.
\end{equation*}
Applying the discrete Gronwall inequality in Lemma \ref{gronwall} to
the above inequality,  yields
\begin{equation*}\label{xi-18}
\ds\frac{\|\xi\|^2_{H^1(I_n)}}{k_n} \leq
C\|\eta\|^2_{H^1(0,t_n)}\exp({C \sum\limits_{i=1}^{n-1}k_i}),
\end{equation*}
which implies
\begin{equation}\label{xi-19}
\|\xi\|^2_{H^1(I_n)}
\leq Ck_n\|\eta\|^2_{H^1(0,t_n)}.
\end{equation}
Summing up the above estimate over the subintervals $I_i,~1\leq
i\leq n$, gives
\begin{equation}\label{xi-20}
\|\xi\|^2_{H^1(0,t_n)}=\ds\sum_{i=1}^{n}\|\xi\|^2_{H^1(I_n)} \leq
C\ds\sum_{i=1}^{n}k_i\|\eta\|^2_{H^1(0,t_i)} \leq
\|\eta\|^2_{H^1(0,t_n)}\ds\sum_{i=1}^{n}k_i \leq C
\|\eta\|^2_{H^1(0,t_n)}.
\end{equation}
This completes the proof of (\ref{xi-err-main-H1}).

We now turn to the proof of (\ref{xi-err-main-H2}). Selecting
$\varphi= \xi''$ in (\ref{xi-1}) and using (\ref{Lip-con}), leads to
\begin{equation*}  \label{xi-22}
\begin{aligned}
\ds\int_{I_n} |\xi''|^2 dt \le& L\ds\int_{I_n}\Big(|u-U|+|u'-U'|\Big)|\xi''| dt
\leq \sqrt{2}L\left\{\ds\int_{I_n}
\left(|u-U|^2+|u'-U'|^2\right)dt\right\}^{\frac{1}{2}}
\| \xi''\|_{L^2(I_n)}\\
=&\sqrt{2}L\|e\|_{H^1(I_n)}\|\xi''\|_{L^2(I_n)},
\end{aligned}
\end{equation*}
which implies that
\begin{equation}  \label{xi-23}
 \|\xi''\|^2_{L^2(I_n)} \leq 2L^2\|e\|^2_{H^1(I_n)}.
\end{equation}
Summing up (\ref{xi-23}) over all subintervals $I_i,~1\leq i\leq n$,
gives
\begin{equation}  \label{xi-24}
 \|\xi''\|^2_{L^2(0,t_n)}=\ds\sum_{i=1}^{n} \|\xi''\|^2_{L^2(I_i)} \leq\ds\sum_{i=1}^{n} 2L^2\|e\|^2_{H^1(I_i)}=2L^2\|e\|^2_{H^1(0,t_n)}.
\end{equation}
Combing (\ref{xi-20}) and (\ref{xi-24}),  yields
\begin{equation}  \label{xi-25}
 \|\xi''\|^2_{L^2(0,t_n)}\leq4L^2(\|\eta\|^2_{H^1(0,t_n)}+\|\xi\|^2_{H^1(0,t_n)})\leq C \|\eta\|^2_{H^1(0,t_n)}.
\end{equation}
Moreover,  by (\ref{xi-20}) and (\ref{xi-25}), we obtain
\begin{equation}  \label{xi-26}
 \|\xi\|^2_{H^2(0,t_n)}= \|\xi\|^2_{H^1(0,t_n)}+ \|\xi''\|^2_{L^2(0,t_n)}\leq C \|\eta\|^2_{H^1(0,t_n)}.
\end{equation}
This ends the proof of (\ref{xi-err-main-H2}).
\end{proof}

\begin{rem}
It is worth noting that the $hp$-version of the  $C^1$-CPG  scheme (\ref{C1CPG-FEM})
is designed for $r_n \ge 2$ with $1\le n \le N$, but the estimates
presented in Lemma \ref{main-lemma} only hold for $r_n\ge 3$. The
main reason is that we have to make use of the globally
$C^1$-continuity of the piecewise polynomial $\Pi^{\bf{r}}u$ in the
proof (see (\ref{xi-13})). However, the   $C^1$-continuity of
$\Pi^{\bf{r}}u$  only holds for $r_n \ge 3$ (see (\ref{C1-con})).
\end{rem}

\subsection{$H^1$- and $H^2$-error
estimates}

In this section, we shall prove optimal $H^1$- and $H^2$-error
estimates of the $C^1$-CPG method. The following results shows that
the global $H^1$- and $H^2$-error estimates can be bounded by the
approximation errors of $\Pi^{\bf{r}}u$ in $H^1$- and $H^2$-norms,
respectively.

\begin{lemma}\label{hp-error}
 Let $u$ be the exact solution of (\ref{IVP}) and $U$ be the $C^1$-CPG solution of (\ref{C1CPG-FEM}). Assume that $k_n$ is sufficiently small and $r_n
 \ge 3$. Then,  we have
\begin{eqnarray} \label{abs-err-H1}
\| u-U \|_{H^1(I)} \le C   \|u-\Pi^{\bf{r}}u\|_{H^1(I)},
\end{eqnarray}
\begin{eqnarray} \label{abs-err-H2}
\| u-U \|_{H^2(I)} \le C   \|u-\Pi^{\bf{r}}u\|_{H^2(I)},
\end{eqnarray}
where the constant $C > 0$ solely depends on $L$ and $T$.
\end{lemma}
\begin{proof}
Since $u-U=\eta+\xi$, by (\ref{xi-err-main-H1}) we get
\begin{eqnarray*}
 \|u-U\|_{H^1(I)} \le  \|\eta\|_{H^1(I)} +\|\xi\|_{H^1(I)} \le  C
 \|\eta\|_{H^1(I)},
\end{eqnarray*}
which implies (\ref{abs-err-H1}).

Moreover,  by (\ref{xi-err-main-H2}) we  obtain
\begin{eqnarray*}
 \|u-U\|_{H^2(I)} \le  \|\eta\|_{H^2(I)} +\|\xi\|_{H^2(I)}
 \le \|\eta\|_{H^2(I)} +C\|\eta\|_{H^1(I)}
 \le  C \|\eta\|_{H^2(I)} .
\end{eqnarray*}
This completes the proof  of (\ref{abs-err-H2}).
\end{proof}

As a direct consequence of Lemma \ref{hp-error} and the
approximation properties of $\Pi^{\bf{r}}u$ as stated in Lemma
\ref{Iu-pro}, we obtain the following typical $hp$-version error
estimates of the $C^1$-CPG method for   (\ref{IVP}).

\begin{theorem}\label{hp-error-111}
Let $\mathcal{T}_h$ be an arbitrary partition of $I$,  $u$ be the
exact solution of (\ref{IVP}) and $U$ be the $C^1$-CPG solution of
(\ref{C1CPG-FEM}). Assume that $u \in H^2(I)$ satisfies $u|_{I_n} \in
H^{s_{0,n}+1}(I_n)$ for  $s_{0,n} \ge 1$. Then, for $k_n$
sufficiently small and  $r_n\ge 3$, we have
\begin{eqnarray}\label{cor-err-H1-1}
 \|u-U\|^2_{H^1(I)} \le C  \ds\sum_{n=1}^N
\Big(\ds\frac{k_n}{2}\Big)^{2s_n}
\ds\frac{\Gamma(r_n-s_n+1)}{\Gamma(r_n+s_n-1)}\ds\frac{1}{r_n(r_n-1)}\|u\|^2_{H^{s_n+1}(I_n)},
\end{eqnarray}
\begin{eqnarray}\label{cor-err-H2-2}
 \|u-U\|^2_{H^2(I)} \le C \ds\sum_{n=1}^N
\Big(\ds\frac{k_n}{2}\Big)^{2s_n-2}
\ds\frac{\Gamma(r_n-s_n+1)}{\Gamma(r_n+s_n-1)}\|u\|^2_{H^{s_n+1}(I_n)}
\end{eqnarray}
for any real $s_n$,  $1\le s_n \le \min\{r_n, s_{0,n} \}$, where the
constant $C>0$ is independent of $k_n,~r_n$ and $s_n$.

In particular, if $\mathcal T_h$ is a quasi-uniform partition, i.e.,
there exists a constant $C_q \ge 1$ such that $1\le {k}/{k_n} \le
C_q, ~1\le n\le N$.   Assume that $u\in   H^{s+1}(I)$ with $s \ge 1$
and  $r_n \equiv r$. Then
\begin{eqnarray}\label{cor-err-H1-quasi}
 \|u-U\|_{H^1(I)} \le C
 \ds\frac{k^{\min\{r,s\}}}{r^s}\|u\|_{H^{s+1}(I)},
 \end{eqnarray}
 \begin{eqnarray}\label{cor-err-H2-quasi}
 \|u-U\|_{H^2(I)} \le C
 \ds\frac{k^{\min\{r,s\}-1}}{r^{s-1}}\|u\|_{H^{s+1}(I)},
 \end{eqnarray}
where the   constant $C > 0$ is independent of $k$ and $r$.
 \end{theorem}
\begin{proof}
The assertions (\ref{cor-err-H1-1}) and (\ref{cor-err-H2-2}) follow
from Lemmas \ref{hp-error} and  \ref{Iu-pro} immediately. Using
(\ref{cor-err-H1-1}), (\ref{cor-err-H2-2}) and the Stirling's formula,
we obtain (\ref{cor-err-H1-quasi}) and (\ref{cor-err-H2-quasi})
directly.
\end{proof}

\begin{rem}
The estimates (\ref{cor-err-H1-1}) and (\ref{cor-err-H2-2})
 are totally  explicit in the local time steps $k_n$, in  the local  approximation
 degrees $r_n$, and in  the local regularities  $s_n$ of the solution.
 The estimates (\ref{cor-err-H1-quasi}) and (\ref{cor-err-H2-quasi}) imply that the
$hp$-version $C^1$-CPG method can  achieve the desired accuracy by
increasing the polynomial approximation degree $r$ or/and decreasing
the time step $k$.  In particular,  these estimates also show that
the $p$-version $C^1$-CPG method on fixed  time partition
  can yield arbitrarily high-order convergence rate (i.e., spectral convergence) as long as
  the solution  $u$ is smooth enough (i.e., the regularity index $s$ is
 large enough).
\end{rem}

\subsection{$L^2$-error estimate}

In this section, we shall prove optimal $L^2$-error estimate of  the
$hp$-version $C^1$-CPG method based on a duality argument.


Using the Taylor's theorem with Lagrange remainder for the function
$f(t,u,u')$ in the  variables  $u$ and $u'$, we find that there
exist functions $\lambda_1$ and $\lambda_2$, such that the value
$\lambda_1(t)$ is between $u(t)$ and $U(t)$, the value
$\lambda_2(t)$ is between $u'(t)$ and $U'(t)$, and there hold
\begin{eqnarray}\label{tay-1}
\begin{aligned}
f(t,U,U')=&f(t,u,u')-f_u(t,u,u')e-f_{u'}(t,u,u')e'+
\ds\frac{1}{2}f_{uu}(t,\lambda_1,\lambda_2)e^2\\
&+\ds\frac{1}{2}f_{uu'}(t,\lambda_1,\lambda_2)ee'
+\ds\frac{1}{2}f_{u'u}(t,\lambda_1,\lambda_2)ee'
+\ds\frac{1}{2}f_{u'u'}(t,\lambda_1,\lambda_2)(e')^2,
\end{aligned}
\end{eqnarray}
where   $e=u-U$ and $e'=u'-U'$. Assume that $f_{uu'}(t,u,u')$ and
$f_{u'u}(t,u,u')$ are continuous with respect to the variables $u$
and $u'$, then we have
$f_{uu'}(t,\lambda_1,\lambda_2)=f_{u'u}(t,\lambda_1,\lambda_2)$, and
thus    (\ref{tay-1})   can be  rewritten  as
\begin{equation}\label{tay-5}
f(t,u,u')-f(t,U,U')=\theta_1(t) e + \theta_2(t)
e'+R_1(t)e^2+R_2(t)ee'+R_3(t)(e')^2,
\end{equation}
where
$$
\theta_1(t):=f_u(t,u,u'),\quad \theta_2(t):=f_{u'}(t,u,u'),$$
$$R_1(t):=-\ds\frac{1}{2}f_{uu}(t,\lambda_1(t),\lambda_2(t)),\quad
R_2(t):=-f_{uu'}(t,\lambda_1(t),\lambda_2(t)),\quad
R_3(t):=-\ds\frac{1}{2}f_{u'u'}(t,\lambda_1(t),\lambda_2(t)).
$$

We are now ready to state the main results of this section.

\begin{theorem}\label{glo-L2-error}
Let $\mathcal{T}_h$ be an arbitrary partition of $I$,  $u$ be the
exact solution of (\ref{IVP}) and $U$ be the $C^1$-CPG solution of
(\ref{C1CPG-FEM}). Assume that   $f(t,u,u')$ is sufficiently smooth
with respect to the variables $t, u$ and $u'$ on $\bar{I}\times
\mathbb{R} \times \mathbb{R}$. We further assume that $u \in H^2(I)$
satisfies $u|_{I_n} \in H^{s_{0,n}+1}(I_n)$ for $s_{0,n} \ge 1$.
Then, for $k_n$ sufficiently small and  $r_n\ge 3$, we have
 \begin{equation}\label{L2-error}
 \begin{aligned}
\|u-U\|^2_{L^2(I)} \le & C\max\limits_{1\le n\le N} \Big\{
\Big(\ds\frac{k_n}{r_n}\Big)^{4} \Big\}  \ds\sum_{n=1}^N
\Big(\ds\frac{k_n}{2}\Big)^{2s_n-2}
\ds\frac{\Gamma(r_n-s_n+1)}{\Gamma(r_n+s_n-1)}\|u\|^2_{H^{s_n+1}(I_n)}\\
&+ C\max\limits_{1\le n\le N} \Big\{
\Big(\ds\frac{k_n}{r_n}\Big)^{3} \Big\} \left(  \ds\sum_{n=1}^N
\Big(\ds\frac{k_n}{2}\Big)^{2s_n-2}
\ds\frac{\Gamma(r_n-s_n+1)}{\Gamma(r_n+s_n-1)}\|u\|^2_{H^{s_n+1}(I_n)}
\right)^2
\end{aligned}
\end{equation}
for any real $s_n$,  $1\le s_n \le \min\{r_n, s_{0,n} \}$, where the
constant $C>0$ is independent of $k_n,~r_n$ and $s_n$.

In particular,  if $\mathcal T_h$ is a quasi-uniform partition,  $r_n
\equiv r$ and $u\in   H^{s+1}(I)$  with $s \ge \frac{3}{2}$, then
\begin{eqnarray}\label{unif-L2}
 \|u-U\|_{L^2(I)} \le C
 \ds\frac{k^{\min\{s,r\}+1}}{r^{s+1}}\|u\|_{H^{s+1}(I)},
 \end{eqnarray}
 where the constant $C > 0$ is independent of $k$ and $r$.
\end{theorem}
\begin{proof}
  We first construct the following auxiliary problem: find $g$
such that
\begin{equation}\label{aux-1}
 \left\{\begin{array}{ll}
  g''+(\theta_2g)' -\theta_1g=e,\quad & t\in [0,T],
\\[5pt]g(T)=0,\quad g'(T)=0,
\end{array}\right.
\end{equation}
where $e=u-U$ and the coefficients $\theta_1, \theta_2$ are given in
(\ref{tay-5}). Suppose that $\theta_1(t)=f_u(t,u,u')$ and
$\theta_2(t)=f_{u'}(t,u,u')$ are sufficiently smooth. We may assume
that
\begin{equation}\label{reg}
\|g\|_{H^2(I)}\le C\|e\|_{L^2(I)}.
\end{equation}

From (\ref{orthog}) and (\ref{tay-5}),  we have
\begin{eqnarray}\label{orht-1}
\ds\int_{I_n} \left( e''- \theta_2 e'- \theta_1e\right)\varphi dt =
\ds\int_{I_n}\left( R_1e^2+R_2ee'+R_3(e')^2\right)\varphi dt , \quad
\forall \varphi \in P_{r_n-2}(I_n).
\end{eqnarray}
For convenience, we set $$ \delta =\int_{0}^{T} \left(e''-\theta_2
e'-\theta_1e\right)g dt,$$ where $g$ is the solution of
(\ref{aux-1}). Using
  integration by parts, the fact $e(0)=e'(0)=0$   and (\ref{aux-1}),
  yields
\begin{eqnarray}\label{L2-error-3}
\delta=\int_{0}^{T} \left(e''-\theta_2 e'-\theta_1e\right)g dt
=\ds\int_{0}^{T}\left( g''+ (\theta_2g)' -\theta_1g
\right)edt=\ds\int_{0}^{T}e^2dt=\|e\|^2_{L^2(I)}.
\end{eqnarray}
For our purpose,  we  define a piecewise $L^2$-projection of $g$ by
 \begin{equation}\label{L2-pir-aa}
 \Pi_*^{{\bf{r-2}}}g |_{I_n}= \pi^{r_n-2}g,\quad   1\le n\le N,
\end{equation}
where $\pi^{r_n-2}$ with $r_n \ge 2$ is the $L^2$-projection
operator given in (\ref{L2-projc}).

Assume that   $f(t,u,u')$ is sufficiently smooth with respect to the
variables $t, u$ and $u'$ on $\bar{I}\times \mathbb{R} \times
\mathbb{R}$ such that $|R_1(t)|, |R_2(t)|, |R_3(t)| \leq C$ for
$t\in \bar{I}$, respectively. Then,  using (\ref{L2-error-3}),
(\ref{orht-1}), (\ref{L2-projc-appro}), (\ref{reg}) and the
$L^2$-stability of the projection operator $\pi^{r_n-2}$,  gives
\begin{eqnarray*}
\begin{aligned}
\|e\|^4_{L^2(I)}
=& \left(\ds\int_{0}^{T} (e''- \theta_2 e'-\theta_1 e)(g-\Pi_*^{\bf{r-2}}g)dt+\ds\int_{0}^{T} (e''- \theta_2 e'-\theta_1 e)\Pi_*^{\bf{r-2}}gdt\right)^2\\
\le&2\left(\ds\int_{0}^{T} (e''- \theta_2 e'-\theta_1 e)(g-\Pi_*^{\bf{r-2}}g)dt\right)^2+2\left(\ds\int_{0}^{T}\left( R_1e^2+R_2ee'+R_3(e')^2\right)\Pi_*^{\bf{r-2}}g dt\right)^2\\
\le&C\|e\|^2_{H^2(I)}\|g-\Pi_*^{\bf{r-2}}g\|^2_{L^2(I)}+C\ds\int_{0}^{T}\left(e^2+(e')^2\right)^2dt  \|\Pi_*^{\bf{r-2}}g\|^2_{L^2(I)} \\
\le&C\|e\|^2_{H^2(I)}\ds\sum_{n=1}^{N}\Big(\ds\frac{k_n}{r_n}\Big)^{4} \|g\|^2_{H^2(I_n)}+C\left(\ds\int_{0}^{T}e^4dt +\ds\int_{0}^{T}(e')^4dt \right)\|g\|^2_{L^2(I)}\\
\le&C\|e\|^2_{H^2(I)}\max\limits_{1\le n\le N} \Big\{
\Big(\ds\frac{k_n}{r_n}\Big)^{4} \Big\}
\|e\|^2_{L^2(I)}+C\left(\ds\int_{0}^{T}e^4dt
+\ds\int_{0}^{T}(e')^4dt \right)\|e\|^2_{L^2(I)}
\end{aligned}
\end{eqnarray*}
for $r_n \ge 3$, which implies that
\begin{eqnarray}\label{L2-error-5}
\|e\|^2_{L^2(I)} \le C \max\limits_{1\le n\le N} \Big\{
\Big(\ds\frac{k_n}{r_n}\Big)^{4} \Big\} \|e\|^2_{H^2(I)}+ C
\ds\int_{0}^{T}e^4dt +C\ds\int_{0}^{T}(e')^4dt.
\end{eqnarray}
Due to the Sobolev inequality, there holds
$\|e\|^2_{L^{\infty}(I)}\leq C\|e\|_{L^{2}(I)}\|e\|_{H^{1}(I)}$, and
then
\begin{equation}\label{L4}
\ds\int_{0}^{T}e^4dt\leq \|e\|^2_{L^{\infty}(I)}\|e\|^2_{L^2(I)}\leq C\|e\|_{H^{1}(I)}\|e\|^3_{L^2(I)}.
\end{equation}
Similarly, using the inequality  $\|e'\|^2_{L^{\infty}(I)}\leq
C\|e\|_{H^{1}(I)}\|e\|_{H^{2}(I)}$,  gives
\begin{equation}\label{H1-4}
\ds\int_{0}^{T}(e')^4dt\leq
\|e'\|^2_{L^{\infty}(I)}\|e'\|^2_{L^2(I)}
       \leq C\|e\|_{H^{2}(I)}\|e\|^3_{H^1(I)}.
\end{equation}
Inserting (\ref{L4}) and (\ref{H1-4} ) into (\ref{L2-error-5}), we
obtain
\begin{eqnarray}\label{L2-error-7}
\|e\|^2_{L^2(I)} \le C\max\limits_{1\le n\le N} \Big\{
\Big(\ds\frac{k_n}{r_n}\Big)^{4} \Big\}
\|e\|^2_{H^2(I)}+C\|e\|_{H^{2}(I)}\|e\|^3_{H^1(I)},
\end{eqnarray}
which together with the  $H^1$- and  $H^2$-error estimates in
Theorem \ref{hp-error-111} leads to
 \begin{equation*}\label{L2-error-1111}
 \begin{aligned}
\|e\|^2_{L^2(I)}\le& C\max\limits_{1\le n\le N} \Big\{
\Big(\ds\frac{k_n}{r_n}\Big)^{4} \Big\}  \ds\sum_{n=1}^N
\Big(\ds\frac{k_n}{2}\Big)^{2s_n-2}
\ds\frac{\Gamma(r_n-s_n+1)}{\Gamma(r_n+s_n-1)}\|u\|^2_{H^{s_n+1}(I_n)}\\
&+C\left(\ds\sum_{n=1}^N \Big(\ds\frac{k_n}{2}\Big)^{2s_n-2}
\ds\frac{\Gamma(r_n-s_n+1)}{\Gamma(r_n+s_n-1)}\|u\|^2_{H^{s_n+1}(I_n)}\right)^{\frac{1}{2}}\\
&\quad \times \left(\ds\sum_{n=1}^N
\Big(\ds\frac{k_n}{2}\Big)^{2s_n}
\ds\frac{\Gamma(r_n-s_n+1)}{\Gamma(r_n+s_n-1)}\ds\frac{1}{r_n(r_n-1)}\|u\|^2_{H^{s_n+1}(I_n)}\right)^{\frac{3}{2}}\\
\le & C\max\limits_{1\le n\le N} \Big\{
\Big(\ds\frac{k_n}{r_n}\Big)^{4} \Big\}  \ds\sum_{n=1}^N
\Big(\ds\frac{k_n}{2}\Big)^{2s_n-2}
\ds\frac{\Gamma(r_n-s_n+1)}{\Gamma(r_n+s_n-1)}\|u\|^2_{H^{s_n+1}(I_n)}\\
&+ C\max\limits_{1\le n\le N} \Big\{
\Big(\ds\frac{k_n}{r_n}\Big)^{3} \Big\} \left(  \ds\sum_{n=1}^N
\Big(\ds\frac{k_n}{2}\Big)^{2s_n-2}
\ds\frac{\Gamma(r_n-s_n+1)}{\Gamma(r_n+s_n-1)}\|u\|^2_{H^{s_n+1}(I_n)}
\right)^2.
\end{aligned}
\end{equation*}
This proves (\ref{L2-error}).

Moreover, if $\mathcal T_h$ is a quasi-uniform partition and  $r_n
\equiv r$, applying the Stirling's formula to (\ref{L2-error}) gives
\begin{eqnarray*}\label{L2-error-8}
\|e\|^2_{L^2(I)} \le
C\ds\frac{k^{2\min\{s,r\}+2}}{r^{2s+2}}\|u\|^2_{H^{s+1}(I)}
+C\ds\frac{k^{4\min\{s,r\}-1}}{r^{4s-1}}\|u\|^4_{H^{s+1}(I)} \le
C\ds\frac{k^{2\min\{s,r\}+2}}{r^{2s+2}}\|u\|^2_{H^{s+1}(I)}
\end{eqnarray*}
for $s\geq\frac{3}{2}$.  This proves (\ref{unif-L2}).
\end{proof}

\subsection{Nodal superconvergence estimates}

In this section, we shall prove that the $hp$-version $C^1$-CPG
method superconverges at the nodal points of the time partition
$\mathcal{T}_h$ with regard to the local time steps $k_n$ and the
local approximation degrees $r_n$.

 The main results of this section are stated in the following theorem.

\begin{theorem}\label{hp-superconvergence}
Let $\mathcal{T}_h$ be an arbitrary partition of $I$,  $u$ be the
exact solution of (\ref{IVP}) and $U$ be the $C^1$-CPG solution of
(\ref{C1CPG-FEM}). Assume that  $f(t,u,u')$ is sufficiently smooth
with respect to the variables $t, u$ and $u'$ on $\bar{I}\times
\mathbb{R} \times \mathbb{R}$. We further assume  that $u \in
H^{s+1}(I)$  with $s \ge 1$. Then, for $k_n$   sufficiently small
and $r_n\ge 3$, we have
\begin{equation}\label{super-nodes}
\begin{aligned}
|e(t_n)|^2+ |e'(t_n)|^2 \le&   C \max\limits_{1\leq i\leq n} \left\{
\ds\frac{k_i^{2\min\{r_i-2,s\}+2}}{r_i^{2(s+1)}} \right\}
\ds\sum_{i=1}^n \Big(\ds\frac{k_i}{2}\Big)^{2\min\{r_i,s\}-2}
\ds\frac{\Gamma(r_i-s+1)}{\Gamma(r_i+s-1)}\|u\|^2_{H^{s+1}(I_i)}\\
&+  C\max\limits_{1\le i\le n} \Big\{
\Big(\ds\frac{k_i}{r_i}\Big)^{3} \Big\} \left(\ds\sum_{i=1}^n
\Big(\ds\frac{k_i}{2}\Big)^{2\min\{r_i,s\}-2}
\ds\frac{\Gamma(r_i-s+1)}{\Gamma(r_i+s-1)}\|u\|^2_{H^{s+1}(I_i)}\right)^2
\end{aligned}
 \end{equation}
 for $1\le n\le N$, where $e=u-U$ and the constant $C>0$ is independent of $k_n$ and
 $r_n$.

In particular,  if $\mathcal T_h$ is a quasi-uniform partition  and
$r_n \equiv r$, then we have
\begin{equation}\label{cor-err-2}
 |e(t_n)|+ |e'(t_n)|\le C
 \ds\frac{k^{\min\{r-2,s-\frac 12\}+\min\{r,s\}}}{r^{2s-\frac 12}}\|u\|_{H^{s+1}(0, t_n)},
 \quad 1\le n\le N,
 \end{equation}
    where the constant $C > 0$ is independent of $k$ and $r$.
\end{theorem}
\begin{proof}
 We first construct the following auxiliary problem: find $w$ such
 that
\begin{equation}\label{auxx-2}
 \left\{\begin{array}{ll}
 w''+(\theta_2w)' -\theta_1w=0,\quad & t\in [0,t_n],\\[3mm]
 w(t_n)=w_0,\quad w'(t_n)=w_1
 \end{array}\right.
\end{equation}
for $1\le n\le N$, where  the coefficients $\theta_1, \theta_2$ are
given in (\ref{tay-5}),    $w_0, w_1$ are suitable terminal values
to be determined later.

 Suppose that   $f(t,u,u')$ is sufficiently
smooth with respect to the variables $t, u$ and $u'$ on
$\bar{I}\times \mathbb{R} \times \mathbb{R}$ such that   $\theta_1$
and $\theta_2$ are also smooth functions. Then, we may assume that
the   problem (\ref{auxx-2}) has a  unique solution $w$ which can be
expressed by
  $$w(t)=w_0 \varphi_1(t)+ w_1 \varphi_2(t), \quad   t\in [0,t_n],$$
with $\varphi_1$ and $\varphi_2$ be smooth functions,
  and  $w$ satisfies the
a priori estimate
\begin{equation}\label{g-prio}
\|w\|^2_{H^{s+1}(0,t_n)}\le C \left(|w_0|^2+|w_1|^2\right).
\end{equation}

For convenience, we set $$\varrho=\int_{0}^{t_n} \left(e''-\theta_2
e'-\theta_1e\right)w dt$$ with $e=u-U$ and $w$ be the solution of
(\ref{auxx-2}). Then, using integration by parts,   the fact
$e(0)=e'(0)=0$ and (\ref{auxx-2}), gives
\begin{eqnarray}\label{conp-1}
\begin{aligned}
\varrho=&w_0 \left(e'(t_n)-\theta_2(t_n)e(t_n)\right)-w_1
e(t_n)+\ds\int_{0}^{t_n} (w''+(\theta_2w)' -\theta_1w)edt\\ =&w_0
\left(e'(t_n)-\theta_2(t_n)e(t_n)\right)-w_1 e(t_n).
\end{aligned}
\end{eqnarray}
Noting that $R_1, R_2$ and   $R_3$ are bounded functions on
$\bar{I}$ provided that  $f(t,u,u')$ is sufficiently smooth, then
using (\ref{orht-1}), (\ref{L2-projc-appro}) and the $L^2$-stability
of the projection operator $\pi^{r_n-2}$, we have
\begin{eqnarray}\label{conp-3}
\begin{array}{lll}
|\varrho|^2&=\left|\ds\int_{0}^{t_n} (e''-\theta_2
e'-\theta_1e)(w-\Pi_*^{\bf{r-2}}w)dt+\ds\int_{0}^{t_n} (e''-\theta_2
e'-\theta_1e)\Pi_*^{\bf{r-2}}wdt\right|^2 \\[3mm]
&\le 2\left|\ds\int_{0}^{t_n} (e''-\theta_2
e'-\theta_1e)(w-\Pi_*^{\bf{r-2}}w)dt\right|^2+2\left|\ds\int_{0}^{t_n} \left( R_1e^2+R_2ee'+R_3(e')^2\right)\Pi_*^{\bf{r-2}}wdt\right|^2\\[5mm]
&\le C \|e\|_{H^2(0,t_n)}^2\|w-\Pi_*^{\bf{r-2}}w\|_{L^2(0,t_n)}^2 +C
\ds\int_{0}^{t_n} \left(e^4+(e')^4\right) dt
\|\Pi_*^{\bf{r-2}}w\|_{L^2(0,t_n)}^2\\[3mm]
&\le   C \|e\|^2_{H^2(0,t_n)}  \ds\sum_{i=1}^n
  \frac{k_i^{2\min\{r_i-2,s\}+2}}{r_i^{2(s+1)}}
  \|w\|^2_{H^{s+1}(I_i)}+C
\ds\int_{0}^{t_n} \left(e^4+(e')^4\right) dt\|w\|^2_{L^2(0,t_n)},
\end{array}
\end{eqnarray}
where $\Pi_*^{\bf{r-2}}w$ is the piecewise $L^2$-projection of $w$
as defined by (\ref{L2-pir-aa}). Similar to the derivation of
(\ref{L4}) and (\ref{H1-4}), there hold
\begin{equation}\label{L-11}
\ds\int_{0}^{t_n}e^4dt\leq C\|e\|_{H^{1}(0,t_n)}\|e\|^3_{L^2(0,t_n)}
\quad \mbox{and} \quad \ds\int_{0}^{t_n}(e')^4dt
       \leq C\|e\|_{H^{2}(0,t_n)}\|e\|^3_{H^1(0,t_n)}.
\end{equation}
Combining (\ref{conp-3}), (\ref{L-11}) and (\ref{g-prio}), yields
\begin{eqnarray}\label{conp-4}
\begin{aligned}
{\varrho}^2
\le&   C \max\limits_{1\leq i\leq n} \left\{
\ds\frac{k_i^{2\min\{r_i-2,s\}+2}}{r_i^{2(s+1)}}\right\}
\|e\|^2_{H^2(0,t_n)} \|w\|^2_{H^{s+1}(0,t_n)}\\
&+C\|e\|_{H^{2}(0,t_n)}\|e\|^3_{H^1(0,t_n)}\|w\|^2_{L^2(0,t_n)}\\
\le&  C \max\limits_{1\leq i\leq n} \left\{
\ds\frac{k_i^{2\min\{r_i-2,s\}+2}}{r_i^{2(s+1)}} \right\}
  \|e\|^2_{H^2(0,t_n)} \left(|w_0|^2+|w_1|^2\right)\\
  &+C\|e\|_{H^{2}(0,t_n)}\|e\|^3_{H^1(0,t_n)}\left(|w_0|^2+|w_1|^2\right).
\end{aligned}
\end{eqnarray}

We now set $w_0=0$ and $w_1=e(t_n)$. By (\ref{conp-1}) and
(\ref{conp-4}), we have
\begin{eqnarray*}
{\varrho}^2=|e(t_n)|^4\leq  C \left(\max\limits_{1\leq i\leq n}
\left\{ \ds\frac{k_i^{2\min\{r_i-2,s\}+2}}{r_i^{2(s+1)}}
\right\}\|e\|^2_{H^{2}(0,t_n)}+\|e\|_{H^{2}(0,t_n)}\|e\|^3_{H^1(0,t_n)}
\right)  |e(t_n)|^2,
\end{eqnarray*}
which implies that
\begin{eqnarray}\label{conp-6}
\begin{array}{lll}
|e(t_n)|^2 \le C  \max\limits_{1\leq i\leq n} \left\{
\ds\frac{k_i^{2\min\{r_i-2,s\}+2}}{r_i^{2(s+1)}} \right\}
  \|e\|^2_{H^2(0,t_n)}+C\|e\|_{H^{2}(0,t_n)}\|e\|^3_{H^1(0,t_n)}.
\end{array}
\end{eqnarray}
On the other hand, we set $w_0=e'(t_n)$ and $w_1=0$. Similar to
(\ref{conp-6}), using (\ref{conp-1}) and (\ref{conp-4}), we get
\begin{equation}\label{conp-8}
\begin{aligned}
&e'(t_n)-\theta_2(t_n)e(t_n)|^2 \le  C \max\limits_{1\leq i\leq n}
\left\{ \ds\frac{k_i^{2\min\{r_i-2,s\}+2}}{r_i^{2(s+1)}} \right\}
  \|e\|^2_{H^2(0,t_n)}+C\|e\|_{H^{2}(0,t_n)}\|e\|^3_{H^1(0,t_n)}.
\end{aligned}
\end{equation}
Noting that  $|e'(t_n)|^2 \le   2|e'(t_n)-\theta_2(t_n)e(t_n)|^2 +
2|\theta_2(t_n)e(t_n)|^2$,   using   (\ref{conp-6}) and
(\ref{conp-8}), gives
\begin{eqnarray}\label{conp-9}
|e'(t_n)|^2 \le C  \max\limits_{1\leq i\leq n} \left\{
\ds\frac{k_i^{2\min\{r_i-2,s\}+2}}{r_i^{2(s+1)}} \right\}
  \|e\|^2_{H^2(0,t_n)}+C\|e\|_{H^{2}(0,t_n)}\|e\|^3_{H^1(0,t_n)}.
 \end{eqnarray}

Combining (\ref{conp-6}), (\ref{conp-9}),  the $H^1$- and
$H^2$-error estimates in Theorem \ref{hp-error-111}, we obtain

\begin{equation*}\label{super-nodes-111}
\begin{aligned}
|e(t_n)|^2+ |e'(t_n)|^2  \le& C \max\limits_{1\leq i\leq n} \left\{
\ds\frac{k_i^{2\min\{r_i-2,s\}+2}}{r_i^{2(s+1)}} \right\}
\ds\sum_{i=1}^n \Big(\ds\frac{k_i}{2}\Big)^{2\min\{r_i,s\}-2}
\ds\frac{\Gamma(r_i-s+1)}{\Gamma(r_i+s-1)}\|u\|^2_{H^{s+1}(I_i)}\\
&+C\left(\ds\sum_{i=1}^n
\Big(\ds\frac{k_i}{2}\Big)^{2\min\{r_i,s\}-2}
\ds\frac{\Gamma(r_i-s+1)}{\Gamma(r_i+s-1)}\|u\|^2_{H^{s+1}(I_i)}\right)^{\frac{1}{2}}\\
&\times \left(\ds\sum_{i=1}^n
\Big(\ds\frac{k_i}{2}\Big)^{2\min\{r_i,s\}}
\ds\frac{\Gamma(r_i-s+1)}{\Gamma(r_i+s-1)}\ds\frac{1}{r_i(r_i-1)}\|u\|^2_{H^{s+1}(I_i)}\right)^{\frac{3}{2}}\\
\le& C \max\limits_{1\leq i\leq n} \left\{
\ds\frac{k_i^{2\min\{r_i-2,s\}+2}}{r_i^{2(s+1)}} \right\}
\ds\sum_{i=1}^n \Big(\ds\frac{k_i}{2}\Big)^{2\min\{r_i,s\}-2}
\ds\frac{\Gamma(r_i-s+1)}{\Gamma(r_i+s-1)}\|u\|^2_{H^{s+1}(I_i)}\\
&+  C\max\limits_{1\le i\le n} \Big\{
\Big(\ds\frac{k_i}{r_i}\Big)^{3} \Big\} \left(\ds\sum_{i=1}^n
\Big(\ds\frac{k_i}{2}\Big)^{2\min\{r_i,s\}-2}
\ds\frac{\Gamma(r_i-s+1)}{\Gamma(r_i+s-1)}\|u\|^2_{H^{s+1}(I_i)}\right)^2.
\end{aligned}
 \end{equation*}
This proves (\ref{super-nodes}).

Moreover, if $\mathcal T_h$ is a quasi-uniform partition and  $r_n
\equiv r$, applying the Stirling's formula to (\ref{super-nodes})
gives
\begin{equation*}
\begin{aligned}
 |e(t_n)|+ |e'(t_n)|  \le & C
 \ds\frac{k^{\min\{r-2,s\}+1 }}{r^{s+1}} \ds\frac{k^{\min\{r,s\}-1 }}{r^{s-1}}\|u\|_{H^{s+1}(0,
 t_n)} +     C \Big(\ds\frac{k}{r}\Big)^{\frac 32}
 \ds\frac{k^{2\min\{r,s\}-2}}{r^{2s-2}}\|u\|^2_{H^{s+1}(0, t_n)}\\
 \le & C
 \ds\frac{k^{\min\{r-2,s-\frac 12\}+\min\{r,s\}}}{r^{2s-\frac 12}}\|u\|_{H^{s+1}(0,
 t_n)}.
 \end{aligned}
 \end{equation*}
This proves (\ref{cor-err-2}).
\end{proof}

\begin{rem}\label{quasi-super}
We note that,    the    estimate (\ref{cor-err-2}) can be read as
$$|e(t_n)|+ |e'(t_n)| \le C  k^{2r-2},\quad  1\le n\le N$$
for the $h$-version  (if $s\ge r$),  and
$$|e(t_n)|+ |e'(t_n)| \le C  r^{-(2s-\frac 12)}, \quad  1\le n\le N$$ for the
$p$-version as $r \rightarrow \infty$.
\end{rem}

\subsection{$L^{\infty}$-error estimate}

The aim of this section is to show the $L^{\infty}$-estimate. For
this purpose, we first recall  the following technique lemma from
\cite{SS}.

\begin{lemma}\label{inverse}
For any $\varphi \in P_{r_n}(I_n)$, there holds
\begin{eqnarray*}
\| \varphi \|^2_{L^\infty(I_n)}  \le C \left(\log(r_n+1)
\ds\int_{I_n} | \varphi'(t)|^2(t-t_{n-1})dt +  |\varphi(t_n)|^2
\right),
\end{eqnarray*}
 where  the constant $C>0$ is independent of $k_n$ and
$r_n$. Moreover, the estimate cannot be improved asymptotically as
$r_n \rightarrow \infty$.
\end{lemma}

 The main results of this section are stated in the following theorem.

\begin{theorem}\label{hp-error-Linfty}
 Let $\mathcal{T}_h$ be an arbitrary partition of $I$,  $u$ be
the exact solution of (\ref{IVP}) and $U$ be the $C^1$-CPG solution
of (\ref{C1CPG-FEM}). Assume that  $f(t,u,u')$ is sufficiently
smooth with respect to the variables $t, u$ and $u'$ on
$\bar{I}\times \mathbb{R} \times \mathbb{R}$. We further assume that
$u \in W^{s+1,\infty}(I)$  with $s \ge \frac 32$. Then, for $k_n$
sufficiently small and $r_n\ge 3$, we have
\begin{eqnarray}\label{L-infty-main}
\|u-U\|^2_{L^\infty(I)} \le C \ds\frac{k^4
\log(\overline{r})}{\underline{r}^3}\ds\max_{1\le n\le N} \left\{
\ds\frac{k_n^{2\min\{r_n,s\}-2}}{r_n^{2s-2}}
  \|u\|^2_{W^{s+1,\infty}(I_n)} \right\},
\end{eqnarray}
where $\overline{r}=\max\limits_{1\le n \le N} \{r_n\}$,
$\underline{r}=\min\limits_{1\le n \le N} \{r_n\}$,
  and the constant $C>0$ is
independent of $k_n$ and $r_n$.

In particular,  if $\mathcal T_h$ is a quasi-uniform partition and
$r_n\equiv r$, then
\begin{eqnarray}\label{L-infty-cor}
 \|u-U\|_{L^\infty(I)} \le C (\log r)^{\frac 12}\ds\frac{k^{\min\{r,s\}+1}}{r^{s+1/2}}\|u\|_{W^{s+1,\infty}(I)},
 \end{eqnarray}
 where the   constant $C > 0$ is independent of $k$ and $r$.
 \end{theorem}
\begin{proof}
Recalling that $e=\eta + \xi$ with   $\eta =u-\Pi^{\bf{r}}u$ and
$\xi =\Pi^{\bf{r}}u-U$,  selecting $\varphi=\pi^{r_n-2}
((t_{n-1}-t)\xi)$ in (\ref{xi-1}), then using  (\ref{L2-projc}), (\ref{Lip-con})  and the $L^2$-stability of the $L^2$-projection
operator $\pi^{r_n-2}$, gives
\begin{eqnarray}\label{Linfty-1}
\begin{aligned}
\ds\int_{I_n} (t_{n-1}-t)\xi\xi''   dt = & \ds\int_{I_n}\left(f(t,
u, u')-f(t,
U, U')\right) \pi^{r_n-2} ((t_{n-1}-t)\xi) dt\\
\le& L\left(\ds\int_{I_n}
\left(|u-U|+|u'-U'|\right)^2dt\right)^{\frac{1}{2}}
\left(\ds\int_{I_n} \left|\pi^{r_n-2} ((t_{n-1}-t)\xi)\right|^2dt\right)^{\frac{1}{2}}\\
\le &\sqrt{2}L\|e\|_{H^1(I_n)}\|\pi^{r_n-2}
((t_{n-1}-t)\xi)\|_{L^2(I_n)}\\
\le & \sqrt{2}L\|e\|_{H^1(I_n)}\|(t_{n-1}-t)\xi\|_{L^2(I_n)}\\
\le &\sqrt{2}Lk_n\|e\|_{H^1(I_n)}\|\xi\|_{L^2(I_n)}.
\end{aligned}
\end{eqnarray}
On the other hand, by integration by parts  we have
\begin{eqnarray*}
\begin{aligned}
\ds\int_{I_n} (t_{n-1}-t)\xi\xi''   dt = & - k_n\xi(t_n)\xi'(t_n) +
\ds\int_{I_n}\xi'\xi dt -\ds\int_{I_n} (t_{n-1}-t)|\xi'|^2 dt\\
=& - k_n\xi(t_n)\xi'(t_n)
+\ds\frac{1}{2}|\xi(t_n)|^2-\ds\frac{1}{2}|\xi(t_{n-1})|^2
+\ds\int_{I_n} (t-t_{n-1})|\xi'|^2 dt,
\end{aligned}
\end{eqnarray*}
which together with (\ref{Linfty-1})  and the inequality
$k_n|\xi(t_n) \xi'(t_n)|\le
\frac{1}{2}k_n^2|\xi'(t_n)|^2+\frac{1}{2}|\xi(t_n)|^2$ leads to
\begin{eqnarray}  \label{Linfty-2}
\begin{aligned}
\ds\int_{I_n} (t-t_{n-1})|\xi'|^2   dt
\leq&k_n|\xi(t_n) \xi'(t_n)|-\ds\frac{1}{2}|\xi(t_n)|^2+\ds\frac{1}{2}|\xi(t_{n-1})|^2\\
&+\sqrt{2}Lk_n\|e\|_{H^1(I_n)}\| \xi\|_{L^2(I_n)}\\
\leq&\ds\frac{1}{2}k_n^2|\xi'(t_n)|^2+\ds\frac{1}{2}|\xi(t_{n-1})|^2\\
&+
\sqrt{2}Lk_n\|e\|_{H^1(I_n)}\left(\|e\|_{L^2(I_n)}+\|\eta\|_{L^2(I_n)}\right),
\end{aligned}
\end{eqnarray}
where we have used the fact
$\|\xi\|_{L^2(I_n)}=\|e-\eta\|_{L^2(I_n)}\le
\|e\|_{L^2(I_n)}+\|\eta\|_{L^2(I_n)}$.

Due to (\ref{C1-con}), we find that
\begin{eqnarray}  \label{Linfty-3}
\xi(t_n)=e(t_n)-\eta(t_n)=e(t_n) \quad \mbox{and}  \quad
\xi'(t_n)=e'(t_n)-\eta'(t_n)=e'(t_n).
\end{eqnarray}
Thanks to Lemma \ref{inverse}, by (\ref{Linfty-2}) and
(\ref{Linfty-3}), we have
\begin{equation}  \label{Linfty-4}
\begin{aligned}
\|\xi\|^2_{L^{\infty}(I_n)} \le & C
\log(r_n+1)\left(k_n^2|\xi'(t_n)|^2+|\xi(t_{n-1})|^2\right)+C|\xi(t_n)|^2\\
&+ C\log(r_n+1)k_n\|e\|_{H^1(I_n)}\left( \|e\|_{L^2(I)} + \|\eta\|_{L^2(I)} \right)   \\
\le& C \log(r_n+1) \max_{1\le n\le
N}\left\{|e'(t_n)|^2+|e(t_{n})|^2\right\}\\
&+ Ck_n\log(r_n+1)
\|e\|_{H^1(I)}\left( \|e\|_{L^2(I)} + \|\eta\|_{L^2(I)} \right).
\end{aligned}
\end{equation}

For convenience, we set $G:=\ds\sum_{n=1}^N
\Big(\ds\frac{k_n}{2}\Big)^{2\min\{r_n,s\}-2}
\ds\frac{\Gamma(r_n-s+1)}{\Gamma(r_n+s-1)}\|u\|^2_{H^{s+1}(I_n)}.$
Using the Stirling's formula, we have
\begin{equation}  \label{G-est}
\begin{aligned}
G\le & \ds C\sum_{n=1}^N
\ds\frac{k_n^{2\min\{r_n,s\}-2}}{r_n^{2s-2}}
  \|u\|^2_{H^{s+1}(I_n)}  \le C\ds\sum_{n=1}^N \ds\frac{k_n^{2\min\{r_n,s\}-1}}{r_n^{2s-2}}
  \|u\|^2_{W^{s+1,\infty}(I_n)}\\ \le & C\max_{1\le n\le N} \left\{ \ds\frac{k_n^{2\min\{r_n,s\}-2}}{r_n^{2s-2}}
  \|u\|^2_{W^{s+1,\infty}(I_n)} \right\}.
  \end{aligned}
\end{equation}
Inserting (\ref{super-nodes}), (\ref{Iu-L2}), (\ref{cor-err-H1-1})
and (\ref{L2-error}) into (\ref{Linfty-4}),  then using
(\ref{G-est}), gives
\begin{equation}
  \begin{aligned}\label{Linfty-5}
\max\limits_{1\le n\le N}\left\{\|\xi\|^2_{L^{\infty}(I_n)}\right\}
\le& C  \log(\overline{r}+1) \left( \max\limits_{1\leq n\leq N}
\left\{ \ds\frac{k_n^{2\min\{r_n-2,s\}+2}}{r_n^{2(s+1)}} \right\} G+
 \max\limits_{1\le n\le N} \Big\{
\Big(\ds\frac{k_n}{r_n}\Big)^{3} \Big\} G^2 \right)  \\
&+Ck\log(\overline{r}+1)\max\limits_{1\le n\le N} \Big\{
\Big(\ds\frac{k_n}{r_n}\Big)^{2} \Big\}
G^{1/2}\times\left\{\max\limits_{1\le n\le N} \Big\{
\Big(\ds\frac{k_n}{r_n}\Big)^{2} \Big\} G^{1/2}
 \right.\\
&+\left. \max\limits_{1\le n\le N} \Big\{
\Big(\ds\frac{k_n}{r_n}\Big)^{\frac 32} \Big\}  G
 + \max\limits_{1\le n\le N} \Big\{
\Big(\ds\frac{k_n}{r_n}\Big)^{2} \Big\} G^{1/2}
\right\} \\
\le & C  \log(\overline{r}+1) \left(
\ds\frac{k^{2\min\{r_n-2,s\}+2}}{{\underline{r}}^{2(s+1)}} +
 \ds\frac{k^3}{\underline{r}^3}   G  +\ds\frac{k^5}{\underline{r}^4}   + \ds\frac{k^{9/2}}{\underline{r}^{7/2}}
G^{1/2}\right) G\\
\le  & C  \frac{k^4 \log(\overline{r})}{\underline{r}^4}\max_{1\le
n\le N} \left\{ \ds\frac{k_n^{2\min\{r_n,s\}-2}}{r_n^{2s-2}}
  \|u\|^2_{W^{s+1,\infty}(I_n)} \right\}
\end{aligned}
\end{equation}
for $s\ge \frac 32$,
where $\overline{r}=\max\limits_{1\le n \le N} \{r_n\}$ and
$\underline{r}=\min\limits_{1\le n \le N} \{r_n\}$.  Moreover, using
(\ref{Iu-infty-1}) and the Stirling's formula,  we have
\begin{equation}  \label{Linfty-6}
\begin{aligned}
\max\limits_{1\le n\le N}\left\{\|\eta\|^2_{L^{\infty}(I_n)}\right\}
\le & C  \max\limits_{1\le n\le N}
\left\{\Big(\frac{k_n}{2}\Big)^{2\min\{r_n,s\}+2} \ds\frac{
\Gamma(r_n-s+1)}{\Gamma(r_n+s-1)}\ds\frac{1}{(r_n-2)^3}\|u\|^2_{W^{s+1,\infty}(I_n)}
\right\}\\ \le & C \frac{k^4}{\underline{r}^3} \max\limits_{1\le
n\le N} \left\{ \ds\frac{k_n^{2\min\{r_n,s\}-2}}{r_n^{2s-2}}
\|u\|^2_{W^{s+1,\infty}(I_n)}\right\}.
\end{aligned}
\end{equation}
Combining (\ref{Linfty-5}) and (\ref{Linfty-6}),  we get
\begin{eqnarray*}\label{Linfty-7}
\begin{aligned}
\|u-U\|^2_{L^{\infty}(I)}\le&
2\|\eta\|^2_{L^{\infty}(I)}+2\|\xi\|^2_{L^{\infty}(I)} \le
C\frac{k^4 \log(\overline{r})}{\underline{r}^3}\max_{1\le n\le N}
\left\{ \ds\frac{k_n^{2\min\{r_n,s\}-2}}{r_n^{2s-2}}
  \|u\|^2_{W^{s+1,\infty}(I_n)} \right\}.
 \end{aligned}
\end{eqnarray*}
This proves (\ref{L-infty-main}). Furthermore, using the
quasi-uniformity of the time partition and the assumption $r_n\equiv
r$,
  we immediately get (\ref{L-infty-cor}) from (\ref{L-infty-main}).
\end{proof}

\section{Application to nonlinear  wave  equations}\label{sec4}

In this section,  we apply the method presented in Section
\ref{sec2} to nonlinear wave equations.  More precisely,
  we shall use the $hp$-version  $C^1$-CPG   time stepping method   to
handle the time integration of the second-order nonlinear
differential system arising after space discretization obtained with
the usual spectral Galerkin or conforming finite element Galerkin
method.

\subsection{Model problem}

Let $I=(0,T)$ be a finite time interval and  $\Omega\in
\mathbb{R}^d$ ($d=1, 2, 3$) be   an open and bounded domain with
  boundary $\partial\Omega$. We consider the nonlinear
wave equation
\begin{equation}\label{wave-eqs}
\left\{
\begin{aligned}
\pd_{tt}u-\nabla \cdot (b \nabla u) & =  f(u)     && \mbox{in}~~\Omega \times I,\\
u & =  0     \quad &&\mbox{on}~~\partial\Omega \times I,\\
 u(\cdot,0) & =  u_0     \quad && \mbox{in}~~\Omega,\\
 \pd_tu(\cdot,0) & =  u_1   \quad  &&\mbox{in}~~\Omega,
\end{aligned}
\right.
\end{equation}
where   $u_0\in H_0^1(\Omega)$ and $u_1\in L^2(\Omega)$ are
prescribed initial conditions,  $f(u)=f(\mathbf{x},t,u)$ is a given
function that depending on the unknown function $u(\mathbf{x},t)$
 with $\mathbf{x}=(x_1,\cdots,x_d) \in
\mathbb{R}^d$.
 Moreover, we assume that $b$ is a piecewise smooth
function and there exist two  positive constants $b_*$ and $b^*$
such that
$$0< b_* \le b(\mathbf{x}) \le b^* < \infty, \quad \forall\mathbf{x}\in \overline{\Omega}.$$

The   weak formulation of the problem  (\ref{wave-eqs}) is to find
$u\in L^2(I; H_0^1(\Omega))$  with $\pd_tu \in L^2(I; L^2(\Omega))$
and $\pd_{tt}u \in L^2(I; H^{-1}(\Omega))$, such that $u(\cdot,0) =
u_0$,  $\pd_tu(\cdot,0) = u_1$, and
\begin{equation}\label{wave-weak}
 \langle \pd_{tt}u, \varphi \rangle + (b \nabla u, \nabla \varphi)=(f(u),
\varphi),\quad \forall \varphi\in H_0^1(\Omega)\quad  \mbox{a.e.}~~
\mbox{in}~~I,
\end{equation}
where $ \langle \cdot,\cdot\rangle $ denotes the duality pairing
between $H^{-1}(\Omega)$ and  $H_0^1(\Omega)$,  $(\cdot,\cdot)$
denotes the inner product in  $L^2(\Omega)$, and  $L^2(I; V)$
  denotes the Bochner space of $V$-valued  functions with $V=H_0^1(\Omega), L^2(\Omega)$ and
  $H^{-1}(\Omega)$, respectively.

\subsection{Galerkin semi-discretization in
space}\label{subsec-4-2}

We shall use the usual spectral Galerkin or  conforming finite
element Galerkin method to discrete the problem (\ref{wave-eqs}) in
space. For details about the theory of the spectral Galerkin and
finite element methods, we refer to \cite{STW,Bre-Sco,Ciarlet78}, respectively.

Let $V_h \subset H_0^1(\Omega)$ be an  approximation  space used for
the spectral Galerkin or  conforming finite element Galerkin method.
We consider the semi-discretized Galerkin approximation of
(\ref{wave-eqs}): find $u_h: V_h \times \bar{I} \rightarrow
\mathbb{R}$ such that
\begin{equation}\label{parabolic-semi}
(\partial_{tt} u_h, \varphi_h)+ (b \nabla u_h, \nabla
\varphi_h)=(f(u_h), \varphi_h),\quad \forall \varphi_h\in V_h,
~~t\in I,
\end{equation}
$u_h(\cdot,0) = u_{0,h}$  and  $\partial_t u_h(\cdot,0) = u_{1,h}$,
where $u_{0,h} \in V_h $ and $ u_{1,h}\in V_h$ are  suitable
approximations (such as the Lagrange interpolation  or
$L^2$-projection) of the initial conditions $u_0$ and $u_1$.

Suppose that $ \{\varphi_i\}_{i=1}^M$ is a basis of the finite
dimensional subspace $V_h$. We can expand the semi-discrete solution
$u_h$ as  $$ u_h=\sum_{i=1}^M \alpha_i(t)\varphi_i(\mathbf{x}).$$
Inserting the above expansion into (\ref{parabolic-semi}) leads to
the following nonlinear second-order ODE system:
  find  $\alpha_i(t),~ 1\le i\le M$,  such that
\begin{equation}\label{parabolic-semi-1}
\ds\sum_{i=1}^M \alpha_i''(t)(\varphi_i, \varphi_j)+\sum_{i=1}^M
\alpha_i(t) (b \nabla \varphi_i, \nabla \varphi_j)
=\left(f\left(\sum_{i=1}^M \alpha_i(t)\varphi_i\right),
\varphi_j\right),\quad j=1,2,\cdots, M,
\end{equation}
 $\alpha_i(0)=\hat u_{0,i}$ and  $\alpha_i'(0)=\hat u_{1,i}$, where $\hat u_{0,i}$ and  $\hat u_{1,i}$ are expansion coefficients of the given initial
approximations $u_{0,h}=\sum\limits_{i=1}^M \hat
u_{0,i}\varphi_i(\mathbf{x})$ and $u_{1,h}=\sum\limits_{i=1}^M \hat
u_{1,i}\varphi_i(\mathbf{x})$.

For convenience, we set $\boldsymbol{\alpha}(t)=\left(\alpha_1(t),
\alpha_2(t), \cdots, \alpha_M(t)\right)^T$,
$\boldsymbol{\alpha_0}=\left(\hat u_{0,1}, \hat u_{0,2}, \cdots,
\hat u_{0,M}\right)^T$ and $\boldsymbol{\alpha_1}=\left(\hat
u_{1,1}, \hat u_{1,2}, \cdots, \hat u_{1,M}\right)^T$. Then, the
nonlinear system (\ref{parabolic-semi-1}) can be expressed as
\begin{equation}\label{IVP-system1}
\left\{
\begin{array}{l}
\mathcal{B} \boldsymbol{\alpha}''(t)+ \mathcal{D}
\boldsymbol{\alpha}(t)= \boldsymbol{\mathcal{F}}(\boldsymbol{\alpha}(t)) ,\quad t\in
I,\\[2mm]
\boldsymbol{\alpha}(0)=\boldsymbol{\alpha_0}, \quad
\boldsymbol{\alpha}'(0)=\boldsymbol{\alpha_1},
\end{array}
\right.
\end{equation}
where $\mathcal{B}=(b_{ij})_{i,j=1}^M$ is the mass matrix with
entries $b_{ij}=(\varphi_i, \varphi_j)$,
$\mathcal{D}=(d_{ij})_{i,j=1}^M$ is the stiffness matrix with
entries $d_{ij}= (b\nabla \varphi_i, \nabla \varphi_j)$, and
$\boldsymbol{\mathcal{F}}(\boldsymbol{\alpha}(t))=(f_j(\boldsymbol{\alpha}(t)))_{j=1}^M$
is a vector function with
$f_j(\boldsymbol{\alpha}(t))=\big(f(\sum_{i=1}^M
\alpha_i(t)\varphi_i), \varphi_j\big)$.

Since the mass matrix $\mathcal{B}$ is  positive definite and
invertible, we can further rewrite the   system (\ref{IVP-system1})
as the following nonlinear second-order ODE system
\begin{equation}\label{IVP-system2}
\left\{
\begin{array}{l}
 \boldsymbol{\alpha}''(t)=\boldsymbol{\widetilde{\mathcal{F}}}(t,\boldsymbol{\alpha}(t)),\quad t\in
 I,\\[2mm]
\boldsymbol{\alpha}(0)=\boldsymbol{\alpha_0}, \quad
\boldsymbol{\alpha}'(0)=\boldsymbol{\alpha_1},
\end{array}
\right.
\end{equation}
where
$\boldsymbol{\widetilde{\mathcal{F}}}(t,\boldsymbol{\alpha}(t)) := -
\mathcal{B}^{-1} \mathcal{D}\boldsymbol{\alpha}(t)+
\mathcal{B}^{-1}\boldsymbol{\mathcal{F}}(\boldsymbol{\alpha}(t))$.

\subsection{Fully discrete scheme}\label{sec4-3}

To discretize (\ref{IVP-system2}) in time, we employ the $hp$-version
of the $C^1$-CPG  time stepping method  as  introduced in Section
\ref{sec2}.

Given an arbitrary partition of   $[0,T]$ with subintervals
$\{I_n=(t_{n-1},t_n)\}_{n=1}^N$. We denote by $P_{r_n}(I_n,
\mathbb{R}^M)$  the set of all polynomials of degree at most $r_n$
on $I_n$ with coefficients in $\mathbb{R}^M$.

The $hp$-version of the $C^1$-CPG time stepping method for
(\ref{IVP-system2}) can be read as: if $\boldsymbol{\alpha}_{hp}$ is
given on the time intervals $I_m,\ 1\leq m\leq n-1$, we can find
$\boldsymbol{\alpha}_{hp}|_{I_n} \in P_{r_n}(I_n, \mathbb{R}^M)$
with $r_n\ge 2$ on the next time step $I_n$ by solving
\begin{equation}\label{ODE-sys-CPG}
\left\{
\begin{aligned}
&\int_{I_n} \left(\boldsymbol{\alpha}_{hp}'', \varphi\right)
dt=\int_{I_n}
\left(\boldsymbol{\widetilde{\mathcal{F}}}(t,\boldsymbol{\alpha}_{hp}),
\varphi\right) dt, \quad \forall \varphi \in P_{r_n-2}(I_n,
\mathbb{R}^M),\\[2mm]
&\boldsymbol{\alpha}_{hp}|_{I_n}(t_{n-1})=\boldsymbol{\alpha}_{hp}|_{I_{n-1}}(t_{n-1}),
\quad
\boldsymbol{\alpha}_{hp}'|_{I_n}(t_{n-1})=\boldsymbol{\alpha}_{hp}'|_{I_{n-1}}(t_{n-1})
\end{aligned}
\right.
\end{equation}
  Here, we denote by $(\cdot,\cdot)$  the standard  Euclidean inner
  product in $\mathbb{R}^M$, and  we set
  $\boldsymbol{\alpha}_{hp}|_{I_1}(t_0)=\boldsymbol{\alpha_0}$ and
  $\boldsymbol{\alpha}_{hp}'|_{I_1}(t_0)=\boldsymbol{\alpha_1}$.

Suppose that we have  obtained $\boldsymbol{\alpha}_{hp}$ from
(\ref{ODE-sys-CPG}). Then, the  fully
 discrete Galerkin approximation (denoted by $u_{h\tau}$)  for
(\ref{wave-eqs}) can be  expressed as
$$u_{h\tau}=\boldsymbol{\phi}\boldsymbol{\alpha}_{hp},$$ where
$\boldsymbol{ \phi}:=(\varphi_1, \varphi_2, \cdots, \varphi_M)$  is
the basis of the  space $V_h$.

\section{Numerical experiments}\label{sec5}

In this section, we present some numerical results to highlight the
performance of the $hp$-version $C^1$-CPG method. Throughout this
section, we use  uniform time partition (with uniform step-size $k$)
associated with uniform approximation degree $r$ for the $h$- and
$p$-versions of the $C^1$-CPG methods. Moreover, we
employ the simple fixed point iteration method to solve the
nonlinear system of the form (\ref{nonlinear-sys}) very accurately.

\subsection{Example 1: a nonlinear scalar problem}

We consider the nonlinear second-order IVP:
\begin{equation}\label{ex1}
 \left\{\begin{array}{ll} u''(t)=\sin(u(t))-2\cos(u'(t))+g(t),\quad & t\in [0,1],
\\[5pt]u(0)=0,\quad u'(0)=1,  \end{array}\right.
\end{equation}
where $g(t)$ is chosen such that the  exact solution $u=\sin t$.
Clearly, $u$ is analytic in $[0,1]$.

\begin{figure}[h!]
\begin{minipage}[h]{0.48\linewidth}
\centering
  \includegraphics [height=2.2in]{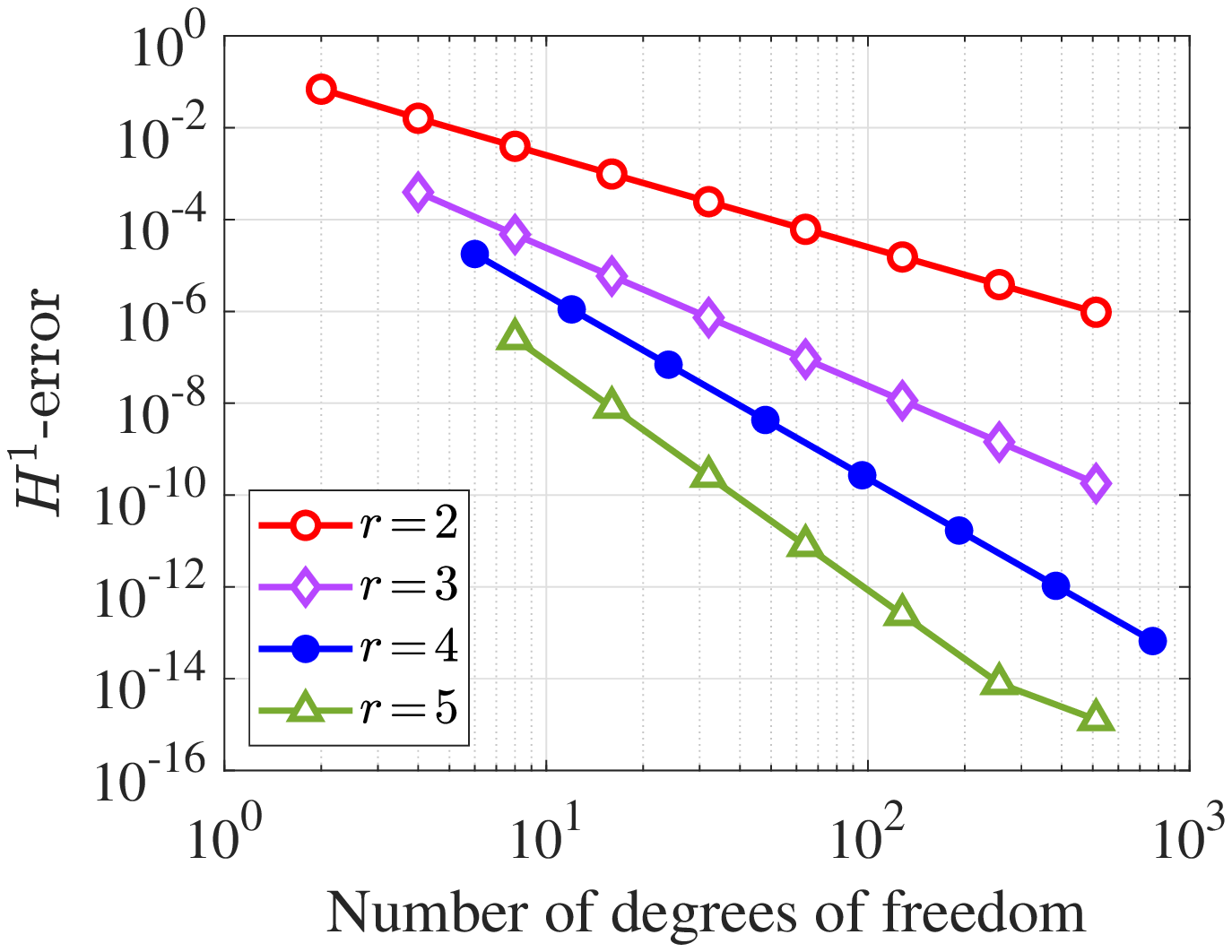}
    \caption{Example 1: $H^1$-errors of the $h$-version.}
    \label{ex1-H1-h-version}
\end{minipage}
  \hfill\quad
\begin{minipage}[h]{0.48\linewidth}
\centering
   \includegraphics [height=2.2in]{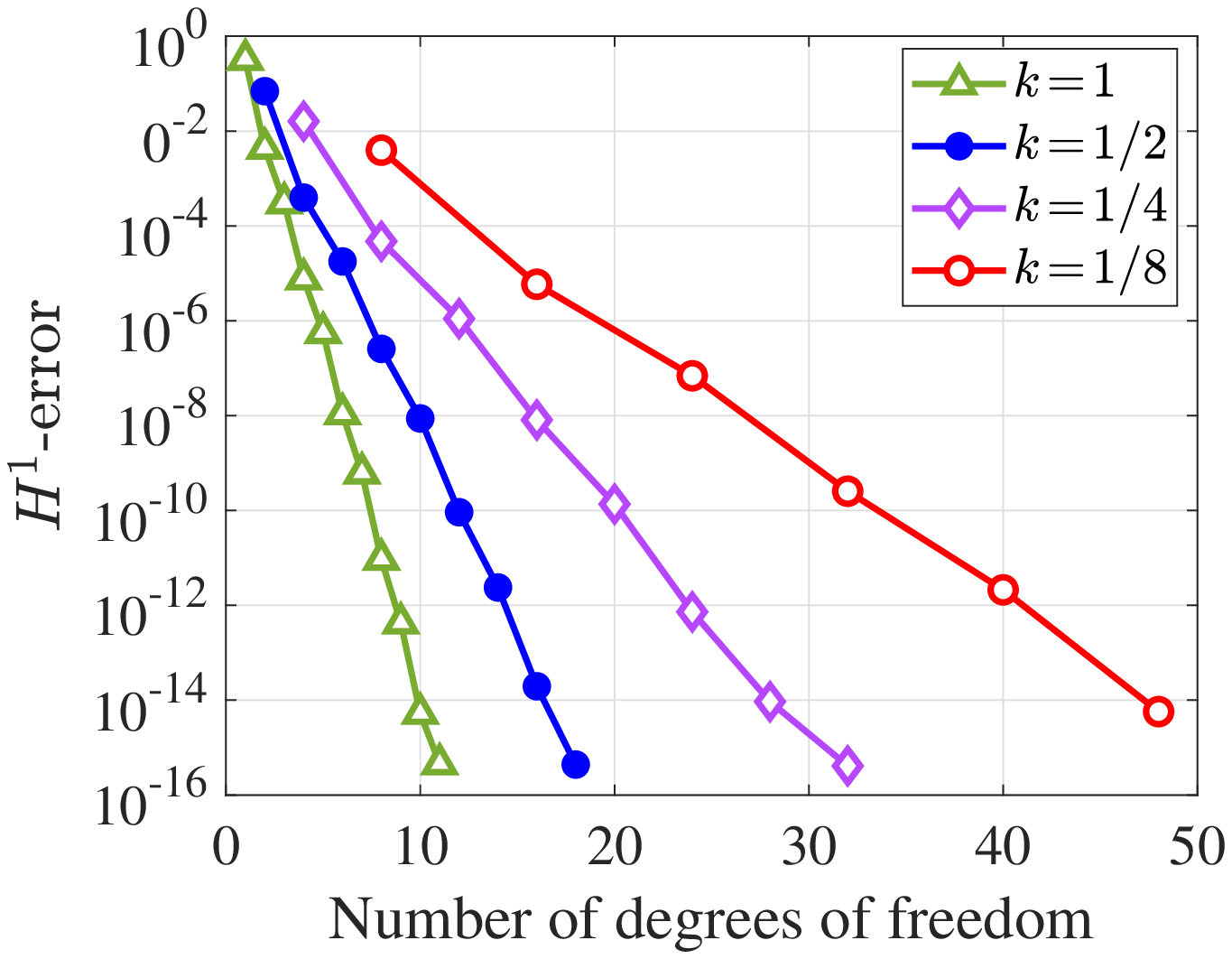}
   \caption{Example 1: $H^1$-errors of the $p$-version.}
   \label{ex1-H1-p-version}
\end{minipage}
\end{figure}

We first consider  the performance of the  $h$- and $p$-versions of
the $C^1$-CPG method for problem (\ref{ex1}), respectively. We use
uniform time partitions with step-size $k$ and uniform approximation
degrees $r$. In Figure \ref{ex1-H1-h-version}, we plot the
$H^1$-errors  against the total number of degrees of freedom (DOF)
in a log-log scale for different $r$.   We observe that the
convergence is algebraic and in accordance with the theoretical
result, i.e., of order $r$ for the $H^1$-errors.  Figure
\ref{ex1-H1-p-version} displays the $H^1$-errors  of the $p$-version
(on fixed   time partitions with $1, 2, 4, 8$ uniform time steps).
It can be seen that exponential convergence is achieved for each
partition as $r$ increases. In particular, we note that the global
$H^1$-error of $10^{-15}$ can be achieved with less than $12$ DOF
for the $p$-version, while this is not possible for the $h$-version
as shown in Figure \ref{ex1-H1-h-version}. Therefore,   for smooth
solutions it is more advantageous to use the $p$-version (i.e.,
increase $r$) rather than to use the $h$-version (i.e., reduce $k$
at fixed low $r$).

\begin{table}[htp]
\tabcolsep 3pt \caption{Example 1: numerical errors and convergence
orders of the $h$-version.}\label{table1}
\begin{center}
\begin{tabular}{|c|c|cc|cc|cc|cc|cc|}
\hline
  $~~r~~$ &   $~~k~~$&$\|e\|_{L^2(I)}$& order & $\|e\|_{H^1(I)}$& order &$\|e\|_{H^2(I)}$&order     & $\|e\|_{L^\infty(I)}$& order &$\|e'\|_{L^\infty(I)}$&order   \\ \hline
\multirow{3}{*}{2}&1/64   &2.41e-05 &2.00 &6.14e-05   &2.00   &3.85e-03&1.00    &5.10e-05 &2.00     &1.02e-04&1.99 \\
                  &1/128  &6.02e-06 &2.00 &1.53e-05   &2.00   &1.92e-03&1.00  &1.28e-05 &2.00     &2.56e-05&1.99\\
                  &1/256  &1.50e-06 &2.00 &3.83e-06   &2.00   &9.62e-04&1.00   &3.19e-06 &2.00     &6.42e-06&2.00\\ \hline
\multirow{3}{*}{3}&1/32   &1.63e-09 &4.00 &9.16e-08   &3.00   &1.90e-05&2.00   &4.20e-09 &3.96     &2.09e-07&3.02 \\
                  &1/64   &1.02e-10 &4.00 &1.15e-08   &3.00   &4.75e-06&2.00    &2.66e-10 &3.98     &2.59e-08&3.01\\
                  &1/128  &6.37e-12 &4.00 &1.43e-09   &3.00   &1.19e-06&2.00    &1.67e-11 &3.99     &3.22e-09&3.01\\ \hline
\multirow{3}{*}{4}&1/16   &4.08e-11 &5.01 &4.32e-09   &4.00   &6.56e-07&3.00  &7.14e-11 &5.03     &7.94e-09&4.00 \\
                  &1/32   &1.27e-12 &5.00 &2.70e-10   &4.00   &8.20e-08&3.00   &2.21e-12 &5.01     &4.97e-10&4.00 \\
                  &1/64   &3.98e-14 &5.00 &1.69e-11   &4.00   &1.02e-08&3.00   &6.88e-14 &5.01     &3.10e-11&4.00 \\ \hline
\multirow{3}{*}{5}&1/8    &3.41e-12 &6.00 &2.54e-10   &5.00   &2.53e-08&4.00    &9.50e-12 &5.91     &5.78e-10&4.93\\
                  &1/16   &5.34e-14 &6.00 &7.96e-12   &5.00   &1.58e-09&4.00     &1.52e-13 &5.96     &1.85e-11&4.97\\
                  &1/32   &9.28e-16 &5.85 &2.49e-13   &5.00   &9.88e-11&4.00    &2.33e-15 &6.03     &5.85e-13&4.98\\ \hline
\end{tabular}
\end{center}
\end{table}

In Table \ref{table1}, we also list the numerical errors (in
different norms) and convergence orders  of the $h$-version
$C^1$-CPG method. The results indicate the convergence orders
  $$ \|e\|_{H^1(I)}=O (k^{r}), \quad  \|e\|_{H^2(I)}=O (k^{r-1}), \quad  r\ge 2,$$
and
\begin{equation*}
\|e\|_{L^2(I)} = \left\{\
\begin{aligned}
&O(k^{r}) \quad\quad \mbox{if} ~r=2,\\
 & O (k^{r+1})  \quad \mbox{if}
~r\ge 3,
\end{aligned}
\right.\\
\quad\quad \|e\|_{L^\infty(I)} = \left\{\
\begin{aligned}
&O(k^{r}) \quad\quad \mbox{if} ~r=2,\\
 & O (k^{r+1})  \quad \mbox{if}
~r\ge 3,
\end{aligned}
\right.\\
\end{equation*}
which confirm the theoretical results well for $r\ge 3$. We note
that,  for $r=2$ the convergence orders of the $L^2$- and
$L^\infty$-errors are only $O(k^2)$. Hence, it seems that $r=2$ is
not a good choice for the $h$-version if we are interested in the
$L^2$- and $L^\infty$-errors.

\begin{figure}[h!]
\begin{minipage}[h]{0.48\linewidth}
\centering
  \includegraphics [height=2.2in]{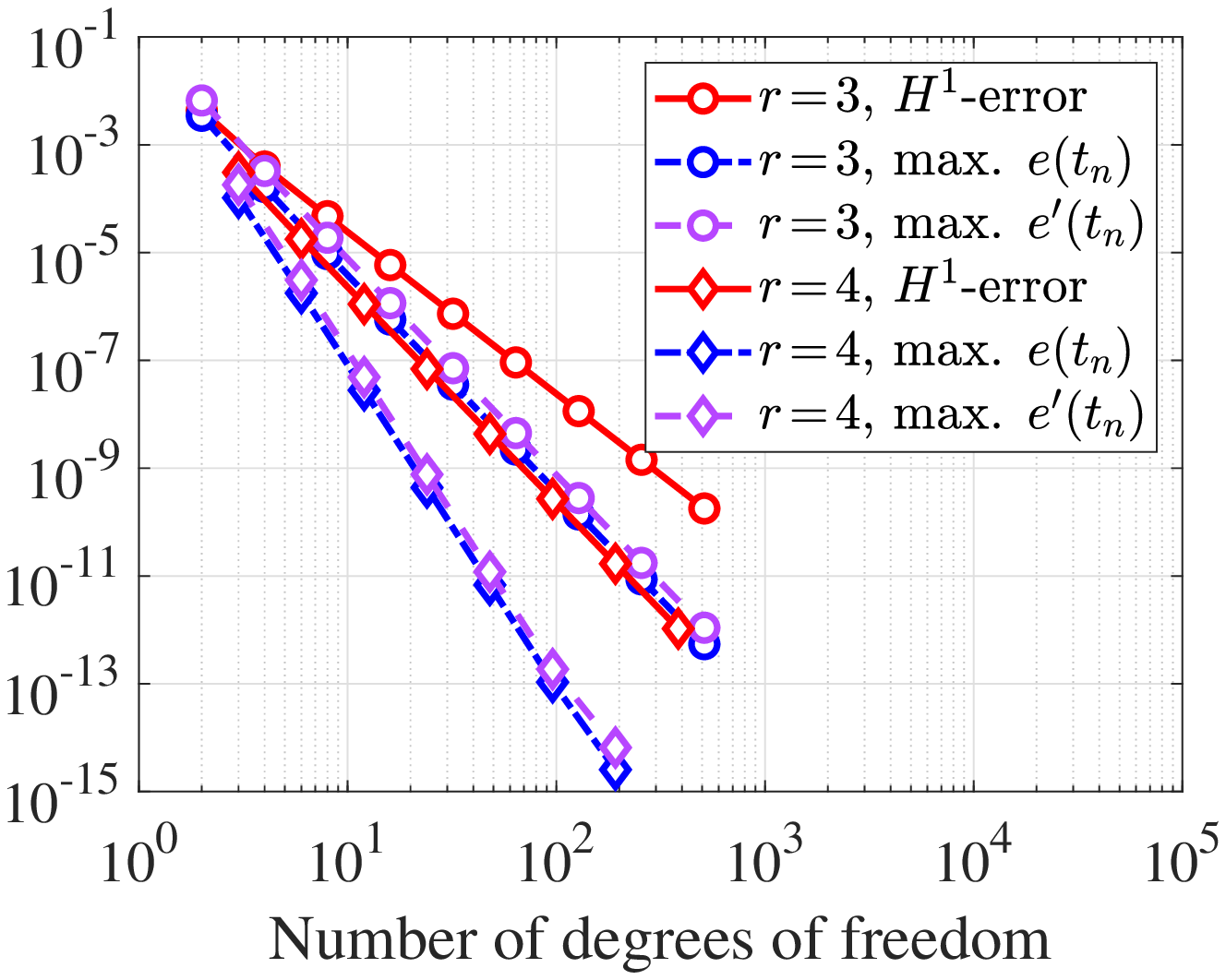}
    \caption{Example 1:  maximum function and derivative approximation  errors  at nodes versus $H^1$-errors of the $h$-version.}
    \label{ex1-super-error-h-version}
\end{minipage}
  \hfill\quad
\begin{minipage}[h]{0.48\linewidth}
\centering
   \includegraphics [height=2.2in]{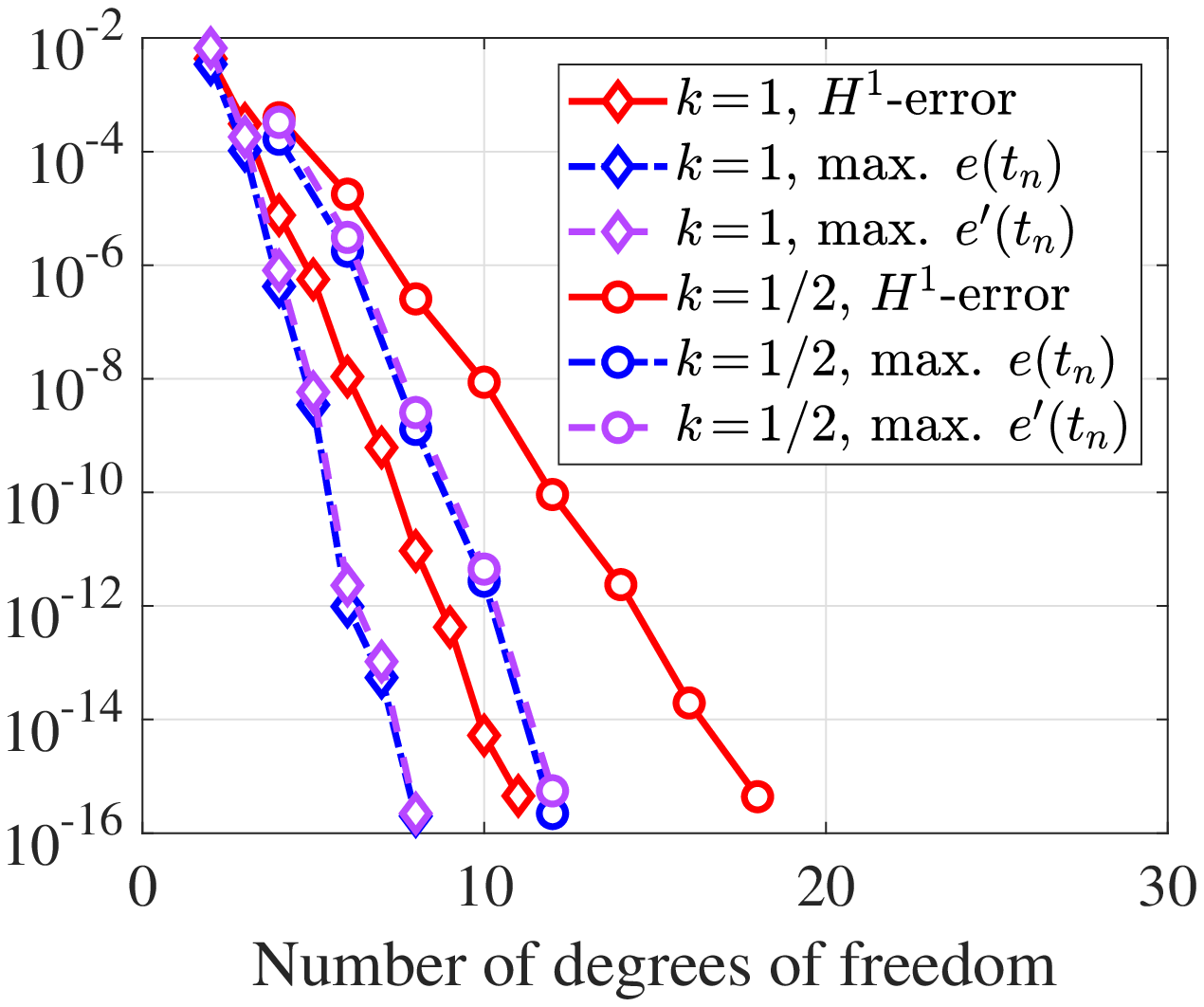}
   \caption{Example 1:  maximum function and derivative approximation errors at nodes versus $H^1$-errors of the $p$-version.}
   \label{ex1-super-error-p-version}
\end{minipage}
\end{figure}

We next consider the performance of the $h$- and $p$-versions of the
$C^1$-CPG method at the nodal points, respectively. We denote by
 $\mbox{max.}  ~e(t_n)$ and
 $\mbox{max.}  ~e'(t_n)$ the maximum function value and derivative approximation
absolute errors at nodal points $\{t_n\}_{n=1}^N$ of the time
partition.  Figures \ref{ex1-super-error-h-version} and
\ref{ex1-super-error-p-version} show that  both the $h$-version and
$p$-version exhibit superconvergence at the nodal points, where the
slopes of the curves of  nodal errors $\mbox{max.} e(t_n)$ and
$\mbox{max.} e'(t_n)$ are approximately twice as steep as those of
the $H^1$-errors.

Finally, we make a simple comparison between  the $hp$-version
$C^1$-CPG method   and the  $hp$-version $C^0$-CPG method developed
in \cite{WY} for second-order   IVPs. In \cite{WY},  the trial
spaces of the $C^0$-CPG method   consist of globally
$C^0$-continuous and piecewise polynomials while the test spaces
consist of discontinuous and piecewise polynomials. In Figure
\ref{ex1-compare-h-version}  we plot the $H^1$-errors of the
$h$-version $C^1$-CPG and   $C^0$-CPG methods, while in  Figure
\ref{ex1-compare-p-version}  we plot the $H^1$-errors for the
$p$-version $C^1$-CPG and $C^0$-CPG methods. Clearly, both the
$C^1$-CPG and the $C^0$-CPG methods exhibit the same convergence
rates, i.e., their error curves have almost the same slopes.
However, it can be seen that  the $C^1$-CPG method is more accurate
than the $C^0$-CPG method if the same number of DOF was used.

\begin{figure}[h!]
\begin{minipage}[h]{0.48\linewidth}
\centering
  \includegraphics [height=2.2in]{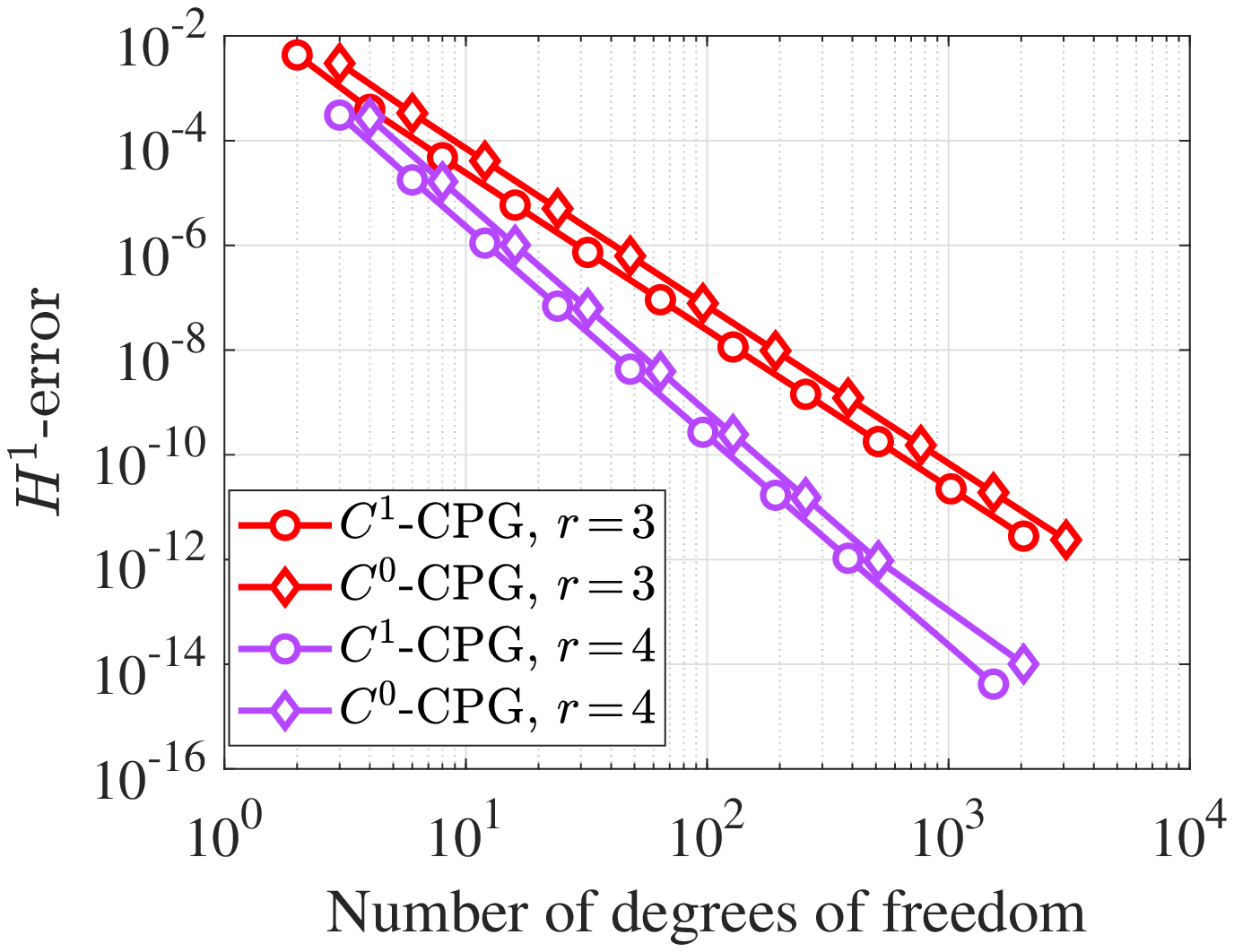}
    \caption{Example 1:   $C^1$-CPG method versus   $C^0$-CPG method, $H^1$-errors of the $h$-version.}
    \label{ex1-compare-h-version}
\end{minipage}
  \hfill\quad
\begin{minipage}[h]{0.48\linewidth}
\centering
   \includegraphics [height=2.2in]{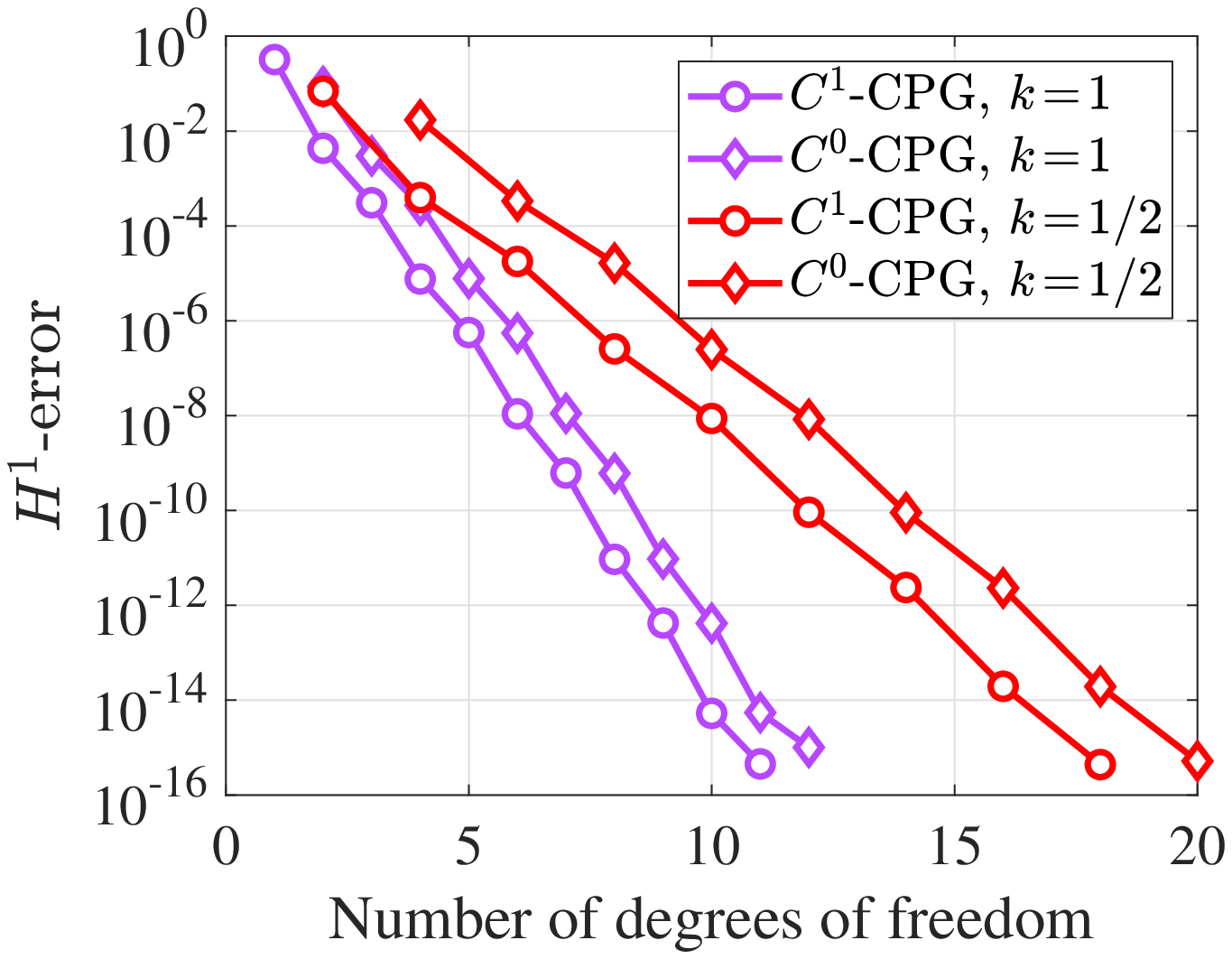}
   \caption{Example 1:  $C^1$-CPG method versus  $C^0$-CPG method, $H^1$-errors of the $p$-version.}
   \label{ex1-compare-p-version}
\end{minipage}
\end{figure}

\begin{table}[htp]
\tabcolsep 2pt \caption{Example 1: $C^1$-CPG method versus $C^0$-CPG
method: superconvergence of the $h$-version at nodal
points.}\label{table2}
\begin{center}
\begin{tabular}{|c|c||cc|cc||cc|cc|}\hline
\multicolumn{2}{|c||}{}  &\multicolumn{4}{|c||}{$C^1$-CPG method} &
\multicolumn{4}{|c|}{$C^0$-CPG method (cf. \cite{WY} )}         \\
\hline
   $~~r~~$ & $~~k~~$&$\mbox{max.}  ~e(t_n)$ & order &$\mbox{max.}  ~e'(t_n)$& order &$\mbox{max.}  ~e(t_n)$&order &$\mbox{max.}  ~e'(t_n)$& order\\ \hline
\multirow{3}{*}{2}&1/16&8.20e-04&2.01 &1.38e-03&2.00   &1.77e-05&3.03&3.38e-05&3.03\\
                  &1/32&2.05e-04&2.00 &3.46e-04&2.00   &2.19e-06&3.01&4.18e-06 &3.01\\
                  &1/64&5.12e-05&2.00 &8.64e-05&2.00   &2.73e-07&3.01&5.20e-07&3.00\\ \hline
\multirow{3}{*}{3}&1/8&5.72e-07&4.04&1.15e-06&4.03     &1.90e-08&5.09 &3.19e-08& 5.10\\
                 &1/16&3.55e-08&4.01&7.18e-08&4.01     &5.77e-10&5.04&9.66e-10& 5.04\\
                 &1/32&2.22e-09&4.00&4.48e-09&4.00     &1.78e-11&5.02&2.98e-11&5.02\\ \hline
\multirow{3}{*}{4}&1/4&2.79e-08 &5.98&4.89e-08&5.98    &8.85e-10&7.04&1.68e-09&7.04 \\
                  &1/8&4.37e-10 &6.00&7.65e-10&6.00    &6.84e-12&7.02&1.30e-11&7.01  \\
                  &1/16&6.84e-12&6.00&1.20e-11&6.00      &5.24e-14&7.03&1.00e-13&7.02 \\ \hline
\multirow{3}{*}{5}&1/2&1.28e-09&8.36&2.54e-09&8.31     &6.80e-11&9.41&1.09e-10&9.47\\
                  &1/4&4.71e-12&8.09&9.44e-12&8.07     &1.29e-13&9.04&2.05e-13&9.05\\
                  &1/8&1.79e-14&8.04&3.67e-14&8.00     &2.22e-16&9.18&3.33e-16&9.27\\ \hline

\end{tabular}
\end{center}
\end{table}

In Table \ref{table2},  we list  the maximum nodal errors
$\mbox{max.} e(t_n)$ and  $\mbox{max.} e'(t_n)$ of   the $h$-version
$C^1$-CPG and $C^0$-CPG methods. Clearly,  the results show that the
$C^1$-CPG method exhibits the superconvergence of the order
$O(k^{2r-2})$   while the  $C^0$-CPG method exhibits the
superconvergence of the order $O(k^{2r-1})$ at the nodes. It seems
that  the  $C^0$-CPG method  can achieve one order higher
superconvergence rate than the $C^1$-CPG method if  we use the same
approximation degree $r$.  However, it is worth  noting that   the
$C^1$-CPG method has $r-1$ DOF  while the $C^0$-CPG method has $r$
DOF (at each subinterval) if  the same approximation degree $r$ was
used. Hence, from another point of view,   if we use $r+1$th degree
$C^1$-CPG method (with the same number of DOF as the $r$th degree
$C^0$-CPG method), we can obtain $O(k^{2r})$ nodal superconvergence
rate (suppose  that $u$ is smooth enough), which is one order higher
than the $r$th degree $C^0$-CPG method.

We point out that for both the  $C^1$-CPG and the
$C^0$-CPG methods, the test spaces are based on piecewise
polynomials that are discontinuous at the time nodes,  and thus the
discrete Galerkin formulations can be decoupled into local problems
on each time step. In practice, both schemes  are transformed into
local algebraic system of the form (\ref{nonlinear-sys}) on each time
step. In this way, the computational difficulty in implementing
these schemes and the computational complexity/cost are basically at
the same level.  However,  if we employ the $C^1$-CPG and  $C^0$-CPG
methods for  time discretization of second-order evolutionary
equations such as wave equations,  a more thorough comparison is
needed, which includes possible preconditioned iterative methods and
  parallelization techniques.

\subsection{Example 2: a nonlinear Hamiltonian system}

We consider the   two-body problem \cite{WWX}:
\begin{equation}\label{ex3}
 \left\{\begin{array}{ll} \partial^2_tq_1(t)=-\ds\frac{q_1(t)}{(q^2_1(t)+q^2_2(t))^{3/2}},\quad & t\in[0, T],
\\[4mm]
\partial^2_tq_2(t)=-\ds\frac{q_2(t)}{(q^2_1(t)+q^2_2(t))^{3/2}},\quad & t\in[0, T],\\[5mm]
q_1(0)=1-\varepsilon,\quad  \partial_t q_1= 0,\\[2mm]
q_2(0)=0,\quad  \partial_t q_2(0)=
\sqrt{\ds\frac{1+\varepsilon}{1-\varepsilon}},
\end{array}\right.
\end{equation}
where $\varepsilon\in [0,1)$ is   the   eccentricity of elliptical
orbit. It is well-known that the Hamiltonian function of the system
is defined as
$$ H(t):= \frac 12 \left(p_1^2(t)+p_2^2(t)\right)-\frac{1}{\left(q_1^2(t)+q_2^2(t)\right)^{1/2}}, $$
where $p_1(t)=  q_1'(t)$ and $p_2(t)=  q_2'(t)$.

\begin{figure}[h!]
\begin{minipage}[h]{0.48\linewidth}
\centering
  \includegraphics [height=2.2in]{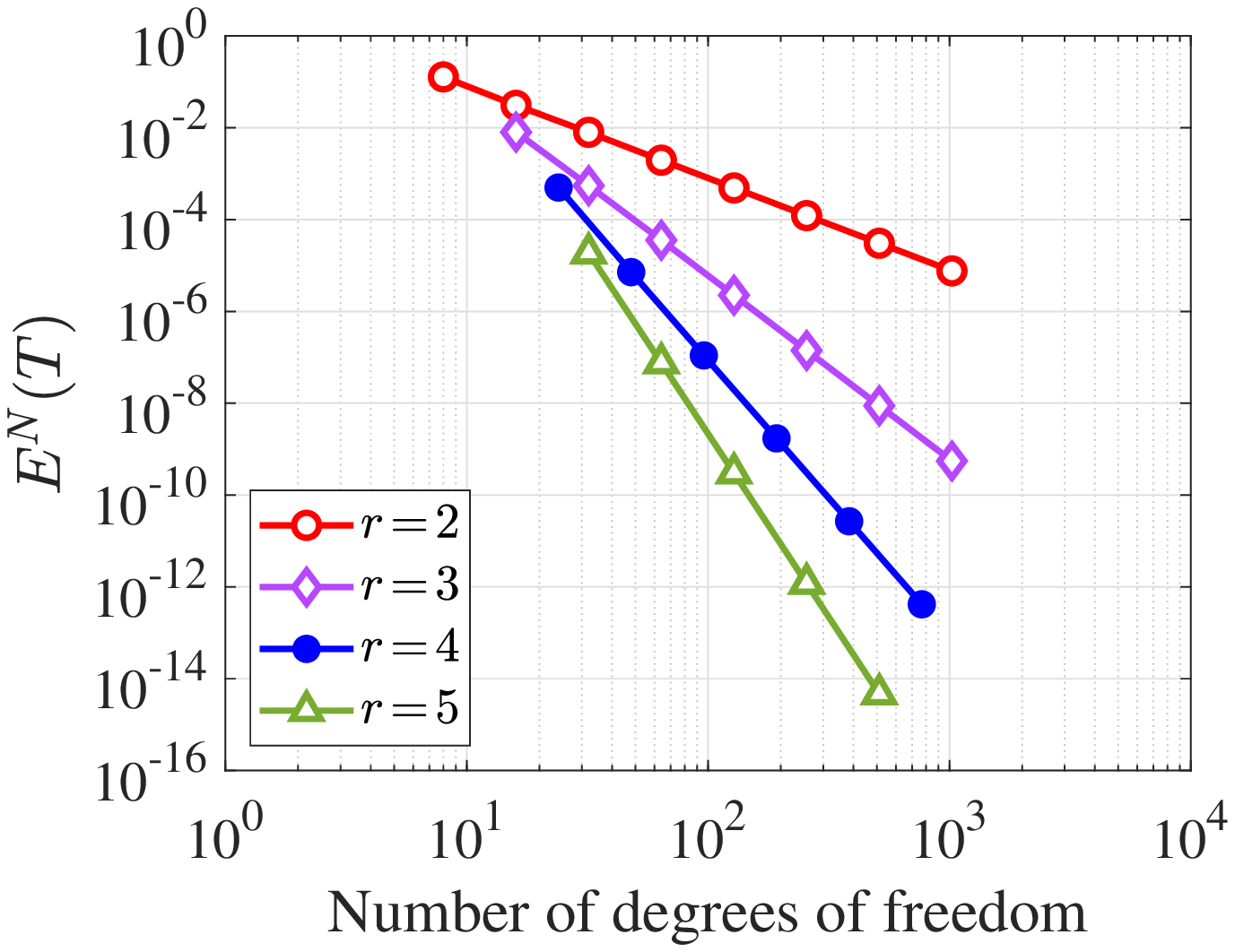}
    \caption{Example 2: energy errors of the $h$-version at $T=10$.}
    \label{Ex3-h-version}
\end{minipage}
  \hfill\quad
\begin{minipage}[h]{0.48\linewidth}
\centering
   \includegraphics [height=2.2in]{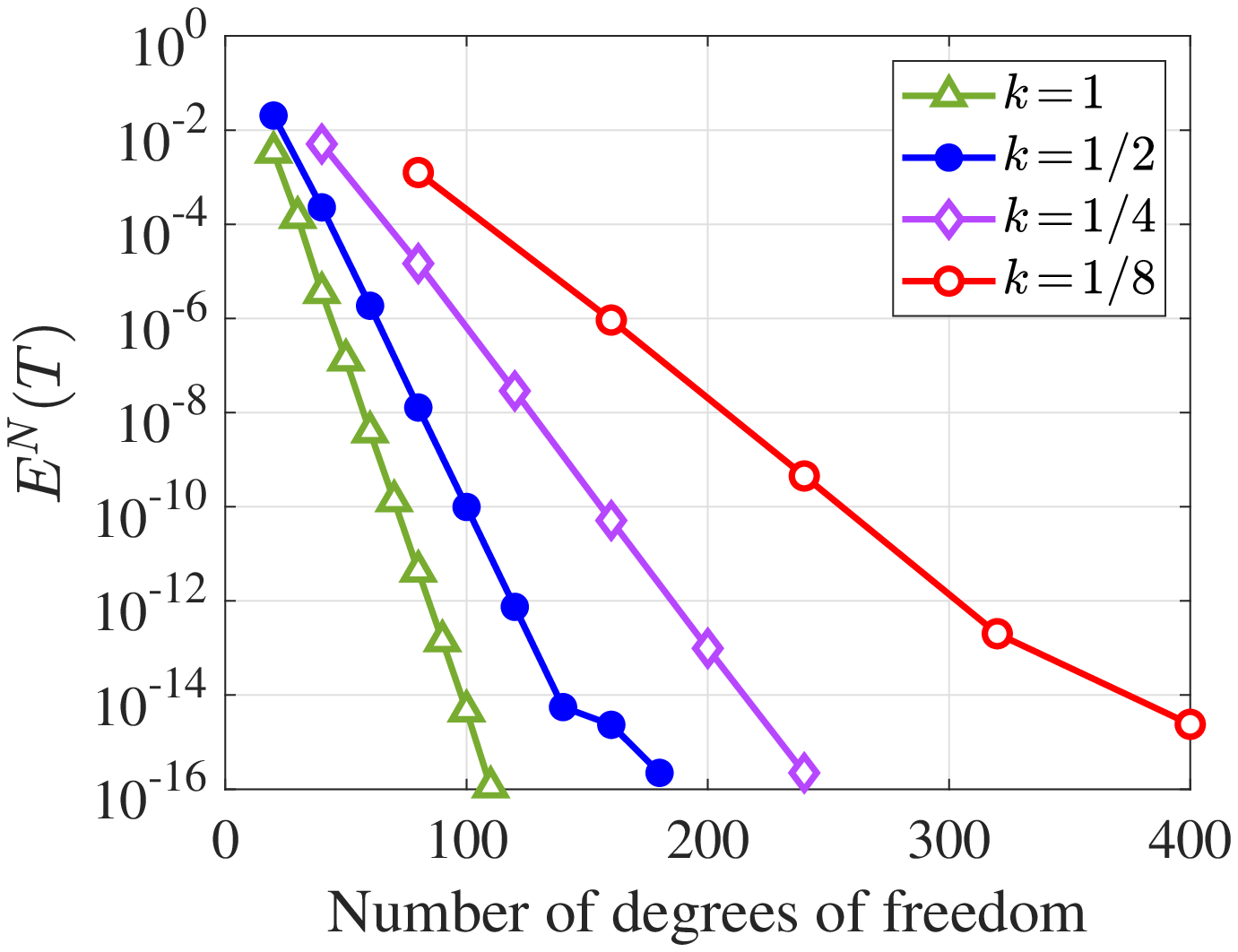}
   \caption{Example 2: energy errors of the $p$-version at $T=10$.}
   \label{Ex3-p-version}
\end{minipage}
\end{figure}

To describe   the numerical errors, we denote by  $Q_1(t)$ and
$Q_2(t)$ the $C^1$-CPG approximations  to $q_1(t)$ and $q_2(t)$,
respectively. We further denote by $H^N(t)$    the numerical energy
of the Hamiltonian, and by   $E^N(t)$ the energy error at  $t$,
i.e.,
\begin{equation*}\label{hami-err}
  E^N(t)=|H^N(t)-H(0)|.
\end{equation*}
Here,  $H(0)$ is the initial energy of the Hamiltonian.

We now consider the performance of the $h$- and $p$-versions of the
$C^1$-CPG method for  problem (\ref{ex3}) with  $\varepsilon=0.2$ and
$T=10$. Figure \ref{Ex3-h-version} shows that  the $h$-version
   exhibits  algebraic convergence rates while Figure \ref{Ex3-p-version} shows that  the $p$-version
   exhibits   exponential convergence rates.

\begin{figure}[h!]
\begin{minipage}[h]{0.48\linewidth}
\centering
  \includegraphics [height=2.2in]{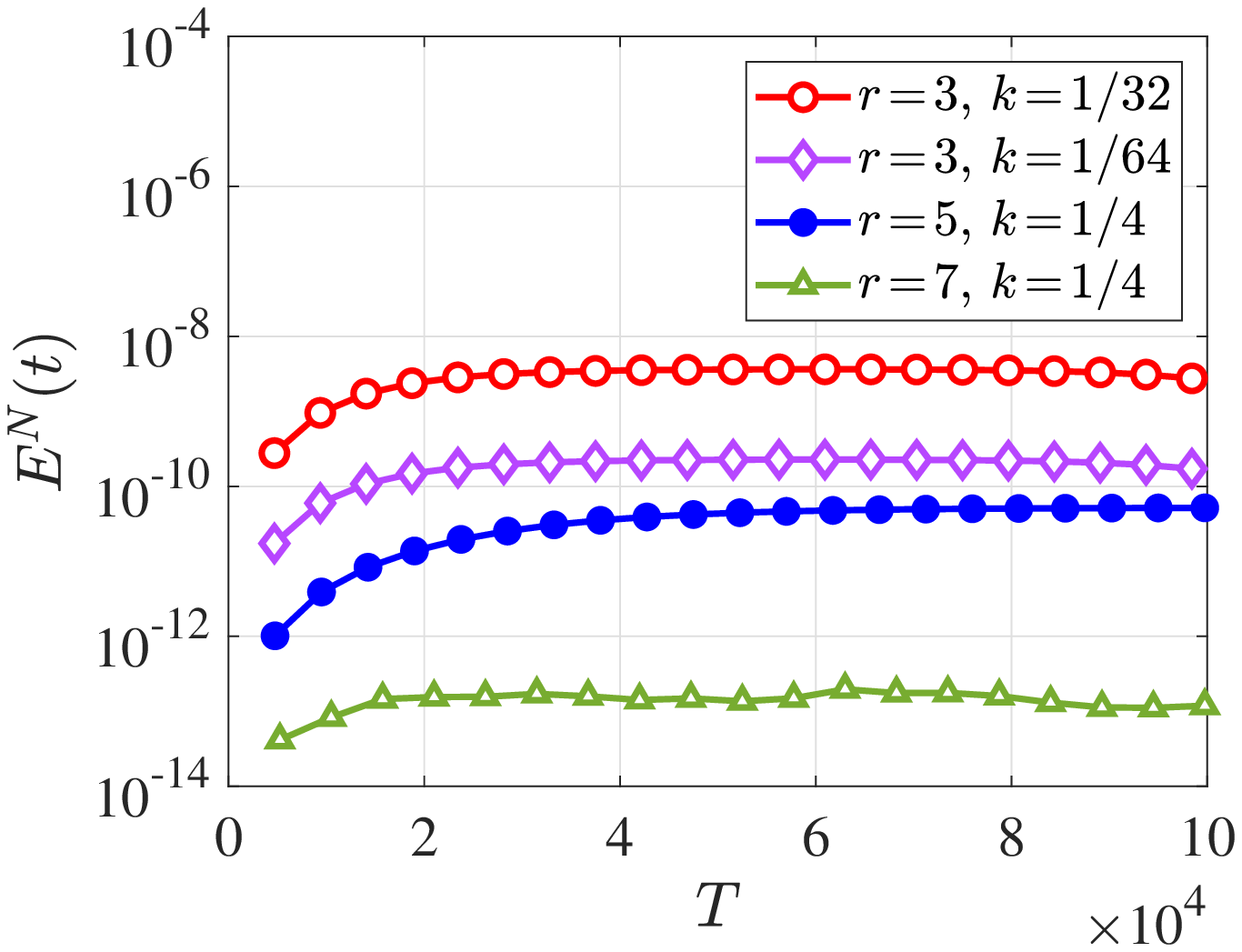}
    \caption{Example 2: energy errors of the $C^1$-CPG method for $t\in [0, 10^5]$.}
    \label{Ex3-long}
\end{minipage}
  \hfill\quad
\begin{minipage}[h]{0.48\linewidth}
\centering
   \includegraphics [height=2.2in]{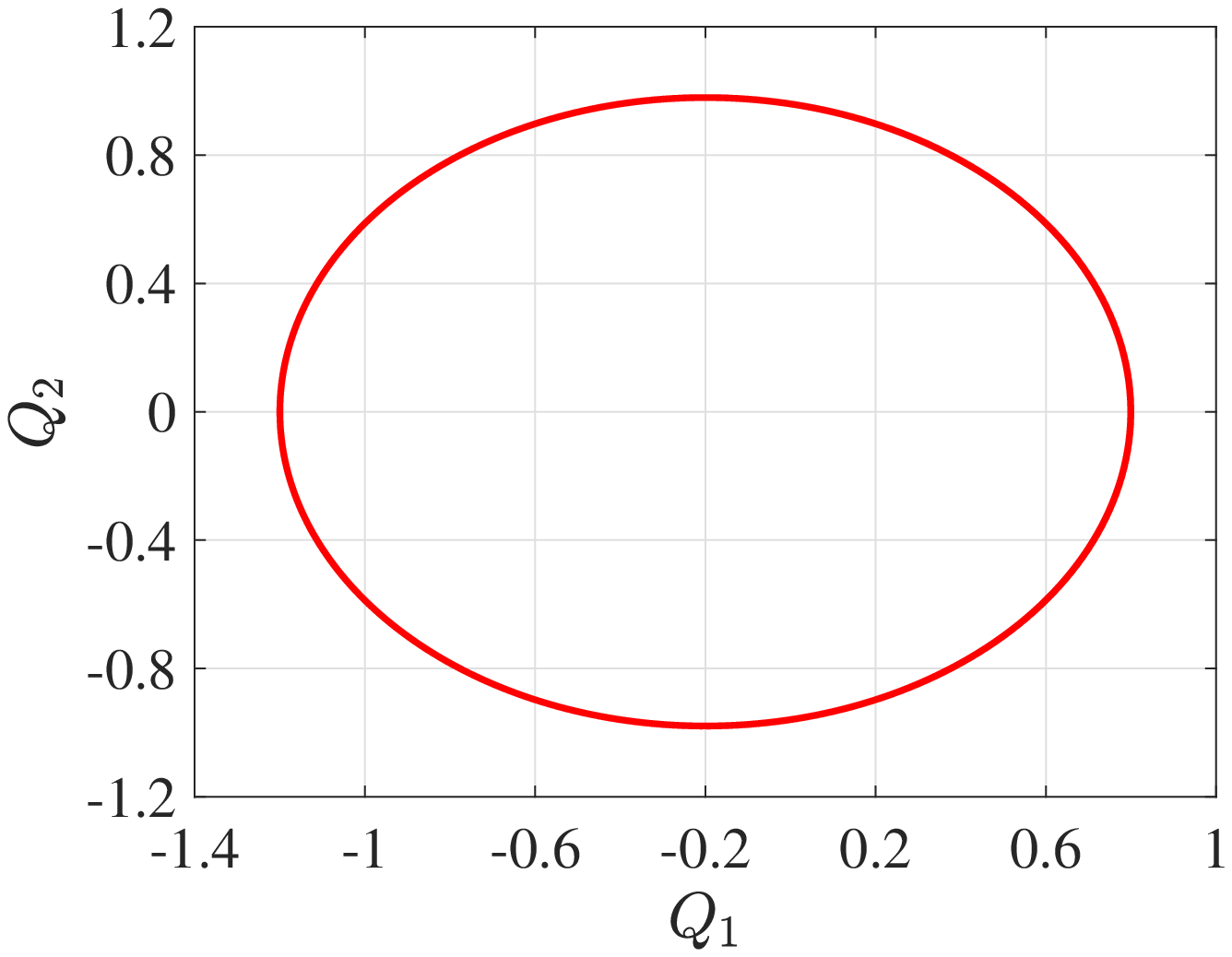}
   \caption{Example 2: numerical orbit $\left(Q_1(t), Q_2(t)\right)$ for $t\in [0, 10^5]$.}
   \label{Ex3-orbit}
\end{minipage}
\end{figure}

In Figure \ref{Ex3-long}, we plot the point-wise  energy errors
$E^N(t)$  of the $C^1$-CPG method (with different step-sizes $k$ and
approximation degrees $r$) for problem (\ref{ex3}) with $t\in [0,
10^5]$. Clearly, the $C^1$-CPG method  is stable and accurate for
long-time computation. In Figure \ref{Ex3-orbit}, we  plot the
numerical orbit $\left(Q_1(t), Q_2(t)\right)$ for  $t\in [0, 10^5]$.
Clearly,

\subsection{Example 3: a   linear wave equation}

We consider the two-dimensional linear wave problem:
\begin{equation}\label{ex-wave}
 \left\{\begin{aligned}
 \partial_t^2u-\Delta u &=f \quad  &&\text{in}~ \Omega \times  I,
 \\  u&=0 \quad  &&\text{on}~ \partial\Omega \times  I,
\\ u(\cdot,0)&=u_0 \quad && \text{in}~ \Omega,
 \\ \partial_t u(\cdot,0)&=u_1  \quad &&\text{in}~ \Omega,
 \end{aligned}\right.
\end{equation}
where $\Omega=[0,1]\times[0,1]$ and $I=(0,T)$ with $T=1$. Let $u_0,
u_1$ and $f=f(x,y,t)$ be chosen such that the exact solution is
given by $u=x(1-x)y(1-y)\cos(t)$.

In this example, we   use the quadratic finite element method (with
uniform partition of $\Omega$ which consists of $100$ square elements) for
spatial discretization. For each fixed time $t$, it can be seen that
the exact solution $u$ is an element of the finite element space
$V_h$, which implies that there is no spatial error. Therefore, we
can concentrate only on the  time discretization error and exclude
interactions with the spatial error in this example.

After spatial discretization, we further use the $C^1$-CPG  time
stepping method for time discretization of the problem
(\ref{ex-wave}).   We employ uniform time partitions with step-size
$k$ and uniform approximation degrees $r$. Let $u_{h\tau}$  be the
fully discrete  Galerkin
 approximations as defined in Section
\ref{sec4-3}. For simplicity, we  denote the error function  by
$e=u-u_{h\tau}$. The   errors  in
 different norms are defined as follows:
$$\|e\|_{L^2(L^2)}:=\left(\int_{I} \|e\|^2_{L^2(\Omega)} dt \right)^{\frac
12},\quad \|e\|_{H^1(L^2)}:=\left(\int_{I}\left(
\|e\|^2_{L^2(\Omega)} + \|\partial_t e\|^2_{L^2(\Omega)}\right)
dt\right)^{\frac 12},$$\vspace{-2mm}
$$\|e\|_{H^2(L^2)}:=\left(\int_{I}\left( \|e\|^2_{L^2(\Omega)} +
\|\partial_t e\|^2_{L^2(\Omega)}+ \|\partial_{tt}
e\|^2_{L^2(\Omega)}\right) dt\right)^{\frac 12}, \quad
\|e\|_{L^\infty(L^2)}:= \max_{t\in \bar I}\|e\|_{L^2(\Omega)}.$$

\begin{figure}[h!]
\begin{minipage}[h]{0.48\linewidth}
\centering
  \includegraphics [height=2.2in]{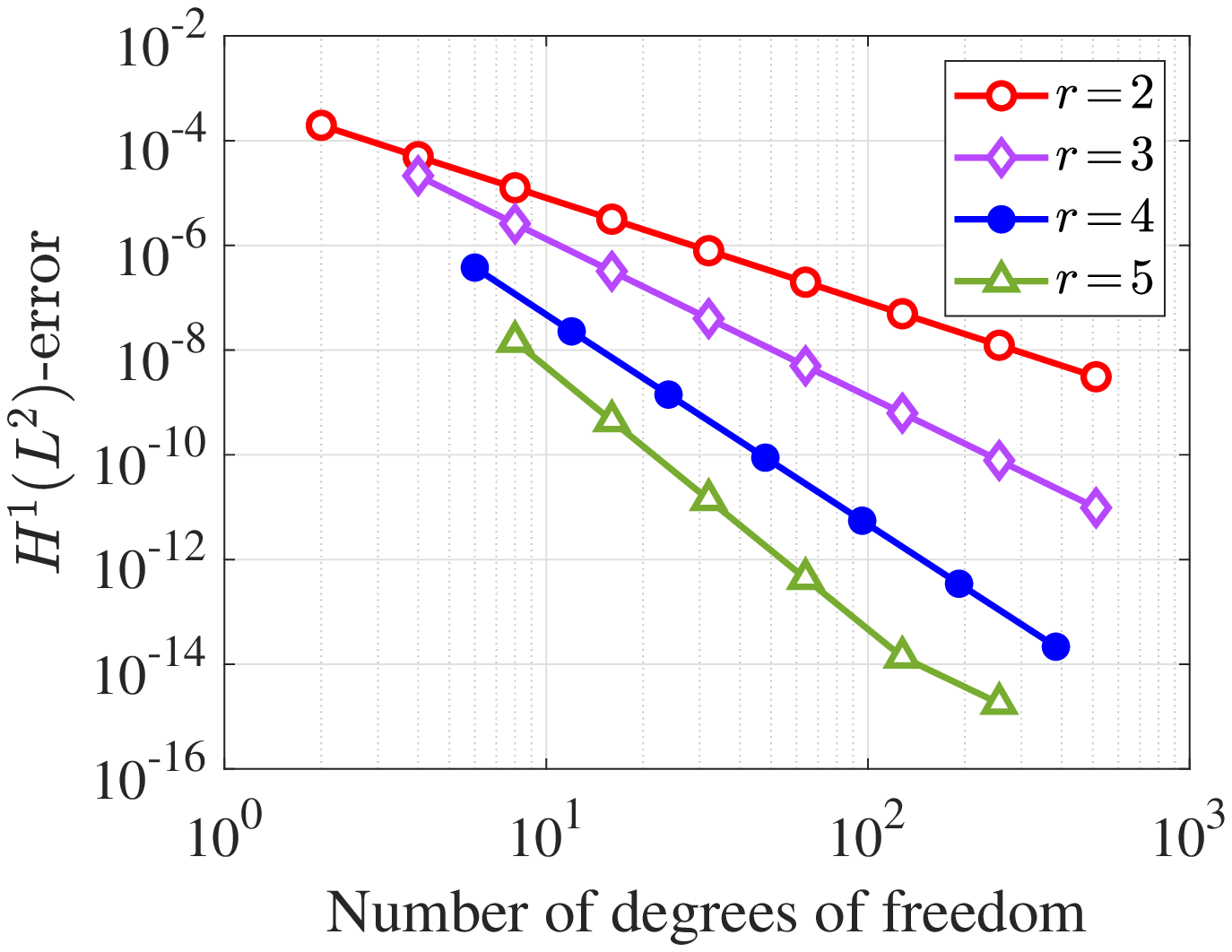}
    \caption{Example 3: $H^1(L^2)$-errors of the $h$-version.}
    \label{wave-H1-h-version}
\end{minipage}
  \hfill\quad
\begin{minipage}[h]{0.48\linewidth}
\centering
   \includegraphics [height=2.2in]{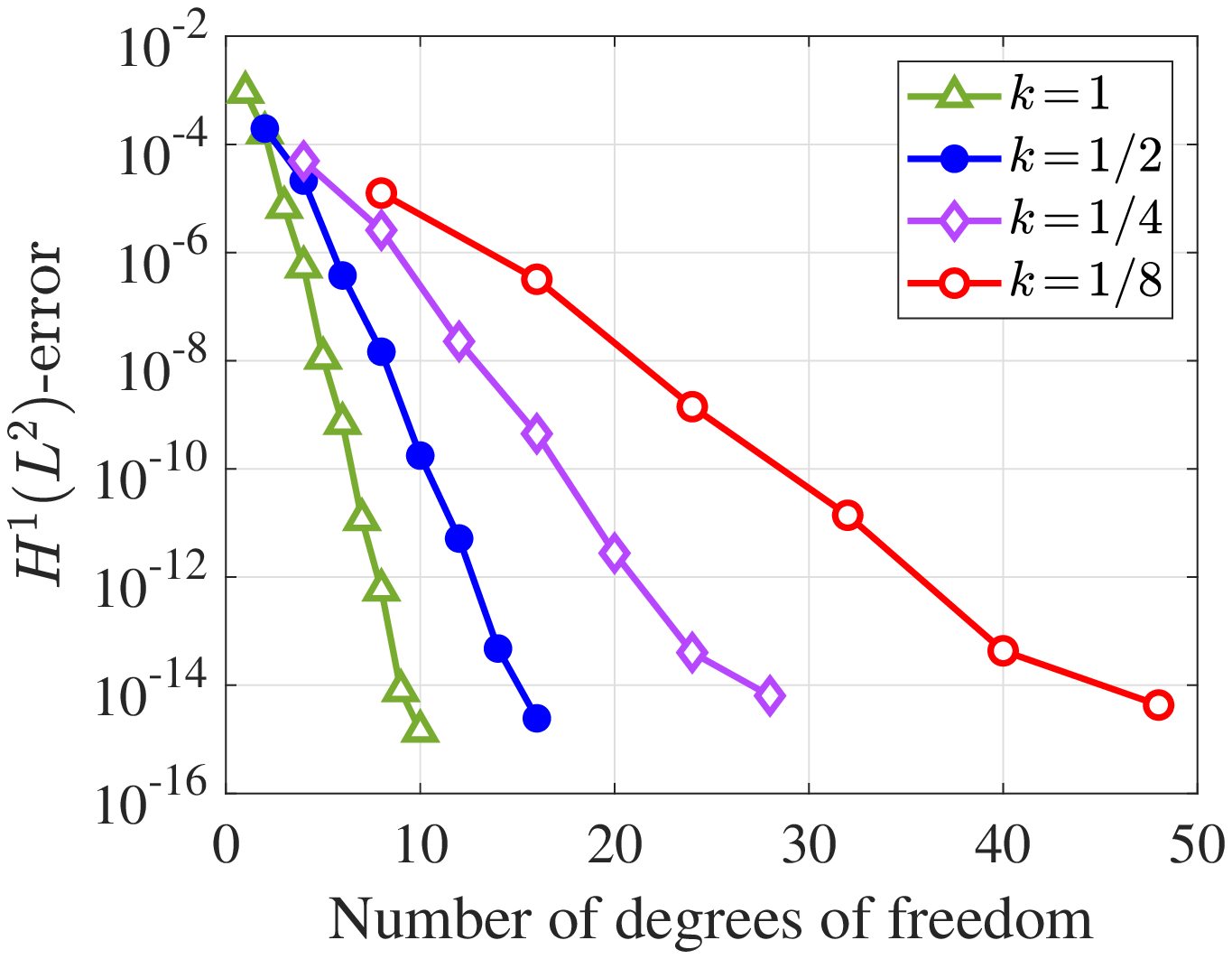}
   \caption{Example 3: $H^1(L^2)$-errors of the $p$-version.}
   \label{wave-H1-p-version}
\end{minipage}
\end{figure}


\begin{table}[htp]
\vspace{-3mm}
\tabcolsep 3pt \caption{Example 3: numerical errors and convergence
orders of the $h$-version.}\label{table3}
\begin{center}
\begin{tabular}{|c|c|cc|cc|cc|cc|}
\hline
  $~~r~~$ &   $~~k~~$&$\|e\|_{L^2(L^2)}$& order & $\|e\|_{H^1(L^2)}$& order &$\|e\|_{H^2(L^2)}$&order     & $\|e\|_{L^\infty(L^2)}$& order \\
  \hline
\multirow{3}{*}{2}&1/64   &4.39e-08 &2.00 &1.98e-07   &2.00
&7.85e-05&1.00   &6.46e-08 &2.01\\
                  &1/128  &1.10e-08 &2.00 &4.95e-08   &2.00   &3.93e-05&1.00   &1.61e-08
                  &2.01\\
                  &1/256  &2.75e-09 &2.00 &1.24e-08   &2.00   &1.96e-05&1.00   &4.01e-09
                  &2.00\\\hline
\multirow{3}{*}{3}&1/32   &4.22e-11 &4.00 &4.99e-09   &3.00
&1.03e-06&2.00   &8.27e-11 &3.99\\
                  &1/64   &2.64e-12 &4.00 &6.24e-10   &3.00   &2.59e-07&2.00   &5.17e-12
                  &4.00\\
                  &1/128  &1.64e-13 &4.00 &7.80e-11   &3.00   &6.47e-08&2.00   &3.23e-13
                  &4.00\\\hline
\multirow{3}{*}{4}&1/8    &2.67e-11 &5.02 &1.41e-09   &4.01
&1.07e-07&3.00   &6.18e-11 &4.99 \\
                  &1/16   &8.32e-13 &5.00 &8.82e-11   &4.00   &1.34e-08&3.00   &1.95e-12 &4.99
                  \\
                  &1/32   &2.60e-14 &5.00 &5.51e-12   &4.00   &1.67e-09&3.00   &6.09e-14 &5.00
                  \\\hline
\multirow{3}{*}{5}&1/4    &1.21e-11 &6.08 &4.47e-10   &5.04   &2.21e-08&4.01   &2.57e-11 &6.05    \\
                  &1/8    &1.87e-13 &6.02 &1.39e-11   &5.01   &1.38e-09&4.00   &3.96e-13 &6.02    \\
                  &1/16   &2.93e-15 &6.00 &4.34e-13   &5.00   &8.61e-11&4.00   &6.17e-15 &6.01    \\ \hline
\end{tabular}
\end{center}
\end{table}

 In Figure \ref{wave-H1-h-version}, we plot the
$H^1(L^2)$-errors of the $h$-version $C^1$-CPG method   against the
total number of DOF in time direction for different $r$. The error
curves in log-log scale imply that the convergence order in time is
algebraic for the $h$-version.  Moreover, we   list the numerical
errors (in different norms) and convergence orders of the
$h$-version method in Table \ref{table3}. It can be seen that the
convergence orders (in time) are similar as those reported in
  Example 1 for  the scalar ODE.
In  Figure \ref{wave-H1-p-version}, we plot the $H^1(L^2)$-errors of
the $p$-version $C^1$-CPG method  (on fixed  time partitions with
$1, 2, 4, 8$ uniform time steps) in a semi-log scale. Clearly, the
exponential convergence is achieved for each time partition as $r$
increases.

\subsection{Example 4: a nonlinear wave equation}

We consider the two-dimensional sine-Gordon equation:
\begin{equation}\label{ex-sine}
 \left\{\begin{aligned}
 \partial_t^2u-\Delta u+\sin u &=f  &&\quad  \text{in}~ \Omega \times  I,\\
 u  &=0 &&\quad  \text{on}~ \partial\Omega \times  I,\\
 u(\cdot,0) &=u_0 &&\quad  \text{in}~ \Omega,\\
 \partial_t u(\cdot,0) &=u_1  &&\quad \text{in}~ \Omega,
 \end{aligned}\right.
\end{equation}
where $\Omega=[-1,1]\times[-1,1]$ and $I=(0,T)$ with $T=2$. Let
$u_0, u_1$ and $f=f(x,y,t)$ be chosen such that the exact solution
is given by $u=\sin(\pi x)\sin(\pi y)\cos(2\pi t)$.

In this example, we  first use the Spectral-Galerkin method  for
spatial discretization.  We choose the Lobatto polynomials of degree
$20$ in both $x$ and $y$ directions  as   basis function of  the
spectral Galerkin approximation  such that the fully discrete error
is dominated by the time discretization error. We further use the
$C^1$-CPG  time stepping method (with uniform step-size $k$ and
approximation degrees $r$) for time discretization of the problem
(\ref{ex-sine}).

\begin{figure}[h!]
\begin{minipage}[h]{0.48\linewidth}
\centering
  \includegraphics [height=2.2in]{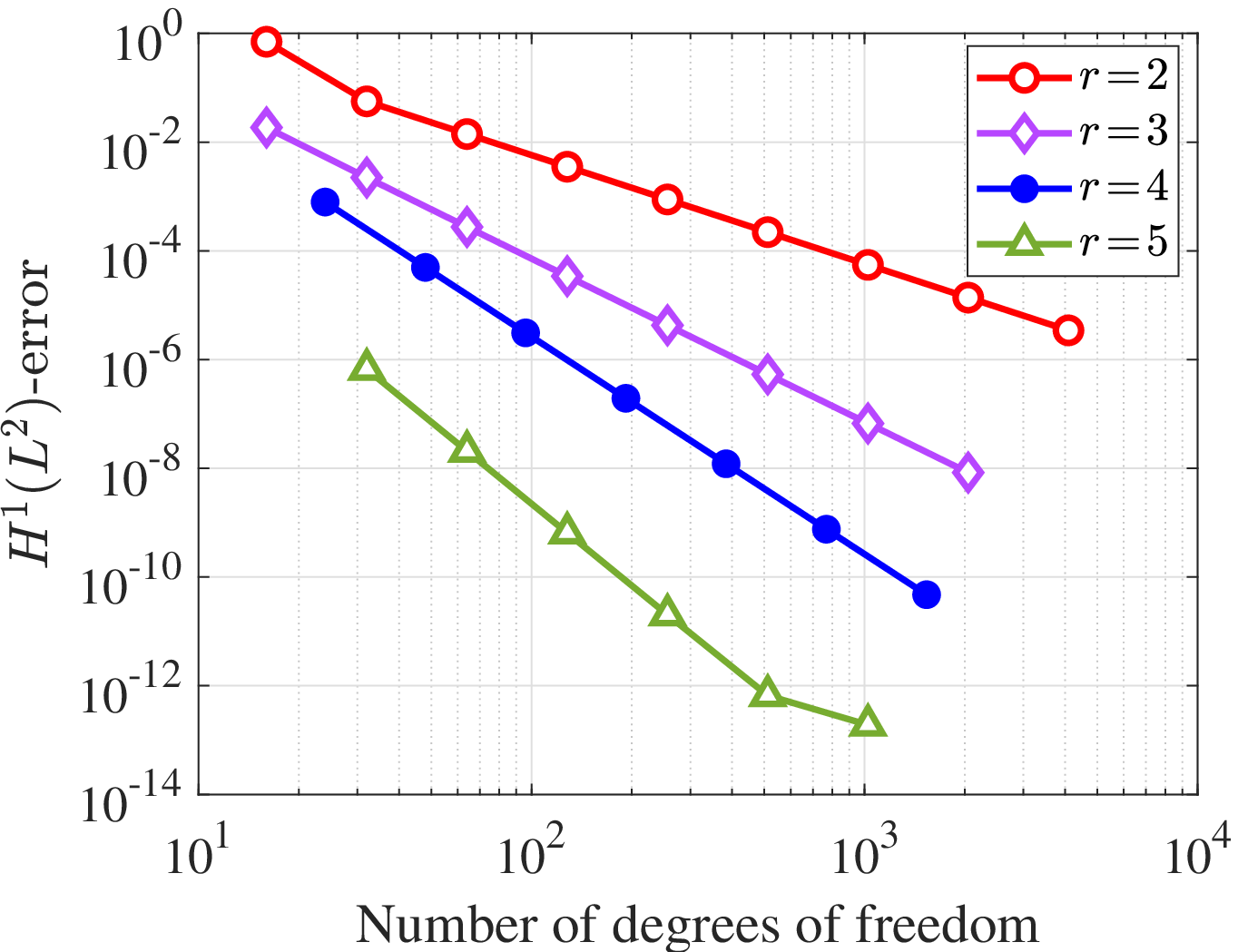}
    \caption{Example 4: $H^1(L^2)$-errors of the $h$-version.}
    \label{nonlinear-wave-H1-h-version}
\end{minipage}
  \hfill\quad
\begin{minipage}[h]{0.48\linewidth}
\centering
  \includegraphics [height=2.2in]{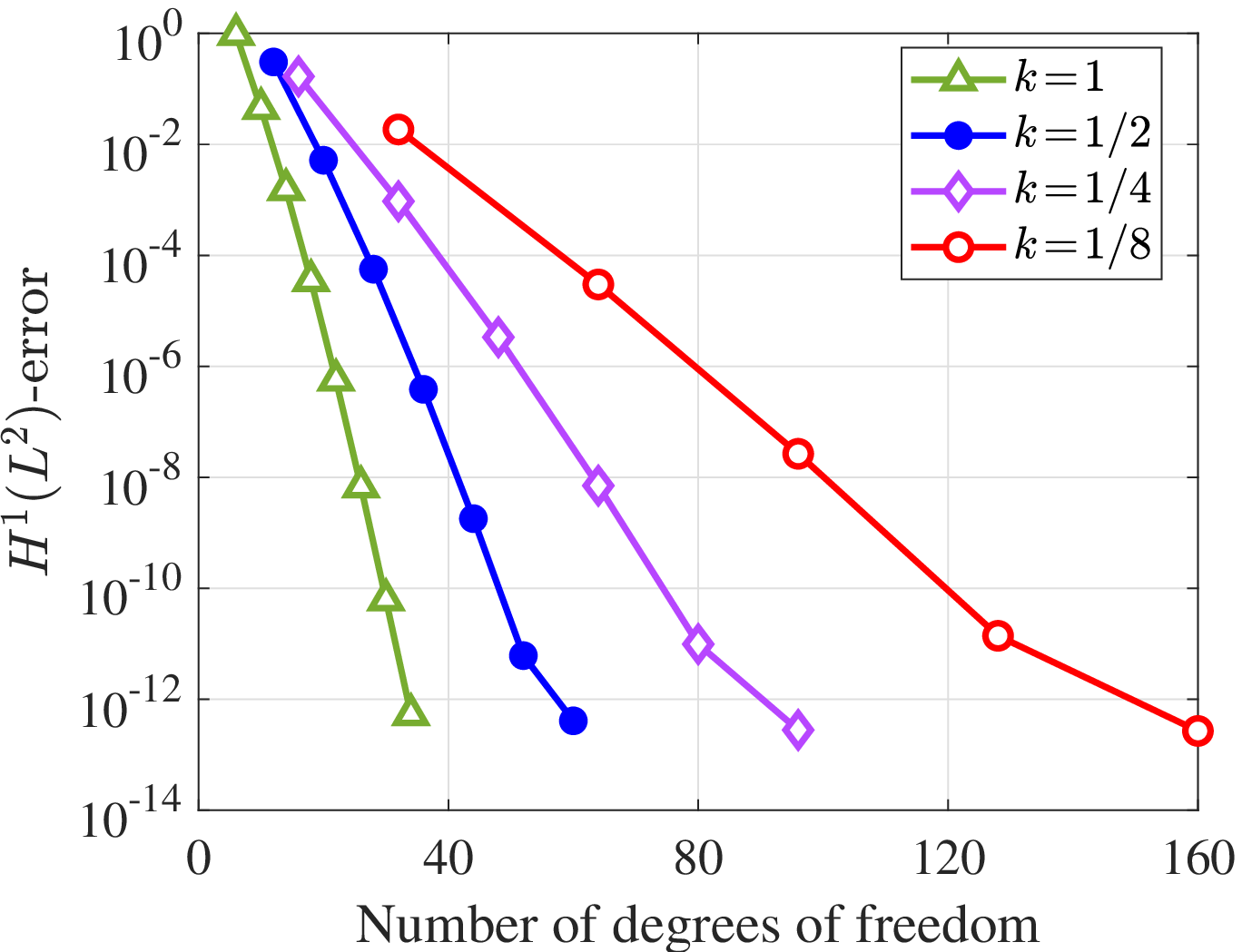}
   \caption{Example 4: $H^1(L^2)$-errors of the $p$-version.}
   \label{nonlinear-wave-H1-p-version}
\end{minipage}
\end{figure}

\begin{table}[htp]
\vspace{-3mm} \tabcolsep 3pt \caption{Example 4: numerical errors
and convergence orders of the $h$-version.}\label{table4-nonlinear}
\begin{center}
\begin{tabular}{|c|c|cc|cc|cc|cc|}
\hline
  $~~r~~$ &   $~~k~~$&$\|e\|_{L^2(L^2)}$& order & $\|e\|_{H^1(L^2)}$& order &$\|e\|_{H^2(L^2)}$&order     & $\|e\|_{L^\infty(L^2)}$& order \\
  \hline
\multirow{3}{*}{2}
        &1/128 & 6.11e-04 & 2.00 & 3.53e-03 & 2.00  & 5.60e-01 & 1.00  & 8.12e-04 & 2.00  \\
        &1/256 & 1.53e-04 & 2.00  & 8.83e-04 & 2.00  & 2.80e-01 & 1.00  & 2.03e-04 & 2.00  \\
        &1/512 & 3.82e-05 & 2.00  & 2.21e-04 & 2.00  & 1.40e-01 & 1.00  & 5.08e-05 & 2.00  \\ \hline
\multirow{3}{*}{3}
        &1/64 & 3.23e-07 & 4.00  & 3.42e-05 & 3.00  & 1.42e-02 & 2.00  & 5.11e-07 & 3.99  \\
        &1/128 & 2.02e-08 & 4.00  & 4.27e-06 & 3.00  & 3.55e-03 & 2.00  & 3.20e-08 & 4.00  \\
        &1/256 & 1.26e-09 & 4.00  & 5.34e-07 & 3.00  & 8.86e-04 & 2.00 & 2.00e-09 & 4.00 \\ \hline
\multirow{3}{*}{4}
        &1/32 & 1.44e-08 & 5.02  & 3.10e-06 & 4.00  & 9.41e-04 & 3.00  & 2.20e-08 & 5.03  \\
        &1/64 & 4.48e-10 & 5.00  & 1.94e-07 & 4.00  & 1.18e-04 & 3.00  & 6.79e-10 & 5.02  \\
        &1/128 & 1.40e-11 & 5.00  & 1.21e-08 & 4.00  & 1.47e-05 & 3.00  & 2.11e-11 & 5.01  \\ \hline
\multirow{3}{*}{5}
        &1/16 & 6.75e-11 & 6.00  & 2.07e-08 & 5.00  & 8.23e-06 & 4.00  & 1.77e-10 & 5.98  \\
        &1/32 & 1.05e-12 & 6.00  & 6.48e-10 & 5.00  & 5.14e-07 & 4.00  & 2.77e-12 & 6.00  \\
        &1/64 & 1.79e-14 & 5.88  & 2.02e-11 & 5.00  & 3.22e-08 & 4.00  & 4.34e-14 & 6.00  \\ \hline
\end{tabular}
\end{center}
\end{table}

 In Figures \ref{nonlinear-wave-H1-h-version} and
 \ref{nonlinear-wave-H1-p-version},  we plot the
$H^1(L^2)$-errors of the $h$- and $p$-versions of the $C^1$-CPG
method, respectively. It can be seen that the convergence (with
respect to the total number of DOF in time direction)  of the
$h$-version is  algebraic and the convergence of the  $p$-version is
exponential, which are   similar as those reported in
  Example $3$ for  the linear wave equation (\ref{ex-wave}).

\begin{table}[htp]
\tabcolsep 8pt \caption{Example 4: superconvergence of the
$h$-version at nodal points.}\label{table4-node}
\begin{center}
\begin{tabular}{|c|c|cc|cc|}\hline
   $~~r~~$ & $~~k~~$&$\mbox{max.}  ~e(t_n)$ & order &$\mbox{max.}  ~e'(t_n)$& order \\ \hline
\multirow{3}{*}{2}&1/32 & 2.34e-03 & 2.00 & 7.50e-03 & 2.00  \\
        &1/64 & 5.85e-04 & 2.00  & 1.88e-03 & 2.00   \\
        &1/128 & 1.46e-04 & 2.00  & 4.69e-04 & 2.00  \\ \hline
\multirow{3}{*}{3}  &1/16 & 3.12e-06 & 3.98  & 1.34e-05 & 3.98  \\
        &1/32 & 1.95e-07 & 4.00  & 8.36e-07 & 4.00 \\
        &1/64 & 1.22e-08 & 4.00 & 5.23e-08 & 4.00  \\ \hline
\multirow{3}{*}{4} &1/8 & 5.46e-08 & 6.00  & 1.77e-07 & 6.00  \\
        &1/16 & 8.61e-10 & 5.99  & 2.76e-09 & 6.00  \\
        &1/32 & 1.34e-11 & 6.00  & 4.32e-11 & 6.00 \\ \hline
\multirow{3}{*}{5} &1/4 & 4.34e-09 & 7.70  & 1.84e-08 & 7.91  \\
        &1/8 & 1.73e-11 & 7.97  & 7.35e-11 & 7.97  \\
        &1/16 & 6.56e-14 & 8.04  & 3.41e-13 & 7.75 \\ \hline
\end{tabular}
\end{center}
\end{table}

In Table \ref{table4-nonlinear},  we  list the numerical errors (in
different norms) and convergence orders of the $h$-version $C^1$-CPG
method. It can be seen that the convergence orders (in time) are
similar as those reported in   Example $1$ for  the scalar ODE. In
Table \ref{table4-node}, we also present  the maximum nodal errors
$\mbox{max.}  ~e(t_n):= \max\limits_{1\le n \le N}
\|e(t_n)\|_{L^2(\Omega)}$ and $\mbox{max.} ~e'(t_n):=
\max\limits_{1\le n \le N} \|e'(t_n)\|_{L^2(\Omega)}$ of the
$h$-version $C^1$-CPG method. Clearly,  the results show that the
$C^1$-CPG method exhibits the superconvergence  order $O(k^{2r-2})$
(with respect to the time step-size $k$) at the nodal points.

\section{ Concluding Remarks}\label{sec6}

In this paper, we have introduced and analyzed an $hp$-version
$C^1$-CPG method for a general nonlinear second-order IVP of ODE. We
have derived several a-priori error estimates that are fully
explicit with regard to the local discretization and local
regularity parameters.   Moreover, we have shown that  the
$hp$-version  $C^1$-CPG method superconverges at the nodal points of
the time partition with respect to the time steps and approximation
degrees. As an application, we have also applied the $hp$-version
$C^1$-CPG method to the time discretization of nonlinear
second-order wave equations. Error analysis of the space-time fully
discrete Galerkin method (based on the $hp$-version  CG  method for
spatial discretization  and the $hp$-version  $C^1$-CPG method for
time   discretization) for wave equations will make the subject of
our future research.

\appendix

\setcounter{equation}{0}

\renewcommand\theequation{A.\arabic{equation}} 

\section{Proofs of some lemmas}

\subsection{ Proof of Lemma \ref{exi-uni}}\label{app-A-1}
\begin{proof}
In view of Remark \ref{rem1}, it suffices to prove that, for given
initial values  $U|_{I_{n-1}} (t_{n-1})$ and $U'|_{I_{n-1}}
(t_{n-1})$,  the problem (\ref{C1CPG-FEM-1}) has a unique solution
on the subinterval $I_n$.

To this end, we define that,  for any $\widetilde{U}\in
P_{r_n}(I_n)$, the polynomial $U:=G \widetilde{U}\in P_{r_n}(I_n)$
as the solution of   the following variational problem
\begin{eqnarray}\label{C1CPG-FEM-2}
\left\{\begin{array}{ll}
 \ds\int_{I_n} U''\varphi dt =
 \ds\int_{I_n} f(t,\widetilde{U},\widetilde{U}')\varphi(t)dt,\\[4mm]
U|_{I_n} (t_{n-1}) =U|_{I_{n-1}} (t_{n-1}),\quad U'|_{I_n} (t_{n-1})
=U'|_{I_{n-1}} (t_{n-1})
\end{array}\right.
\end{eqnarray}
for all $\varphi\in P_{r_n-2}(I_n)$. Clearly,  (\ref{C1CPG-FEM-2})
is a linear system of $r_n+1$ equations which is uniquely solvable.
Hence,  $G \widetilde{U}$ is well-defined.

We will prove that, the  operator $G$  is a contraction on
$P_{r_n}(I_n)$ provided that $k_n$ is sufficiently small. Then, by
Banach's fixed point theorem, there exists a unique fixed point
$U^*\in P_{r_n}(I_n)$ with $U^*=G U^*$, namely, $U^*$ is the unique
solution of (\ref{C1CPG-FEM-1}).

It remains to prove the contraction property of the operator $G$, we
let $\widetilde{V},~\widetilde{W}\in P_{r_n}(I_n)$ and set
$V=G\widetilde{V}$, $W=G\widetilde{W}$. By the definition of $G$ and
(\ref{Lip-con}), there holds
\begin{equation}\label{uniq-10}
\ds\int_{I_n} (V-W)''\varphi dt =
 \ds\int_{I_n}
 \left(f(t,\widetilde{V},\widetilde{V}')-f(t,\widetilde{W},\widetilde{W}')\right)\varphi(t)dt
 \le   L \ds\int_{I_n} \Big(|\widetilde{V}-\widetilde{W}|+|\widetilde{V}'-\widetilde{W}'|\Big)\varphi dt
\end{equation}
for all $\varphi\in P_{r_n-2}(I_n)$. Selecting $\varphi=(V-W)''$ in
(\ref{uniq-10}) and using the Cauchy-Schwarz inequality,  gives
\begin{eqnarray*}
\|V''-W''\|^2_{L^2(I_n)} \leq \sqrt{2}L\left\{\ds\int_{I_n}
\left(|\widetilde{V}-\widetilde{W}|^2+|\widetilde{V}'-\widetilde{W}'|^2\right)dt\right\}^{\frac{1}{2}}
\|V''-W''\|_{L^2(I_n)},
\end{eqnarray*}
which implies
\begin{equation}\label{uniq-11}
\|V''-W''\|_{L^2(I_n)}\leq
\sqrt{2}L\|\widetilde{V}-\widetilde{W}\|_{H^1(I_n)}.
\end{equation}

Since $(V-W)(t_{n-1})=(V-W)'(t_{n-1})=0$, there holds
\begin{equation}\label{uniq-12}
(V-W)(t)=\ds\int_{t_{n-1}}^{t}\left(\ds\int_{t_{n-1}}^{\eta}(V-W)''(s)ds\right)d\eta.
\end{equation}
Using the  Cauchy-Schwarz inequality, we have
\begin{eqnarray*}
\begin{aligned}
\ds\int_{t_{n-1}}^{t} \left(\ds\int_{t_{n-1}}^{\eta}
(V-W)''(s)ds\right)d\eta &\leq
\ds\int_{t_{n-1}}^{t}\left(\ds\int_{t_{n-1}}^{\eta}ds\right)^{\frac{1}{2}}
\left(\ds\int_{t_{n-1}}^{\eta}|V''-W''|^2ds\right)^{\frac{1}{2}}d\eta\\
&=\ds\int_{t_{n-1}}^{t}
(\eta-t_{n-1})^{\frac{1}{2}}\left(\ds\int_{t_{n-1}}^{\eta}|V''-W''|^2ds\right)^{\frac{1}{2}}d\eta\\
&\leq \left\{\ds\int_{t_{n-1}}^{t}(\eta-t_{n-1})d\eta
\right\}^{\frac{1}{2}} \left\{\ds\int_{t_{n-1}}^{t}
\left(\ds\int_{t_{n-1}}^{\eta}|V''-W''|^2ds\right)d\eta\right\}^{\frac{1}{2}}\\
&\leq
\sqrt{\ds\frac{1}{2}}(t-t_{n-1})\left\{\ds\int_{t_{n-1}}^{t}\left(\ds\int_{t_{n-1}}^{\eta}|V''-W''|^2ds\right)d\eta\right\}^{\frac{1}{2}},
 \end{aligned}
\end{eqnarray*}
and hence, by (\ref{uniq-12}) we get
\begin{eqnarray}\label{uniq-15}
\begin{aligned}
\|V-W\|^2_{L^2(I_n)} &=  \ds\int_{I_n}\left\{\ds\int_{t_{n-1}}^{t}
\left(\ds\int_{t_{n-1}}^{\eta} (V-W)''(s)ds\right)d\eta\right\}^2dt
\\
&\leq
\ds\frac{1}{2}\ds\int_{I_n}(t-t_{n-1})^2\left\{\ds\int_{t_{n-1}}^{t}
\left(\ds\int_{t_{n-1}}^{\eta}|V''-W''|^2ds\right)d\eta\right\}dt\\
&\leq
\ds\frac{1}{2}\ds\int_{I_n}(t-t_{n-1})^2\left\{\ds\int_{t_{n-1}}^{t}d\eta
\ds\int_{I_n}|V''-W''|^2ds \right\}dt\\
&= \ds\frac{k_n^4}{8}\|V''-W''\|^2_{L^2(I_n)}.
 \end{aligned}
\end{eqnarray}

It is well known that,  for any function  $v\in H^1(a,b)$ and
satisfies $v(a)=0$, there holds the  Poincar\'{e}-Friedrichs
inequality \cite{bra}
\begin{equation}\label{PF-ine}
\|v\|_{L^2(a,b)} \le  (b-a) \|v'\|_{L^2(a,b)}.
\end{equation}
Noting the fact $(V-W)'(t_{n-1})=0$ and using (\ref{PF-ine}), we
obtain
\begin{equation}\label{uniq-16}
  \|V'-W'\|_{L^2(I_n)}\leq k_n\|V''-W''\|_{L^2(I_n)}.
\end{equation}
Inserting (\ref{uniq-11}) into (\ref{uniq-15}) and (\ref{uniq-16}),
we get
\begin{equation}\label{uniq-18}
\|V-W\|^2_{L^2(I_n)}\leq
\ds\frac{L^2k_n^4}{4}\|\widetilde{V}-\widetilde{W}\|^2_{H^1(I_n)}
\end{equation}
and
\begin{equation}\label{uniq-17}
\|V'-W'\|^2_{L^2(I_n)}\leq
2L^2k^2_n\|\widetilde{V}-\widetilde{W}\|^2_{H^1(I_n)}.
\end{equation}
Moreover, combining  (\ref{uniq-18}) and (\ref{uniq-17}),  yields
\begin{equation*}
\|V-W\|^2_{H^1(I_n)}\leq
\ds\frac{L^2k_n^2}{4}(8+k^2_n)\|\widetilde{V}-\widetilde{W}\|^2_{H^1(I_n)},
\end{equation*}
which implies that
\begin{equation*}
\|G \widetilde{V}-G\widetilde{W}\|_{H^1(I_n)}\leq
\ds\frac{Lk_n}{2}\sqrt{8+k^2_n}\|\widetilde{V}-\widetilde{W}\|_{H^1(I_n)}.
\end{equation*}
Therefore,  the operator $G$ is a contraction $P_{r_n}(I_n)$ if the
condition (\ref{mesh-cond}) is satisfied. This completes the proof.
\end{proof}

\subsection{Proof of Lemma  \ref{wellpose}}\label{app-A-2}
\begin{proof}
For $r=2$ and $3$, the existence and uniqueness of
$\Pi_{\Lambda}^{r}u$ can be easily verified by setting
$$\Pi_{\Lambda}^{r}u=u(-1)- \frac{(x+1)(x-3)}{4}
u'(-1)+\frac{(x+1)^2}{4}u'(1)$$ and $\Pi_{\Lambda}^{r}u=H_3u(x)$,
respectively. Here, $H_3u$ is the cubic Hermite interpolation of $u$
defined by (\ref{u-exp-1}).

Selecting  $\varphi=1$ in (\ref{def-proj}) and using the third
condition of (\ref{def-proj}), yields
$u'(1)=(\Pi_{\Lambda}^{r}u)'(1)$. If $r \ge 3$,  selecting
$\varphi=x$ in  (\ref{def-proj}) and performing an integration by
parts, then using the second condition of (\ref{def-proj}), gives
$u(1)=\Pi_{\Lambda}^{r}u(1)$. Therefore, we have $\Pi_{\Lambda}^{r}u
(\pm 1)=u(\pm 1)$  and $(\Pi_{\Lambda}^{r}u)' (\pm 1)=u'(\pm 1)$ for
  $r \ge 3$.

We next show the uniqueness of $\Pi_{\Lambda}^{r}u$ for $r > 3$. Let
$u_1$ and $u_2$ be two polynomials in $P_r(\Lambda)$ that satisfy
(\ref{def-proj}). Then, the difference $u_1-u_2$ can be expanded by
the Legendre series
\begin{equation}\label{I-1}
u_1-u_2=\ds\sum_{i=0}^{r} c_iL_i(x),
\end{equation}
where $c_i=\frac{2i+1}{2}\int_{-1}^1 (u_1-u_2)L_idx$. By
(\ref{def-proj}), there holds
\begin{equation*}
\ds\int_{\Lambda}(u_1-u_2)''\varphi dt=0 ,\quad \forall \varphi \in
P_{r-2}(\Lambda).
\end{equation*}
Then, by  integration by parts and the fact $(u_1-u_2)(\pm
1)=(u_1-u_2)'(\pm 1)=0$, we have
\begin{equation*}
\ds\int_{\Lambda}(u_1-u_2)\varphi'' dt=0 ,\quad \forall \varphi \in
P_{r-2}(\Lambda),
\end{equation*}
Using the orthogonality properties of the Legendre polynomials, we
get  $c_i=0$ for $0\leq i\leq r-4$. Hence, the difference $u_1-u_2$
is given by
\begin{equation*}
u_1-u_2=c_{r-3}L_{r-3}+c_{r-2}L_{r-2}+c_{r-1}L_{r-1}+c_{r}L_{r}.
\end{equation*}
Moreover, using $L_i(\pm 1)=(\pm 1)^i$ and the fact $(u_1-u_2)(\pm
1)=(u_1-u_2)'(\pm 1)=0$ again, we obtain
$c_{r-3}=c_{r-2}=c_{r-1}=c_{r}=0$. Hence,  we have $u_1-u_2\equiv
0$, which proves the uniqueness of a polynomial satisfying the
conditions in Definition \ref{proj} for $r>3$. The existence follows
immediately   by setting
\begin{equation}\label{I-5}
\Pi_{\Lambda}^{r}u(x)=H_3u(x)+\ds\sum\limits_{i=4}^{r}b_i
J^{-2,-2}_i(x),
\end{equation}
which is the truncation of  the expansion of $u$ given in
(\ref{u-exp-1}). By the properties of the generalized Jacobi
polynomial $J^{-2,-2}_i(x)$ listed in Lemma \ref{pro-plo}, it is
easy to verify that the polynomial $\Pi_{\Lambda}^{r}u$  defined in
(\ref{I-5}) satisfies  (\ref{def-proj}). This completes the proof.
\end{proof}

\subsection{proof of Lemma \ref{Pi-pro}} \label{app-A-3}

\begin{proof}
We first define the following weighted $L^2$-norm by
$$\|u^{(s+1)}\|^2_{L^2_{\omega^{s-1,s-1}}(\Lambda)}:=\ds\int_{-1}^{1} |u^{(s+1)}|^2(1-x^2)^{s-1}dx.$$
 Clearly, for any integer $s\ge 1$, there holds
\begin{equation}\label{normeq}
 \|u^{(s+1)}\|^2_{L^2_{\omega^{s-1,s-1}}(\Lambda)} \le
\|u^{(s+1)}\|^2_{L^2(\Lambda)}.
\end{equation}
Moreover, by (\ref{Leg-ort-1}) and (\ref{u-exp}), we  have
\begin{eqnarray} \label{u-exp-2}
\begin{aligned}
 \|u^{(s+1)}\|^2_{L^2_{\omega^{s-1,s-1}}(\Lambda)}=& \ds\int_{-1}^{1} \Big(\ds\sum\limits_{i=s-1}^{\infty}a_i L_i^{(s-1)}(x)\Big)^2
 (1-x^2)^{s-1}dx\\
 =& \ds\sum\limits_{i=s-1}^{\infty}a_i^2 \ds\int_{-1}^{1} (L_i^{(s-1)}(x))^2
 (1-x^2)^{s-1}dx\\
 =& \ds\sum\limits_{i=s-1}^{\infty}\ds\frac{2}{2i+1} \ds\frac{(i+s-1)!}{(i-s+1)!}
 a_i^2,
\end{aligned}
\end{eqnarray}
where $a_i=\frac{2i+1}{2}\int_{-1}^1 u''L_idx$. Then, by
(\ref{u-exp-1}), (\ref{I-5}),  (\ref{J-ort}), (\ref{normeq}),  and
(\ref{u-exp-2}), we obtain
\begin{equation*}\label{I-6}
\begin{aligned}
\|u-\Pi_{\Lambda}^{r}u\|^2_{L^2(\Lambda)}=& \ds\int_{-1}^{1}\Big(
\ds\sum\limits_{i=r+1}^{\infty} b_iJ^{-2,-2}_i(x) \Big)^2dx
\le \ds\int_{-1}^{1}\Big( \ds\sum\limits_{i=r+1}^{\infty} b_iJ^{-2,-2}_i(x) \Big)^2 (1-x^2)^{-2}dx\\
=& \ds\sum\limits_{i=r+1}^{\infty} b_{i}^2 \gamma^{2,2}_{i-4}
=\ds\sum\limits_{i=r+1}^{\infty} \ds\frac{2}{i(i-1)(i-2)(i-3)(2i-3)} a_{i-2}^2\\
=&\ds\sum\limits_{i=r-1}^{\infty}\left( \ds\frac{2}{2i+1}\ds\frac{(i+s-1)!}{(i-s+1)!} a_{i}^2\right) \ds\frac{1}{(i+2)(i+1)i(i-1)} \ds\frac{(i-s+1)!}{(i+s-1)!}\\
\le &\ds\frac{1}{(r+1)r(r-1)(r-2)} \ds\frac{(r-s)!}{(r+s-2)!}
\|u^{(s+1)}\|^2_{L^2(\Lambda)}
\end{aligned}
\end{equation*}
for any integer $s$, $1\leq s \leq \min\{r,s_0\}$. This completes
the proof of (\ref{L2-pro}).

Similarly, by (\ref{u-exp-1}), (\ref{I-5}),  (\ref{J-4}),
(\ref{J-ort}), (\ref{normeq}),
 and (\ref{u-exp-2}), we have
\begin{equation*}\label{I-7}
\begin{aligned}
\|(u-\Pi_{\Lambda}^{r}u)'\|^2_{L^2(\Lambda)}=& \ds\int_{-1}^{1}\Big(
\ds\sum\limits_{i=r+1}^{\infty} b_i \partial_x J^{-2,-2}_i(x)
\Big)^2dx
\le \ds\int_{-1}^{1}\Big( \ds\sum\limits_{i=r+1}^{\infty} b_i \partial_x J^{-2,-2}_i(x) \Big)^2 (1-x^2)^{-1}dx\\
=&\ds\int_{-1}^{1}\Big( \ds\sum\limits_{i=r+1}^{\infty} -2(i-3)b_i J^{-1,-1}_{i-1}(x) \Big)^2 (1-x^2)^{-1}dx\\
=&\ds\sum\limits_{i=r+1}^{\infty} \ds\frac{1}{4(i-2)^2} a_{i-2}^2
\gamma^{1,1}_{i-3}
=\ds\sum\limits_{i=r+1}^{\infty}\ds\frac{2}{(2i-3)(i-1)(i-2)} a_{i-2}^2\\
=&\ds\sum\limits_{i=r-1}^{\infty} \left( \ds\frac{2}{2i+1}\ds\frac{(i+s-1)!}{(i-s+1)!} a_{i}^2\right)  \ds\frac{1}{(i+1)i} \ds\frac{(i-s+1)!}{(i+s-1)!}\\
\le &\ds\frac{1}{r(r-1)} \ds\frac{(r-s)!}{(r+s-2)!}
\|u^{(s+1)}\|^2_{L^2(\Lambda)}
\end{aligned}
\end{equation*}
for any integer $s$, $1\leq s \leq \min\{r,s_0\}$. This completes
the proof of (\ref{H1-pro}).

Finally, by (\ref{u-exp-1}), (\ref{I-5}),   (\ref{Leg-ort}),
(\ref{J-4}), (\ref{normeq}), and (\ref{u-exp-2}), we get
\begin{equation*}\label{I-7}
\begin{aligned}
\|(u-\Pi_{\Lambda}^{r}u)''\|^2_{L^2(\Lambda)}=&
\ds\int_{-1}^{1}\Big( \ds\sum\limits_{i=r+1}^{\infty} b_i
\partial^2_x J^{-2,-2}_i(x) \Big)^2dx
=\ds\int_{-1}^{1}\Big(\ds\sum\limits_{i=r+1}^{\infty}  a_{i-2} L_{i-2}(x)\Big)^2 dx\\
=&\ds\sum\limits_{i=r-1}^{\infty}\ds\frac{2}{2i+1} a_{i}^2 =
\ds\sum\limits_{i=r-1}^{\infty} \left( \ds\frac{2}{2i+1}\ds\frac{(i+s-1)!}{(i-s+1)!} a_{i}^2\right) \ds\frac{(i-s+1)!}{(i+s-1)!}\\
\le &\ds \ds\frac{(r-s)!}{(r+s-2)!} \|u^{(s+1)}\|^2_{L^2(\Lambda)}
\end{aligned}
\end{equation*}
for any integer $s$, $1\leq s \leq \min\{r,s_0\}$. This completes
the proof of (\ref{H2-pro}).
\end{proof}

\end{document}